\documentclass[a4paper,12pt,twoside]{amsart}
\usepackage{amsmath}
\usepackage{amsthm}
\usepackage{amssymb}
\theoremstyle{plain}
\newtheorem{thm}{\indent\bf Theorem}[section]
\newtheorem{lem}[thm]{\indent\bf Lemma}
\newtheorem{prop}[thm]{\indent\bf Proposition}
\newtheorem{cor}[thm]{\indent\bf Corollary}
\theoremstyle{definition}
\newtheorem{rem}{\indent\it Remark}[section]

\numberwithin{equation}{section}
\numberwithin{figure}{section}
\def \re {\mathrm{Re\,}}
\def \im {\mathrm{Im\,}}
\def \diag {\mathrm{diag}}

\def \z {\mathfrak{z}}
\def \l {\mathbf{c}}
\def \sn {\mathrm{\check{s}n}}
\begin{document}
\title[Elliptic asymptotic representation]{Elliptic asymptotic 
representation of the fifth Painlev\'e transcendents
 {\rm (corrected version)}}
\author[Shun Shimomura]{Shun Shimomura} 
%%%%%%%%%%%%%%%%%%%%%%%%%%%%%%%%%%%%%%%%%%%%%
%%%%%%%%%%%%%%%%%%%%%%%%%%%%%%%%%%%%%%%%%%%%%%
\address{Department of Mathematics, 
Keio University, 
3-14-1, Hiyoshi, Kohoku-ku,
Yokohama 223-8522 
Japan \quad
{\tt shimomur@math.keio.ac.jp}}
\date{}
%%%%%%%%%%%%%%%%%%%%%%%%%%%%%%%%%%%%%%%
% \subjclass[2010]{34M55, 34M56, 34M35}
% \keywords{elliptic asymptotic representation; fifth Painlev\'e transcendents;
% isomonodromy deformation; monodromy data; Jacobi elliptic functions}
%%%%%%%%%%%%%%%%%%%%%%%%%%%%%%%%%%%%%%%
\maketitle
\begin{abstract}
For the fifth Painlev\'e transcendents an asymptotic representation by the 
Jacobi $\mathrm{sn}$-function is presented in cheese-like strips along generic 
directions near the point at infinity. Its elliptic main part may be understood
to depend on the phase shift as a single integration constant, which  
is parametrised by monodromy data for the associated isomonodromy deformation.
The other integration constant is contained in the error term or in a correction
function. This paper contains corrections of the Stokes graph and of the related
results in the early version. 
\end{abstract}
%\ams{34M55}
\allowdisplaybreaks
%%%%%%%%%%%%%%%%%%%%%%%%%%%%%%%%%%%%%
%%%%%% section 1 %%%%%%%
\section{Introduction}\label{sc1}
%%%%%%%%%%%%%%%%%%%%%%%%%%%%%%%
%%%%%%%%%%%%%%%%%%%%%%%%%%%%%%%%%%
The fifth Painlev\'e equation 
%%%%%%%%%%%%%%%%%%%%%%%%%%%%%%%%
\begin{align*}
\tag*{(P$_\mathrm{V}$)}
&\frac{d^2y}{d x^2}=  \Bigl(\frac 1{2y} + \frac 1{y-1}
 \Bigr) \Bigl(\frac{d y}{d x} \Bigr)^{\!\! 2}
- \frac 1x \frac{d y}{d x}
\\[0.2cm]
&  +\frac{(y-1)^2}{8x^2} \Bigl((\theta_0-\theta_1+\theta_{\infty}
)^2 y - \frac{(\theta_0 -\theta_1 - \theta_{\infty})^2 }
{y}\Bigr) + (1-\theta_0-\theta_1) \frac{y}{x} -\frac{y(y+1)}{2(y-1)}
\end{align*}
%%%%%%%%%%%%%%%%%%%%%%%%
with $\theta_0,$ $\theta_1,$ $\theta_{\infty} \in \mathbb{C}$ defines nonlinear
special functions, which are meromorphic on the universal covering space of
$\mathbb{C} \setminus \{0\}$. 
A general solution is expressed by a convergent series in a spiral domain 
around $x=0$, and admits asymptotic representations as $x\to \infty$ along the
real and the imaginary axes (cf. e.g. \cite{S3}, \cite{S1}).  
Using the isomonodromy property and the WKB analysis,
for a generic case of (P$_{\mathrm{V}}$)
Andreev and Kitaev \cite{Andreev-Kitaev} obtained families of solutions
near $x=0$ and $x=\infty$ on the positive real axis, and connection formulas
for these solutions. Along the imaginary axis asymptotic solutions and 
the associated monodromy data have been studied by 
\cite{Andreev-Kitaev-2019}, \cite{S-2018}.   
Furthermore Lisovyy et al. \cite{L} gave a connection formula 
for the tau-function $\tau_{\mathrm{V}}(x)$ between $x=0$ and $x=i\infty$ 
and the ratios of multipliers of $\tau_{\mathrm{V}}(x)$ as $x\to 0, +\infty, 
i\infty$.
\par
Near $x=\infty$ along directions other than the real or imaginary axis,
a general solution of (P$_{\mathrm{V}}$) behaves quite differently.
In generic directions it is known that, for solutions of the Painlev\'e 
equations (P$_{\mathrm{I}}$), $\ldots$, (P$_{\mathrm{IV}}$) except  
truncated or classical ones, the Boutroux ansatz holds
\cite{Boutroux}, \cite{Boutroux-2}. Elliptic asymptotic representations 
have been studied for (P$_{\mathrm{I}}$), (P$_{\mathrm{II}}$)
by Joshi and Kruskal \cite{Joshi-Kruskal-1}, \cite{Joshi-Kruskal-2},
Its and Kapaev \cite{Its-Kapaev},
Kapaev \cite{Kapaev-1}, \cite{Kapaev-2}, Kapaev and Kitaev \cite{Ka-Ki},
Kitaev \cite{Kitaev-2}, 
\cite{Kitaev-3} and Novokshenov \cite{Novokshenov-1}, \cite{Novokshenov-2};
for (P$_{\mathrm{III}}$) by Novokshenov \cite{Novokshenov-3}, \cite
{Novokshenov-4}; and
for (P$_{\mathrm{IV}}$) by Kapaev \cite{Kapaev-3}, Vereshchagin \cite{Vere}.
The elliptic representations for (P$_{\mathrm{II}}$) and (P$_{\mathrm{III}}$), 
nonlinear Stokes phenomena and connection problems are also in the monograph 
\cite{FIKN} (see also \cite{IN}). 
Concerning the elliptic representation
for solutions of (P$_{\mathrm{I}}$) Iwaki's recent work \cite{Iwaki} by the 
topological recursion is remarkable. 
%%%%%%%%%%%%%%%%%%%%%%%%%%%
\par
In this paper we show the Boutroux ansatz for the fifth Painlev\'e equation
(P$_{\mathrm{V}}$), that is, present an elliptic asymptotic representation
for a general solution of (P$_{\mathrm{V}}$), which is given by the Jacobi
$\mathrm{sn}$-function, along generic directions near the point at infinity 
(Theorem \ref{thm2.1}).
In deriving our results we employ the isomonodromy
property described as follows:
for the complex parameter $x$ the isomonodromy deformation of the 
two-dimensional linear system 
%%%%% (1.1) %%%%%%%%%%%%%%%%%%
\begin{align} \label{1.1}
& \frac{d\Xi}{d\xi} = \Bigl(\frac x2 \sigma_3 + \frac{\mathcal{A}_0}{\xi}
 + \frac{\mathcal{A}_1}{\xi-1} \Bigr) \Xi,
\qquad  
\sigma_3 = \begin{pmatrix}  1  &  0  \\  0  &  -1
\end{pmatrix},
\\ \notag
& \mathcal{A}_{0} =
\begin{pmatrix}
 \z+ \theta_0/2  & -u(\z+ \theta_0) \\
\z/u   &            -\z-\theta_0/2 \\
\end{pmatrix},
\\ \notag
& \mathcal{A}_1=
\begin{pmatrix}
-\z-(\theta_0+\theta_{\infty})/2  & uy(\z+(\theta_0-\theta_1
 +\theta_{\infty})/2) \\
-(uy)^{-1}(\z+(\theta_0+\theta_1 +\theta_{\infty})/2) & 
\z+(\theta_0+\theta_{\infty})/2  \\
\end{pmatrix}
\end{align}
%%%%%%%%%%%%%%%%%%%%%%%%%%%%%%%%%%%%
with $(y, \z, u)=(y(x), \z(x), u(x))$ is governed by the system of equations
%%%%%%%%%%%%%%%%%%%%%%%%%%%%%%%%%%%
\begin{align*}
& x\frac{dy}{dx} =xy -2\z (y-1)^2 -(y-1) \Bigl(\frac 12 (\theta_0-\theta_1
+\theta_{\infty})y -\frac 12 (3\theta_0 +\theta_1 +\theta_{\infty})\Bigr),
\\
& x\frac{d\z}{dx} =y\z  \Bigl(\z+ \frac 12 (\theta_0-\theta_1
+\theta_{\infty})\Bigr)  - \frac{(\z+\theta_0)}{y} \Bigl(\z+ \frac 12 
(\theta_0 +\theta_1 +\theta_{\infty})\Bigr),
\\
& x\frac{d}{dx}\log u = -2\z -\theta_0 + y \Bigl(\z+ \frac 12 (\theta_0-\theta_1
+\theta_{\infty})\Bigr)  + \frac{1}{y} \Bigl(\z+ \frac 12 
(\theta_0 +\theta_1 +\theta_{\infty})\Bigr),
\end{align*}
which is equivalent to (P$_{\mathrm{V}}$) \cite[Appendix C]{JM}, \cite{Jimbo}, 
\cite{Andreev-Kitaev}; that is, 
$y(x)$ solves (P$_{\mathrm{V}}$) if and only if the monodromy data $M^0,$
$M^1,$ $S_1,$ $S_2$ for \eqref{1.1} (defined in Section \ref{sc2}) remain 
invariant under a small change of $x.$  
We apply the WKB analysis to calculate the monodromy data, 
which is used in deriving the elliptic expression of the main results. 
The main term expressed by the $\mathrm{sn}$-function
%% with the modulus $A_{\phi}^{1/2}$ 
may be understood to depend on the phase shift only 
as a single integration constant parametrised by the monodromy data.  
The other integration constant 
is hidden in the error term, and is also deeply related to the correction 
function $B_{\phi}(t)=t(a_{\phi}(t)-A_{\phi})$ with $\phi=\arg x.$  
%% Solving equations \eqref{6.5}, \eqref{6.6} on the Lagrangian, 
%% we may derive $a_{\phi}$, and find an explicit asymptotic expression of 
%% the error term under a certain supposition. 
%%% The error term is expressed by an explicit asymptotic formula, 
%%% whose 
%%% leading term is written in terms of integrals of the $\mathrm{sn}$-function 
%%% and the $\vartheta$-function, and contains the other integration constant  
The error term admits an explicit asymptotic formula 
containing the other integration constant \cite[Theorems 2.2 and 2.3]{S-2024}.
%% \par
%% In the early version of \cite{S} the Stokes graph is incorrect, and
%% it affects the phase shifts of asymptotic solutions in \cite[Theorems 2.1 and
%% 2.2]{S}. The present version contains the amended Stokes graph and 
%% the subsequent modifications including the corrected phase shifts in 
%% these theorems (Corrigendum of \cite{S}).
\par
This paper is organised as follows.
Section \ref{sc2} describes our main results on the expression of a general
solution by the 
Jacobi $\mathrm{sn}$-function for $0<|\phi|<\pi/2$  
(Theorem \ref{thm2.1}) and for $0<|\phi - \pi|<\pi/2$ (Theorem \ref{thm2.2}).
%%% and on the error term with the explicit leading term 
%%% (Theorem \ref{thm2.1a} from \cite{S-2024}).  
%%% Theorems \ref{thm2.1} and \ref{thm2.2} contain corrected phase shifts.
Section \ref{sc3} provides basic facts necessary in proving the main 
results: parametrisation of $y(x)$ by the monodromy data; 
turning points and Stokes curves for the symmetric linear system \eqref{3.6}
with $\lambda=e^{i\phi}(2\xi-1)$; and a WKB-solution in the canonical domain
and local solutions around turning points. 
Section \ref{sc4} is devoted to a direct monodromy problem for system 
\eqref{3.6}. The monodromy matrices are obtained from connection matrices
along suitably chosen paths, which are calculated by WKB 
analysis under the supposition \eqref{4.1}. 
%% Our discussion follows
%% the justification scheme of Kitaev \cite{Kitaev-1}, and the WKB analysis is 
%% carried out under the supposition \eqref{4.1}. 
The monodromy matrices are expressed in term of integrals related
to the characteristic roots and the WKB-solutions. In Section \ref{sc5}, 
these integrals are represented by elliptic integrals and the 
$\vartheta$-function (Propositions \ref{prop5.3}, \ref{prop5.5} and 
Corollary \ref{cor5.4}).
In Section \ref{sc6}, considering an inverse monodromy problem by the use of the 
propositions in Section \ref{sc5} and following the justification scheme 
by Kitaev \cite{Kitaev-1}, we arrive at the elliptic asymptotic representation 
of $y(x)$ as in Theorem \ref{thm2.1} and an asymptotic expression of 
the correction function $B_{\phi}(t)$. For Theorem \ref{thm2.2}  
the proof is sketched.
Section \ref{sc7} gives an alternative approach to 
$B_{\phi}(t)$, by using a system equivalent 
to (P$_{\mathrm{V}}$) via its Lagrangian. This system is employed 
in the study of the error term \cite{S-2024}.
Several properties of the Boutroux equations and its solution
are important in our arguments, which are discussed in the final section. 
\par
In the early version of \cite{S} the Stokes graph is incorrect, and
it affects the phase shifts of asymptotic solutions in \cite[Theorems 2.1 and
2.2]{S}. The amended Stokes graph and 
the subsequent modifications including the corrected phase shifts  
are contained in the second version (Corrigendum of \cite{S}).
In the preceding version, the further correction of the phase shift in 
Theorem \ref{thm2.1} ($-\Omega_{\mathbf{a}}/2$ is added), and
the following changes were made: 
\par
(i) Theorem \ref{thm2.2} is renewed.  
Theorems \ref{thm2.1a} and \ref{thm2.4} are replaced 
with unconditional results cited from \cite{S-2024}, and 
the early proofs are removed from Section \ref{sc6}.   
\par
(ii) In the proof of Proposition \ref{prop3.6} a minor revision is made. 
Estimates in Proposition \ref{prop3.7} are improved. 
Remarks \ref{rem3.100} and \ref{rem4.1} are added, which are related to 
Proposition \ref{prop3.7}.
\par
(iii) Subsection \ref{ssc4.2} is revised and Figure \ref{comploops} is added.
\par
(iv) The contents of Section \ref{sc7} of the early version is replaced 
with a description on $B_{\phi}(t)$, which was contained in Section \ref{sc6}.
\par\noindent
This version contains several minor revisions for readability, e.g. 
Theorem \ref{thm2.1}.
\par
Throughout this paper we use the following symbols:
\par
(1) $\sigma_1,$ $\sigma_2$, $\sigma_3$ are the Pauli matrices
$$
\sigma_1=\begin{pmatrix}
0 & 1 \\ 1 & 0
\end{pmatrix},
\quad
\sigma_2=\begin{pmatrix}
0 & -i \\ i & 0
\end{pmatrix},
\quad
\sigma_3=\begin{pmatrix}
1 & 0 \\ 0 &-1 
\end{pmatrix};
$$
\par
(2) for complex-valued functions $f$ and $g$, we write $f\ll g$ or $g\gg f$
if $f=O(|g|)$, and write $f \asymp  g$ if $g \ll f \ll g$.
%% \vskip0.2cm
%%%%%%%%%%%%%%%%%%%%%%%%%%%%%%%%%%%%%
%%%%%% Section 2 %%%%%%%
\section{Main results}\label{sc2}
%%%%%%%%%%%%%%%%%%%%%%%%%%%%%%%%%%%%
To state our main theorems, we make necessary preparations.
Let system \eqref{1.1} admit the isomonodromy property with respect to 
a fundamental matrix solution of the form
%%%%%%%%%%% (1.2) %%%%%%%%%%%%%
\begin{equation}\label{1.2}
\Xi(\xi) =\Xi(x, \xi) = (I+O(\xi^{-1})) \exp\Big(\frac 12 (x\xi -\theta_{\infty}
\log\xi)\sigma_3 \Bigr)
\end{equation}
%%%%%%%%%%%%%%%%%%%
as $\xi\to\infty$ through the sector $|\arg (x\xi) -\pi/2|<\pi$. 
Let $M^0, M^1, M^{\infty} \in SL_2(\mathbb{C})$ be the monodromy matrices 
defined by $\Xi^{(l_{\nu})}(\xi)=\Xi(\xi)M^{\nu}$ for $\nu=0,1,\infty.$
Here $\Xi^{(l_{\nu})}(\xi)$, $\nu=0,1,\infty$ denote the 
analytic continuations of $\Xi(\xi)$ along the respective loops 
$l_0,$ $l_1$, $l_{\infty}
\in \pi_1(P^1(\mathbb{C})\setminus\{0,1,\infty\})$ as in Figure \ref{Loops}
defined for $-\pi/2 <\arg x<\pi/2$, 
which start from the point 
$p_{\mathrm{st}}$ satisfying $100<|p_{\mathrm{st}}|<\infty$ and 
$\arg(xp_{\mathrm{st}})=\pi/2$ and surround, respectively, 
$\xi=0$, $\xi=1$ and $\xi=\infty$ anticlockwise. 
It is easy to see $M^{\infty}M^1 M^0 = I.$ 
%%%%%%%%%%%%%%%%%%%%%%%%%%%%%%%%%%%%%%%%%%%%%%%%%%%%%
%%%%%%%%%%%%%%%%%%%%%%%%%%%%%%%%%%%%%%%%%%%%%%%%%%%%%
%%%%%%%%%%%%%%% Figure 2.1 %%%%%%%%%%%%%%%%%%%%%%%%%%
%%%%%%%%%%%%%%%%%%%%%%%%%%%%%%%%%%%%%%%%%%
{\small
\begin{figure}[htb]
\begin{center}
\unitlength=0.63mm
\begin{picture}(80,53)(-10,-5)
\thicklines
\put(0,40){\circle*{1}}
\put(0,40){\line(0,-1){35}}
\put(0,0){\circle{10}}
\put(0,0){\circle*{1}}
\put(0,40){\line(1,-1){28.3}}
\put(31.8, 8.2){\circle{10}}
\put(31.8, 8.2){\circle*{1}}
\put(0,40){\line(1,0){70}}
\put(75,40){\circle{10}}
\put(75,40){\circle*{1}}
%%%%%%%%%%%%%%%%%%%%%%%%%%%%
\thinlines
%% \put(-2,25){\vector(0,-1){8}}
 \put(-2,17){\vector(1,-3){0}}
\qbezier(-2,25)(-2,17)(-2,17)
%% \put(2,17){\vector(0,1){8}}
 \put(2,25){\vector(-1,3){0}}
\qbezier(2,17)(2,25)(2,25)
\qbezier(-2,-7)(0,-8)(2,-7)
\put(2,-7){\vector(3,2){0}}
%% \put(12,25){\vector(1,-1){5}}
 \put(17,20){\vector(2,-1){0}}
\qbezier (12,25) (17,20) (17,20)
%% \put(21,22){\vector(-1,1){5}}
 \put(16,27){\vector(-2,1){0}}
\qbezier(21,22)(16,27)(16,27)
\qbezier(34.7, 1.9)(36.8, 2.6)(37.6, 4.7)
\put(37.6, 4.7){\vector(1,4){0}}
%% \put(35,38){\vector(1,0){8}}
 \put(43,38){\vector(3,1){0}}
\qbezier (35,38)(43,38)(43,38)
%%\put(43,42){\vector(-1,0){8}}
\put(35,42){\vector(-3,-1){0}}
\qbezier (43,42)(35,42)(35,42)
\qbezier(82, 38)(83, 40)(82, 42)
\put(82, 42){\vector(-2,3){0}}
%%%%%%%%%%%%%%%%%%
\put(-3,43){\makebox{$p_{\mathrm{st}}$}}
\put(70,28){\makebox{$\xi=\infty$}}
\put(43,4){\makebox{$\xi=1$}}
\put(9,-6){\makebox{$\xi=0$}}
\put(56,43){\makebox{$l_{\infty}$}}
\put(25,18){\makebox{$l_{1}$}}
\put(3,9){\makebox{$l_{0}$}}
\end{picture}
\end{center}
\caption{Loops $l_0$, $l_1$, $l_{\infty}$ for $-\pi/2 <\arg x <\pi/2$}
\label{Loops}
\end{figure}
}
%%%%%%%%%%%%%%%%%%%%%%%%%%%%%%%%%%%%
%%%%%%%%%%%%%%%%%%%%%%%%%%%%%%%
To define Stokes matrices, for each $k\in \mathbb{Z}$,
denote by $\Xi_k(\xi)$ the matrix solution of system \eqref{1.1} admitting 
the same asymptotic representation as \eqref{1.2} as $\xi\to \infty$ through 
the sector 
$|\arg (x\xi) -\pi/2 -(k-2)\pi|<\pi$. Then the Stokes matrices $S_k$ are defined 
by $\Xi_{k+1}(\xi)=\Xi_k(\xi)S_k$, in particular,
for $\Xi(\xi)=\Xi_2(\xi)$ 
%%%%%%%% (1.3) %%%%%%%%%%%%
\begin{equation}\label{1.3}
\Xi(\xi)=\Xi_1(\xi)S_1, \quad \Xi_3(\xi)=\Xi(\xi) S_2, \quad
S_1=\begin{pmatrix}
1  &  0 \\  s_1  & 1 
\end{pmatrix},  \quad
S_2=\begin{pmatrix}
1  &  s_2 \\  0  & 1 
\end{pmatrix}. 
\end{equation}
%%%%%%%%%%%%%
%% These monodromy data $M^0,$ $M^1,$ $S_1,$ $S_2$ are the same as in 
The monodromy data $M^0=(m^0_{ij})$ and $M^1=(m^1_{ij})$ satisfy 
\eqref{3.1} \cite{Andreev-Kitaev}. 
Set
\begin{align*}
\mathcal{M}_{(\theta_0,\theta_1,\theta_{\infty})}:=\{(M^0,M^1)
\in SL_2(\mathbb{C})^2\,|\,&
 m^1_{11}m^0_{11}+m^1_{12}m^0_{21}= e^{-\pi i\theta_{\infty}},\,
\\
&\mathrm{tr}\,M^0=2\cos \pi\theta_0,\, \mathrm{tr}\,M^1=2\cos\pi\theta_1 \}
/\sim,
\end{align*}
where $(M^0,M^1)\sim (\tilde{M}^0,\tilde{M}^1)$ if 
$(M^0,M^1)=c^{\sigma_3}(\tilde{M}^0,\tilde{M}^1)c^{-\sigma_3}$ for some 
$c\in\mathbb{C}\setminus\{0\}.$ Then $\mathrm{dim}_{\mathbb{C}}\mathcal{M}
_{(\theta_0,\theta_1,\theta_{\infty})}=2.$ Let us call 
$\mathcal{M}_{(\theta_0,\theta_1,\theta_{\infty})}$
the {\it monodromy manifold}.
As shown in Proposition \ref{prop3.1}, for every  
$(M^0,M^1) \in \mathcal{M}_{(\theta_0,\theta_1,\theta_{\infty})}$ 
such that $M^0,$ $M^1 \not= \pm I$, there exists  
uniquely a solution $y(x)$ of (P$_{\mathrm{V}}$) corresponding to $(M^0,M^1)$,  
and then $y(x)$ is labelled with $(M^0,M^1)$.
%% Thus solutions of (P$_{\mathrm{V}}$) is 
%% parametrised by equivalence classes of monodromy data $(M^0,M^1)$. 
%% The solution corresponding to $(M^0,M^1)$ is denoted by 
%% $y(x,M^0,M^1)$, and is called a solution {\it labelled
%% with} $(M^0,M^1)$.
\par
For $w(A, z) :=\sqrt{(1-z^2)(A-z^2) }$ 
consider the Riemann surface $\Pi_{A} =\Pi_+ \cup \Pi_-$ such that  
$\Pi_+$ and $\Pi_-$ are glued along the cuts $[-1,-A^{1/2}]$ and $[A^{1/2}, 1]$
with $\re A^{1/2}\ge 0$.   
Let the branches be such that $z^{-2} \sqrt{(1-z^2)(A-z^2)} \to -1$ and 
$ \sqrt{(A-z^2)/(1-z^2)} \to 1$ as $z\to \infty$ on the upper 
sheet $\Pi_+.$\footnote[1]
%%%%%%%%%%%% [1] %%%%%%%%%%%%%%%%%%% 
{Then $A^{-1/2}\sqrt{(1-z^2)(A-z^2)}\to 1$ and $A^{-1/2}\sqrt{(A-z^2)/(1-z^2)}
\to 1$ as $z\to 0$ on $\Pi_+.$}
%%%%%%%%%%%%%%%%%%%%%%%%%%%%%%%%%%%%%%%%%%
Let $\mathbf{a}$ and $\mathbf{b}$ denote basic cycles on $\Pi_{A}$
as drawn in Figure \ref{cycles1} when $A=A_{\phi}.$ 
%%%%%%%%%%%%%%%%%%%%%%%%%%%
For a given number $\phi \in \mathbb{R}$, the Boutroux equations
%%%%%% (2.1) %%%%%%%%%%%%%%%%%
\begin{equation}\label{2.1}
\re e^{i\phi} \int_{\mathbf{a}} \sqrt{\frac{A-z^2}{1-z^2} } dz =
\re e^{i\phi} \int_{\mathbf{b}} \sqrt{\frac{A-z^2}{1-z^2} } dz =0
\end{equation}
%%%%%%%%%%%%%%%%%%%%%%%%%%%%%%
admit a unique solution $A=A_{\phi}\in \mathbb{C}$ having the properties:
\par 
(i) $0 \le \re A_{\phi} \le 1$ for $\phi \in \mathbb{R}$, and $A_0=0,$
$A_{\pm \pi/2}=1;$
\par 
(ii) $A_{-\phi}=\overline{A_{\phi}},$ $A_{\phi\pm \pi}=A_{\phi}$ for $\phi\in
\mathbb{R};$
\par
(iii) for $0 \le \phi \le \pi/2,$ $\im A_{\phi}\ge 0$,
and, for $-\pi/2 \le \phi \le 0,$ $\im A_{\phi}\le 0$
\par\noindent
(cf. Proposition \ref{propA.17} and Figure \ref{trajectories}, (a)). 
By (i) we fix $\re A^{1/2}_{\phi} \ge 0,$ and then 
$(\overline{A_{\phi}})^{1/2} =\overline{A^{1/2}_{\phi}}.$
%% Here $\mathbf{a}$ and $\mathbf{b}$ denote basic cycles as in 
%% Figure \ref{cycles1} on the elliptic curve 
%% $\Pi_{A_{\phi}} =\Pi_+ \cup \Pi_-$  
%% given by $w(A_{\phi}, z) =\sqrt{(1-z^2)(A_{\phi}-z^2) }$ such that  
%% $\Pi_+$ and $\Pi_-$ are glued along the cuts 
%% $[-1,-A_{\phi}^{1/2}]$ and $[A_{\phi}^{1/2}, 1]$. The branches of 
%% these algebraic functions are fixed in such a way that $z^{-2} 
%% \sqrt{(A_{\phi}-z^2)(1-z^2)} \to -1$ and $\sqrt{(A_{\phi}-z^2)/(1-z^2)}\to 1$ 
%% as $z\to \infty$ on the upper sheet $\Pi_+.$ 
In what follows let $\mathbf{a}$ and $\mathbf{b}$ be on $\Pi_{A_{\phi}}
=\Pi_+ \cup\Pi_-$ as in Figure \ref{cycles1}. Write
$$
\Omega_{\mathbf{a,\, b}}= \Omega_{\mathbf{a,\, b}}(\phi) 
=\int_{\mathbf{a, \,b}} \frac{dz}{w(A_{\phi}, z)},\quad
%% \Omega_{\mathbf{b}} =\int_{\mathbf{b}} \frac{dz}{w(A_{\phi}, z)}, \quad
\mathcal{E}_{\mathbf{a,\,b}}= \mathcal{E}_{\mathbf{a,\,b}}(\phi) 
=\int_{\mathbf{a, \, b}} \sqrt{\frac{A_{\phi}-z^2}{1-z^2}}dz, \,\,\,\,
%\mathcal{E}_{\mathbf{b}}=\int_{\mathbf{b}} \sqrt{\frac{A_{\phi}-z^2}{1-z^2}}dz, 
$$
and let $\mathrm{sn}(u; k)$ denote the Jacobi $\mathrm{sn}$-function with modulus
$k$.
%%%%%%%%%%%%%%%%%%%%%%%%%%%%%%%%%%%%%%%%%%%%%%%%%%%%%%%%
%%%%%%%%%%%%%%%%%%%%%%%%%%%%%%%%%%%%%%%%%%%%%%%%%%%%%%
%%%%%%%%%%%%%%%%% Figure 2.2 %%%%%%%%%%%%%%%%%
%%%%%%%%%%%%%%%%%%%%%%%%%%%%%%%%%%%%%%%%%%
{\small
\begin{figure}[htb]
\begin{center}
\unitlength=0.83mm
%%%%%%%%%%%%%%%%%%%%%%%%%%%%%%%%%%
%%%%%%%%%%%%%%%%%%%%%%%%%%%%%%%%%%
%%%%%%%%%%%%%%%%%%
%\begin{picture}(80,50)(-10,0)
\begin{picture}(80,43)(0,3)
\put(0,17){\makebox{$-1$}}
\put(22,3){\makebox{$-A_{\phi}^{1/2}$}}
\put(73,25){\makebox{$1$}}
\put(52,39){\makebox{$A_{\phi}^{1/2}$}}
\thinlines
\put(10,25.5){\line(2,-1){17}}
\put(10,24.5){\line(2,-1){17}}
\put(70,25.5){\line(-2,1){17}}
\put(70,24.5){\line(-2,1){17}}
\qbezier(35,33.5) (40,37) (46,38)
\qbezier(9,31.5) (14,32.7) (19,30.5)
\put(46,38){\vector(4,1){0}}
\put(9,31.5){\vector(-4,-1){0}}
\put(31,31){\makebox{$\mathbf{a}$}}
\put(20,29){\makebox{$\mathbf{b}$}}
\put(65,8){\makebox{$\Pi_{+}$}}
\thicklines
\put(10,25){\circle*{1}}
\put(27,16.5){\circle*{1}}
\put(53,33.5){\circle*{1}}
\put(70,25){\circle*{1}}
\qbezier(24,18) (25,25) (40,32.5)
\qbezier(40,32.5) (56,40) (56.5,31.7)
\qbezier[15](40,18) (54,25.5) (56,30)
\qbezier[15](24,16) (26,11) (40,18)
\qbezier(6,27) (10,32.5) (25.9,24.8)
\qbezier(29,23.2) (35.5,18) (34,13)
\qbezier(6,27) (3.5,21) (14,15)
\qbezier (14,15)(30,7)(34,13)
\end{picture}
\end{center}
\caption{Cycles $\mathbf{a},$ $\mathbf{b}$ on $\Pi_{A_{\phi}}$}
\label{cycles1}
\end{figure}
}
%%%%%%%%%%%%%%%%%%%%%%%%%%%%%%%%%%%%%%%%%%
%%%%%%%%%%%%%%%%%%%%%%%%%%%%%%%%%%%%%%%%%%%%%%%
%%%%%%%%%%%%%%%%%%%%%%%%%%%%%%%%%%%%%%
%%%%%%%%%%%%%%%%%%%%%%%%%%%%%%%%%%%%%%
%%%%% ssc2.1 %%%%%%
\subsection{Elliptic representation}\label{ssc2.1}
%%%%%%%%%%%%%%%%%%%%%%%%%%%%%%%%%%%%%%%%%%%%%%%%%%%%%%%%%
%% Let $y(x,M^0,M^1)$ denote a solution of (P$_{\mathrm{V}}$) corresponding 
%% to $(M^0,M^1).$ 
For $0<|\phi|<\pi/2$ we have the following.
%%%%% Theorem 2.1 %%%%%%%
\begin{thm}\label{thm2.1}
Let $M^0=(m_{ij}^0)$ and $M^1=(m_{ij}^1)$ be such that
$m_{11}^0 m_{21}^0m_{12}^1 \not=0$ (respectively, 
$m_{11}^1 m_{21}^0m_{12}^1 \not=0$). Then, for $-\pi/2<\phi <0$ (respectively,
$0<\phi<\pi/2$), 
a solution $y(x)$ labelled with $(M^0,M^1)$ 
admits an elliptic representation of the form  
\begin{equation*}
\frac{y(x) +1}{y(x)-1}
=A^{1/2}_{\phi} \mathrm{sn} ((x-x_0)/2 +\Delta(x); A^{1/2}_{\phi}),
\end{equation*}
with $\Delta(x)= O(x^{-\delta})$
as $x=e^{i\phi} t \to \infty$ through the cheese-like strip
\begin{align*}
& S(\phi, t_{\infty}, \kappa_0, \delta_0)= \{ x=e^{i\phi}t \,| \,\,
\re t > t_{\infty}, \,\, |\im t |<\kappa_0 \} \setminus \bigcup_{\rho\in
\mathcal{P}_0} \{|x-\rho|<\delta_0 \},
\\
&\mathcal{P}_0 = \{ \rho\,| \,\, \mathrm{sn}((\rho-x_0)/2 ; A_{\phi}^{1/2})
=\infty \}= \{ x_0 + \Omega_{\mathbf{a}}  \mathbb{Z}
 +\Omega_{\mathbf{b}}(2\mathbb{Z}+1) \},
\end{align*}
$\delta>0$ being some positive number,
$\kappa_0 >0$ a given number, $\delta_0>0$ a given small number,
and $t_{\infty}=t_{\infty}(\kappa_0,\delta_0)$ a large number depending on 
$(\kappa_0,\delta_0).$ 
Furthermore $x_0$ is such that 
%%% $x_0 \in S(\phi, t_{\infty}, \kappa_0, \delta_0)$ and that
\begin{align*}
x_0 &\equiv \frac {-1}{\pi i} \Bigl(\Omega_{\mathbf{b}} \log(m^0_{21}m^1_{12})
+ \Omega_{\mathbf{a}} \log \mathfrak{m}_{\phi} \Bigr) -\Bigl( \frac {\Omega
_{\mathbf{a}}}2 +\Omega_{\mathbf{b}}\Bigr) (\theta_{\infty} + 1)   
-\frac{\Omega_{\mathbf{a}}}2
\\
 &= \frac {-1}{\pi i} \Bigl(\Omega_{\mathbf{b}} \log(e^{\pi i\theta_{\infty}}
m^0_{21}m^1_{12}) + \Omega_{\mathbf{a}}\log \mathfrak{m}_{\phi,\theta_{\infty}} 
\Bigr) - \Omega_{\mathbf{a}} -\Omega_{\mathbf{b}}
\mod 2\Omega_{\mathbf{a}}\mathbb{Z}+ 2\Omega_{\mathbf{b}}\mathbb{Z},
\end{align*}
where $\mathfrak{m}_{\phi,\theta_{\infty}} = e^{\pi i\theta_{\infty}/2}m^0_{11}$ if 
$-\pi/2 <\phi<0,$ and
$= e^{-\pi i\theta_{\infty}/2} (m^1_{11})^{-1}$ if $0 <\phi<\pi/2,$ and
$\mathfrak{m}_{\phi}=e^{-\pi i\theta_{\infty}/2}\mathfrak{m}
_{\phi,\theta_{\infty}}$.
%%%%%%%%%%%%%%%%%%%%%%%%%%%%%%%%%%%%%%%%%%%%%%%%%%
\end{thm}
%%%%%%%%%%%%%%%%%%%%%%%%%%%%%%%%%%%%%%%%%%%%%%%%%%
%%%%%%%%%%%%%%%%%%%%%%%%%%%%%%%%
%%%%% Remark 2.1 %%%%%%%%%%%%%%
\begin{rem}\label{rem2.1}
In Theorem \ref{thm2.1}
the complementary cases, say $m^1_{11}m^0_{21}m^1_{12}=0$ for $0<\phi<\pi/2$
correspond to truncated or triply-truncated solutions \cite[Corollaries 5.2
and 5.3]{Andreev-Kitaev}, \cite[\S 5]{Andreev-Kitaev-2},
\cite[Theorem 1.1]{Andreev}, \cite[Theorem 2.21]{S-2018} (see also Remark
\ref{rem3.3c}).
In Theorem \ref{thm2.2} below similar truncated solutions exist in the
complementary cases. 
\end{rem}
%%%%%% Remark 2.4 %%%%%%%%%%%%%%%
\begin{rem}\label{rem2.4}
The pair $(\mathfrak{m}_{\phi},m^0_{21}m^1_{12})$ uniquely determines 
$(M^0,M^1)\in \mathcal{M}_{(\theta_0,\theta_1,\theta_{\infty})}$.  
\end{rem}
%%%%%%%%%%%%%%%%%%%%%%%%%%%%%%
%%%% Remark 2.2 %%%%%%%%%%%%%%
\begin{rem}\label{rem2.2}
%% At every excluded point $\rho\in \mathcal{P}_0,$ $y(\rho)=1.$
The point $x=\rho$ is in $\mathcal{P}_0$ if and only if $y(\rho)=1.$
%% Among the singular values $y=0,$ $1$, $\infty$ of (P$_{\mathrm{V}}$), the
%% neighbourhoods of only $1$-points are excluded from the cheese-like domain
%% of the elliptic expression of $y(x)$. 
\end{rem}
%%%%%%%%%%%%%%%%%%%%%%%%%%%%%%%
%%%%%% Remark 2.3 %%%%%%%%%%
\begin{rem}\label{rem2.3}
Our calculation of Section \ref{sc4} leads to $\delta=1/4-\varepsilon$ for any 
$0<\varepsilon<1/4$, whose numerical value is caused by a technical reason.  
In fact we have $\delta=1$ in Theorem \ref{thm2.1a}, \cite[Theorem 2.1]{S-2024}.
\end{rem}
%%%%%%%%%%%%%%%%%%%%%%%%%%%%%
%%%%%%%%%%%%%%%%%%%%%%%%%%%%%%%%%%%%%%%%%%%%%%%%%%%%%%%%%%%%
The solution in Theorem \ref{thm2.1} labelled with $(M^0,M^1)$ 
may be written as $y(x,M^0, M^1)$.
Let us denote by the same symbol $y(x,M^0, M^1)$
the analytic continuation to the sectors $0<|\phi-\pi|<\pi/2$.
Note that $\Omega_{\mathbf{a}}$ and $\Omega_{\mathbf{b}}$ are
determined by $A_{\phi}$, which does not depend on $M^{0}, M^1$, and satisfy
$\Omega_{\mathbf{a},\mathbf{b}}(\phi+\pi)=\Omega_{\mathbf{a},\mathbf{b}}(\phi)$. 
%%%%%%%%%%%%%%%%%%%%%%%%%%%%%%%%%%%%%%%%%
%%%%%%%%%%%%%%%%%%%%%%%%%%%%%%%%%%%%%%%
%%%%%%% Theorem 2.2 %%%%%%%%%%%%%
%%%%%%%%%%%%%%%%%%%%%%%%%%%%%%%%
\begin{thm}\label{thm2.2}
Write $\breve{M}^0 =(\breve{m}^0_{ij})
 =S^{-1}_2 M^0 S_2$ and $\breve{M}^1=(\breve{m}^1_{ij})=S^{-1}_2 M^1 S_2$. 
Suppose that $\breve{m}^0_{22}\breve{m}^0_{12}\breve{m}^1_{21}\not=0$ 
(respectively, $\breve{m}^1_{22}\breve{m}^0_{12}\breve{m}^1_{21}\not=0$). 
Then, for $\pi/2<\phi<\pi$ (respectively, $\pi<\phi<3\pi/2$),
$y(x)=y(x,M^0, M^1)$ admits an elliptic representation of the form  
$$
\frac{y(x)+1}{y(x)-1}=A^{1/2}_{\phi}\mathrm{sn}((x-\breve{x}_0)/2 +\Delta(x);
A^{1/2}_{\phi}),
$$
as $x=e^{i\phi}t \to\infty$ through $S(\phi,t_{\infty},\kappa_0,\delta_0)$,
where the phase shift is given by
\begin{align*}
\breve{x}_0&\equiv 
 \frac {-1}{\pi i} \Bigl(\Omega_{\mathbf{b}} \log(\breve{m}^0_{12}
\breve{m}^1_{21})
+ \Omega_{\mathbf{a}}\log\breve{\mathfrak{m}}_{\phi}\Bigr)-\Bigl(\frac {\Omega
_{\mathbf{a}}}2 +\Omega_{\mathbf{b}}\Bigr) (\theta_{\infty} + 1)   
-\frac{\Omega_{\mathbf{a}}}2
\\
&=\frac{-1}{\pi i}\Bigl(\Omega_{\mathbf{b}}\log(e^{\pi i
\theta_{\infty}}(\breve{m}^0_{12}\breve{m}^1_{21})^{-1})+\Omega_{\mathbf{a}}
\log\breve{\mathfrak{m}}_{\phi,\theta_{\infty}}\Bigr)-\Omega_{\mathbf{a}}
-\Omega_{\mathbf{b}}
\mod 2\Omega_{\mathbf{a}}\mathbb{Z}+2\Omega_{\mathbf{b}}\mathbb{Z}
\end{align*}
with $\breve{\mathfrak{m}}_{\phi,\theta_{\infty}}=e^{\pi i\theta_{\infty}/2}
(\breve{m}^0_{22})^{-1}$ if $\pi/2<\phi<\pi,$ and $=e^{-\pi i\theta_{\infty}/2}
\breve{m}^1_{22}$ if $\pi<\phi<3\pi/2,$ and 
$\breve{\mathfrak{m}}_{\phi}=e^{-\pi i
\theta_{\infty}/2} \breve{\mathfrak{m}}_{\phi,\theta_{\infty}}.$ 
\end{thm}
%%%%%%%%%%%%%%%%%%%%%%%%%%%%%
%%%%%% Remark 2.5 %%%%%%%%%%%%%%%%%%%%%%%
\begin{rem}\label{rem2.5}
Let $0< |\phi - 2p\pi| <\pi/2$ or $0<|\phi - 2p\pi -\pi|<\pi/2$ with 
$p\in \mathbb{Z}\setminus \{0\}$. Set
\begin{align*}
& M^0_p =((m_p^0)_{ij}) =U_p^{-1}M^0 U_p,
& M^1_p =((m_p^1)_{ij}) =U_p^{-1}M^1 U_p,
\\
& \breve{M}^0_p =((\breve{m}_p^0)_{ij}) =\breve{U}_p^{-1}{M}^0 \breve{U}_p,
&\breve{M}^1_p =((\breve{m}_p^1)_{ij}) =\breve{U}_p^{-1}{M}^1 \breve{U}_p,
\end{align*}
where
\begin{equation*}
 U_p= \begin{cases}
 S_2S_3 \cdots S_{2p}S_{2p+1},   \qquad  &\text{if $p>0$},
\\
I, \qquad &\text{if $p=0$},
\\
 S_1^{-1}S_0^{-1} \cdots S_{2p+3}^{-1} S_{2p+2}^{-1}, \qquad &\text{if $p<0$,} 
\end{cases}
\qquad
  \breve{U}_{p} = U_p S_{2p+2} \quad \text{for $p\in \mathbb{Z}$}.
\end{equation*}
Since $S_{k+2}= e^{i\pi\theta_{\infty}\sigma_3} S_k e^{-i\pi \theta_{\infty}
\sigma_3}$ for $k \in\mathbb{Z}$ \cite[(2.5)]{Andreev-Kitaev}, we have
$U_p=(M^1M^0)^{-p}e^{-\pi i\theta_{\infty}p\sigma_3}.$
Then $y(M_0,M_1; x)$ admits an elliptic representation as of Theorem \ref{thm2.1}
or \ref{thm2.2} with the following substitutions in the phase shift 
(see Section \ref{ssc6.5}):
\par
(1) for $0<|\phi - 2p\pi|<\pi/2,$ 
$m^0_{21}m^1_{12} \mapsto (m^0_p)_{21} (m^1_p)_{12},$
$\mathfrak{m}_{\phi} \mapsto \mathfrak{m}^p_{\phi} 
=\mathfrak{m}_{\phi} |_{m^{0,1}_{11} \mapsto (m^{0,1}_p)_{11}}$ in Theorem
\ref{thm2.1}; and
\par
(2) for $0<|\phi - 2p\pi -\pi|<\pi/2,$ 
$(\breve{m}^0_{12}\breve{m}^1_{21})^{-1}
\mapsto((\breve{m}^0_p)_{12}(\breve{m}^1_p)_{21})^{-1},$
$\breve{\mathfrak{m}}_{\phi} \mapsto \breve{\mathfrak{m}}^p_{\phi}
  =\breve{\mathfrak{m}}_{\phi}
 |_{\breve{m}^{0,1}_{22} \mapsto (\breve{m}^{0,1}_p)_{22}}$ in Theorem 
\ref{thm2.2}. 
\end{rem}
%%%%%%%%%%%%%%%%%%%%%%%%%%%%%%%%%%%%%%%%%%%%%
%% The facts in the remark above suggest the nonlinear monodromy and Stokes
%% structure of (P$_{\mathrm{V}}$), which will be discussed elsewhere.
%%%%%%%%%%%%%%%%%%%%%%%%%%%%%%%%%%%%%%%%%%%%%
%%%%%% ssc2.2 %%%%%%%%%%%%%%%%%%%%%
\subsection{Error term $\Delta(x)$}\label{ssc2.2}
%%%%%%%%%%%%%%%%%%%%%%%%%%%%%%%%%%%%%%%%%%%%%%%%%%
The elliptic expression above apparently contains the single integration 
constant $x_0$, and the other one is hidden in the error term $\Delta(x) \ll 
x^{-\delta}.$ Let us define necessary functions and constants to express 
$\Delta(x)$. 
\par
For $y(x)=y(x,M^0,M^1)$ of Theorem $\ref{thm2.1}$, let $b(x)$ be such that 
$$
a_{\phi}=A_{\phi}+ \frac{B_{\phi}(t)}t =A_{\phi}+ \frac{b(x)}x \quad 
\text{with $x=e^{i\phi}t$, \,\,\, i.e. $ b(e^{i\phi}t)=e^{i\phi} B_{\phi}(t) $},
$$
which is related to $y(x)$ via \eqref{3.12}, \eqref{6.7} with $y^*=y'(t)$. Set
\begin{align*}
 \psi_0(x) &= A_{\phi}^{1/2}\mathrm{sn} ((x-x_0)/2; A_{\phi}^{1/2}),
\\
 b_0(x)&= \beta_0 -\frac{2\mathcal{E}_{\mathbf{a}}}{\Omega_{\mathbf{a}}} x
- \frac 8{\Omega_{\mathbf{a}}} \frac{\vartheta'}{\vartheta}\Bigl( \frac 1
{2\Omega_{\mathbf{a}}} (x-x_0), \tau_0 \Bigr), 
\\
 \tau_0&=\frac{\Omega_{\mathbf{b}}}{\Omega_{\mathbf{a}}},
\qquad
\beta_0 = - \frac 8{\Omega_{\mathbf{a}}}( \log(m^0_{21}m^1_{12}) +\pi i
(\theta_{\infty}+1)), 
\end{align*} 
where
$$
\vartheta(z,\tau)=\sum_{n\in \mathbb{Z}} e^{\pi i\tau n^2 + 2\pi i zn},
\qquad \im \tau >0
$$
with $\vartheta'(z,\tau)=(d/dz)\vartheta(z,\tau)$ is the $\vartheta$-function
(cf. Section \ref{ssc5.2}).
Then $\psi_0(x)$ solves $2\psi'_0 =w(A_{\phi},\psi_0) 
=\sqrt{(1-\psi_0^2)(A_{\phi}-\psi_0^2)},$
and $b_0(x)$ fulfills $b_0(x)-b(x)
= b_0(e^{i\phi}t)- e^{i\phi}B_{\phi}(t) \ll t^{-\delta}$
in $S(\phi,t_{\infty}, \kappa_0,\delta_0)$ (Proposition \ref{prop5.7} and 
Corollary \ref{cor6.1}).  
Furthermore as in Section \ref{sc7}, $b'_0(x)=2(\psi_0(x)^2-A_{\phi}) 
+4\psi_0'(x)$. Write
\begin{align*}
&\check{S}(\phi, t_{\infty}, \kappa_0, \delta_0)=S(\phi, t_{\infty},\kappa_0,
\delta_0) \setminus \bigcup_{\rho \in \mathcal{Q} } \{|x-\rho|<\delta_0 \},
\\
& \mathcal{Q}=\{\rho\,| \,\, \mathrm{sn}((\rho-x_0)/2; A_{\phi}^{1/2})=
\pm A^{-1/2}_{\phi}, \,\pm 1 \}.
\end{align*}
For $\sigma=e^{i\phi}t_{\sigma}\in \mathcal{Q}$ let $l(\sigma)$ be the line
defined by $x=e^{i\phi}(\re t_{\sigma}+i\eta)$ with $\eta \ge \im t_{\sigma}$
if $\im t_{\sigma}\ge 0$ (respectively, $\eta <\im t_{\sigma}$ if
$\im t_{\sigma} <0$); and, if necessary, modify $l(\sigma)$ not to touch
other circles $|x-\sigma'|=\delta_0$ with $\sigma'\in\mathcal{P}_0\cup
\mathcal{Q} \setminus \{\sigma\}$. Then $\check{S}_{\mathrm{cut}}
(\phi,t_{\infty},\kappa_0,\delta_0)$ denote $\check{S}(\phi,t_{\infty},\kappa_0,
\delta_0)$ equipped with the cuts along $l(\sigma)$ or its modification
for all $\sigma\in \mathcal{Q}.$
\par
For $\Delta(x)$ and $b(x)-b_0(x)$ we have the following 
\cite[Theorems 2.2 and 2.3]{S-2024}. 
%%%%%%%%%%%%%%%%%%%%%%%%%%%%%%%%%%
%%%%%%%%%%%%%%%%%%
%%%%% Theorem 2.3
%%%%%%%%%%%%%%%%%%%%%%%%%%
\begin{thm}\label{thm2.1a}
The error term $\Delta(x)=h(x)/2$ is explicitly represented by
$$
h(x)=-\frac{2((\theta_0-\theta_1)^2+\theta_{\infty}^2)}{A_{\phi}-1}x^{-1}
-\int^x_{\infty}F_1(\psi_0,b_0)\frac{d\xi}{\xi}-\frac 32\int^x_{\infty}
F_1(\psi_0,b_0)^2\frac{d\xi}{\xi^2}+O(x^{-2}),
$$
with
$$
F_1(\psi_0, b_0)= \frac{4(\theta_0+\theta_1)\psi_0(\xi)-b_0(\xi)}
{2(A_{\phi}-\psi_0(\xi)^2)}.
$$
Here
$$
\int^x_{\infty}F_1(\psi_0,b_0)\frac{d\xi}{\xi} \ll x^{-1}, \quad
\int^x_{\infty}F_1(\psi_0,b_0)^2\frac{d\xi}{\xi^2} \ll x^{-1}
$$ 
as $x\to\infty$ through $\check{S}_{\mathrm{cut}}(\phi,t_{\infty},\kappa_0,
\delta_0).$ Furthermore,
$$
xh(x)=h_0\beta_0^2+h_1(x)\beta_0+h_2(x)+O(x^{-1}),
$$
where $h_0=(1/8)A_{\phi}^{-1}(1-A_{\phi})^{-1},$ $h_1(x)\ll x^{-1}$ 
and $h_2(x)\ll x^{-1}.$
\end{thm} 
%%%%%%%%%%%%%%%%%%%%%%%%%%%%%
%%%%% Theorem 2.4 %%%%%%%%%%%%%
\begin{thm}\label{thm2.4}
We have 
\begin{align*}
b(x)-b_0(x)=& b_0'(x)h(x)
\\
&-4((\theta_0-\theta_1)^2+\theta_{\infty}^2)x^{-1}
-\int^x_{\infty}(A_{\phi}-\psi_0^2)F_1(\psi_0,b_0)^2\frac{d\xi}{\xi^2}
+O(x^{-2}),
\end{align*}
in which $b_0'(x)=4\psi_0'-2(A_{\phi}-\psi_0^2),$ and
$$
\int^x_{\infty}(A_{\phi}-\psi_0^2)F_1(\psi_0,b_0)^2\frac{d\xi}{\xi^2} \ll x^{-1} $$
as $x\to\infty$ through $\check{S}_{\mathrm{cut}}(\phi,t_{\infty},\kappa_0,
\delta_0)$, and $b(x)-b_0(x)\ll x^{-1}$ in $S(\phi,t_{\infty},\kappa_0,
\delta_0).$
\end{thm}
By Theorem \ref{thm2.1a} we also have
$$
\frac{y(x)+1}{y(x)-1}=A^{1/2}_{\phi}\mathrm{sn}((x-x_0)/2;A_{\phi}^{1/2})
+O(x^{-1})
$$
in $S(\phi,t_{\infty},\kappa_0,\delta_0)$ \cite[Theorem 2.1]{S-2024}.
%%%%%%%%%%%%%%%%%%%%%%%%%%%%%%%%%%%%%%%%%%%%%%%%%%%%%%%%%%
%%%%%% Remark 2.3 %%%%%%%
%%%%%%%%%%%%%%%%%%%%%%%%%
%% For (P$_{\mathrm{I}}$) or (P$_{\mathrm{II}}$) the relation between the error 
%% term $h(x)$ and $b(x)-b_0(x)$ in the 
%% $\tau$-function is referred to by Kitaev \cite[p. 121]{Kitaev-3}.
%%%%%%%%%%%%%%%%%%%%%%%%%%%%%%%%%%%%%
%%%%%% Section 3 %%%%%%%
%%%%%%%%%%%%%%%%%%%%%%%%%%%%%%%%
\section{Basic facts}\label{sc3}
%%%%%%%%%%%%%%%%%%%%%%%%%
%%%%%%%%% ssc 3.1 %%%%%%%%%%%%%%
%%%%%%%%%%%%%%%%%%%%%%%%%%
\subsection{Parametrisation of $y(x)$ by the monodromy data}\label{ssc3.1}
%%%%%%%%%%%%%%%%%%%%%%%%%%
Note that
%%%%%%%%%% (3.2) %%%%%%%%%%%%%%
\begin{equation}\label{3.2}
M^1M^0 = S_1^{-1} e^{-\pi i \theta_{\infty}\sigma_3} S_2^{-1}
\end{equation}
\cite[(2.8), (2.13)]{Andreev-Kitaev}.  
For the monodromy matrices $M^0=(m^0_{ij}),$ $M^1=(m^1_{ij})$, let
$\mathcal{M}^*_{(\theta_0,\theta_1,\theta_{\infty})}$ be the algebraic variety 
consisting of $(M^0, M^1) \in SL_2(\mathbb{C})^2$ such that
%%%%%%%%%%%%%%%%%%%%%%%%%%%%%%%%
%%%%% (3.1) %%%%%%%%
\begin{equation}\label{3.1}
\begin{split}
&\mathrm{tr}\, M^0= 2\cos \pi \theta_0, \quad 
 \mathrm{tr}\, M^1= 2\cos \pi \theta_1, 
\\
 & (M^1M^0)_{11}=m^1_{11}m^0_{11}+m^1_{12}m^0_{21}=e^{-\pi i \theta_{\infty}},
\end{split}
\end{equation}
which is called the {\it manifold of monodromy data} in \cite{Andreev-Kitaev}. 
Then $\dim_{\mathbb{C}} \mathcal{M}^*_{(\theta_0,\theta_1,\theta
_{\infty})}=3.$ 
Our monodromy manifold $\mathcal{M}_{(\theta_0,\theta_1,
\theta_{\infty})}$ defined in Section \ref{sc2} is written as
$\mathcal{M}_{(\theta_0,\theta_1,\theta_{\infty})} 
=\mathcal{M}^*_{(\theta_0,\theta_1,\theta_{\infty})}/\sim$. 
For $d_0\in \mathbb{C}\setminus \{0\}$,
the gauge transformation $\Xi = d_0^{\sigma_3} \hat{\Xi} $
changes \eqref{1.1} with $(y,\z, u)$ to an isomonodromy system with 
$(y, \z, d_0^{-2}u)$, 
and the monodromy matrices for the canonical solution $\hat{\Xi}(\xi)$ 
%% of the same asymptotic form as of \eqref{1.2} 
become $d_0^{-\sigma_3} M^0 d_0^{\sigma_3},$ 
$d_0^{-\sigma_3} M^1 d_0^{ \sigma_3}.$ 
By this fact combined with the surjectivity of the Riemann-Hilbert 
correspondence \cite[\S\S 3, 4, 5]{Andreev-Kitaev}, 
\cite{Andreev, Andreev-Kitaev-1, Andreev-Kitaev-2, Bol, Put}, 
and the uniqueness \cite[Propositions 2.1 and 2.2]{Andreev-Kitaev} (see also
Proposition \ref{prop3.a}, \cite[Proposition 5.9 and Theorem 5.5]{FIKN}, 
\cite{FMZ}) we have the following.
%%%%%%%%%%%%%%%%%%%%%%%%%%%%%%%%%%%%%
%%%%%%% Proposition 3.1 %%%%%%%%%%%%%%%%%%%
\begin{prop}\label{prop3.1}
Let $\mathcal{Y}(\mathrm{P}_{\mathrm{V}})$ be the family of solutions of 
$(\mathrm{P}_{\mathrm{V}})$, and let 
$\varphi:\, \mathcal{Y}(\mathrm{P}_{\mathrm{V}}) \to \mathcal{M}_{(\theta_0,
\theta_1,\theta_{\infty})}$ be such that $y\mapsto (M^0,M^1)$ if
$(y,\z,u)$ governs the isomonodromy deformation of \eqref{1.1} with
monodromy data $(M^0,M^1)$.
%% a map such that the image of each solution of $(\mathrm{P}_{\mathrm{V}})$ is 
%% (an equivalence orbit containing) the corresponding monodromy data.
%% For $(M^0, M^1),$ $(\tilde{M}^0, \tilde{M}^1) \in \mathcal{M},$
%% write $(M^0, M^1) \sim (\tilde{M}^0, \tilde{M}^1) $ if there exists $d_0\in
%% \mathbb{C}\setminus \{0\}$ such that
%% $(\tilde{M}^0, \tilde{M}^1) = d_0^{-\sigma_3} (M^0, M^1) d_0^{\sigma_3}$. 
Then we have the canonical bijection
$$
\varphi: \,\,\, %% \mathcal{Y}^*(\mathrm{P}_{\mathrm{V}}) 
 \mathcal{Y}(\mathrm{P}_{\mathrm{V}}) \setminus 
\mathcal{Y}_0(\mathrm{P}_{\mathrm{V}}) 
\to %% \mathcal{M}^*
\mathcal{M}_{(\theta_0,\theta_1,\theta_{\infty})} \setminus \mathcal{M}_0, 
%% \quad
%%  \mathcal{Y}_0(\mathrm{P}_{\mathrm{V}})= \{ y\in \mathcal{Y}(\mathrm{P}
%% _{\mathrm{V}})\,| \,\,\, \varphi(y) \in \mathcal{M}_0 \},
$$
where $\mathcal{M}_0 = \{(M^0, M^1) \in \mathcal{M}_{(\theta_0,\theta_1,\theta
_{\infty})} \,|\,\, M^0 = \pm I \,\, \text{or} \,\, M^1 = \pm I \}$, 
$ \mathcal{Y}_0(\mathrm{P}_{\mathrm{V}})= \varphi^{-1}( \mathcal{M}_0 ).$
\end{prop}
%%%%%%%%%%%%%%%%%%%%%%%%%%%%%%%%%
Thus the solutions in $\mathcal{Y}(\mathrm{P}_{\mathrm{V}})\setminus 
\mathcal{Y}_0(\mathrm{P}_{\mathrm{V}})$ are parametrised
by $(M^0, M^1) \in \mathcal{M}_{(\theta_0,\theta_1,\theta_{\infty})}\setminus
\mathcal{M}_0$, essentially two parameters.
%%%%%%%%%%%%%%%%%%%%%%%%%%%%%%%%%%%%
%% \begin{rem}\label{rem3.0}
%% Instead of $(M_0,M_1)$ as above, i.e., a pair of the monodromy matrices
%% for the solution \eqref{1.2}, we may alternatively define $(M_0, M_1)$ 
%% as a conjugate class for the relation $(M_0,M_1)\sim T^{-1}(M_0,M_1)T$ 
%%with some $T\in SL_2(\mathbb{C})$ and $\mathcal{M}$ as the family of them. 
%% Then the bijection $\varphi^*$ becomes $mathcal{Y}^*(\mathrm{P}_{\mathrm{V}}) 
%% = \mathcal{Y}(\mathrm{P}_{\mathrm{V}}) \setminus
%% \mathcal{Y}_0(\mathrm{P}_{\mathrm{V}}) \to  \mathcal{M}^*
%% = \mathcal{M} \setminus \mathcal{M}_0.$ 
%% \end{rem}
%%%%%%% Remark 3.1 %%%%%%%%%%%%%%%%%%%%%%%%%%%%%
\begin{rem}\label{rem3.1}
If $\theta_0\in \mathbb{Z}$ or $\theta_1 \in \mathbb{Z}$, then there exists
a one-parameter family of classical solutions corresponding to $M^0=I$ or 
$M^1=I$, which is represented by the Whittaker function. In particular, if
$\theta_0=0$ and $\mathfrak{z}=0$, then $\mathcal{A}_0$ in \eqref{1.1} vanishes
and (P$_{\mathrm{V}}$) admits a family of solutions solving
$$
x\frac{dy}{dx}=xy-\frac 12(y-1)(\theta_{\infty}(y-1) -\theta_1(y+1))
$$
\cite[Remark 2.1]{Andreev-Kitaev}.
\end{rem}
%%%%%%%%%%%%%%%%%%%%%%%%%%%%%
%%%%%%% ssc 3.2 %%%%%%
%%%%%%%%%%%%%%%%%%%
%%%%%%%%%%%%%%%%%%%%%%%%%%
\subsection{Symmetric linear system}\label{ssc3.3}
To consider $y(x)\in \mathcal{Y}(\mathrm{P}_{\mathrm{V}})$  %%% \setminus 
%%% \mathcal{Y}_0(\mathrm{P}_{\mathrm{V}})  
along a ray 
$\arg x=\phi$ with $|\phi|<\pi/2,$ and to convert \eqref{1.1} to a symmetric
form, set
%%%%% (3.5) %%%%%%
\begin{equation}\label{3.5}
x= e^{i\phi}t,\quad t>0, \qquad \xi=(e^{-i\phi}\lambda +1)/2.
\end{equation}
%%%%%%%%%%%%%%%%%%%%%% 
Then by the gauge transformation $Y= \exp(-\varpi(t,\phi)\sigma_3)\Xi$ with
$\varpi(t,\phi)=e^{i\phi}t/4+(\theta_{\infty}/2)(i\phi +\log2)$ system
\eqref{1.1} is taken to
%%%%%%% (3.6) %%%%%%%%%
\begin{equation}\label{3.6}
\frac{dY}{d\lambda} =t \mathcal{B}(t,\lambda) Y
\end{equation}
with $\mathcal{B}(t,\lambda)=\hat{u}^{\sigma_3/2}(b_3\sigma_3 +b_2\sigma_2 +b_1\sigma_1)
\hat{u}^{-\sigma_3/2}.$ Here $\hat{u}=u\exp(-2\varpi(t,\phi))$ and
%%%%%%% (3.7) %%%%%%%%%%
\begin{equation}\label{3.7}
\begin{split}
& b_3 = b_3(t,\lambda) =\frac 14 + t^{-1}\Bigl( \frac{a^{11}_0}{\lambda
+e^{i\phi}} + \frac{a^{11}_1}{\lambda-e^{i\phi} } \Bigr),
\\
& b_2 = b_2(t,\lambda) =\frac{i t^{-1}}{2} \Bigl( \frac{a^{12}_0-a^{21}_0}
{\lambda +e^{i\phi}} + \frac{a^{12}_1- a^{21}_1}{\lambda-e^{i\phi} } \Bigr),
\\
& b_1 = b_1(t,\lambda) =\frac{t^{-1}}{2} \Bigl( \frac{a^{12}_0+a^{21}_0}
{\lambda +e^{i\phi}} + \frac{a^{12}_1+ a^{21}_1}{\lambda-e^{i\phi} } \Bigr),
\end{split}
\end{equation}
%%%%%%%%%%%%%%%%
$a^{ij}_0$ and $a^{ij}_1$ being the entries of $\mathcal{A}_0 |_{u=1},$
$\mathcal{A}_1 |_{u=1},$ that is,
\begin{align*}
a^{11}_0=&\z +\frac{\theta_0}2, \quad  &  a^{12}_0&= -(\z+\theta_0),  
\\
a^{21}_0=&\z , \quad  &a^{22}_0&= -a^{11}_0;  
\\
a^{11}_1=& -\z - \frac{\theta_0+\theta_{\infty}}2, \quad &  
a^{12}_1&= y\Bigl(\z + \frac{\theta_0 -\theta_1+\theta_{\infty}}2 \Bigr),  
\\
a^{21}_1=& -\frac 1y\Bigl(\z+\frac{\theta_0 +\theta_1+\theta_{\infty}}2 \Bigr),  
\quad
&a^{22}_1&= -a^{11}_1.
\end{align*}
%%%%%%%%%%%%%%%%%%%%%%%%%%%
%%%%%%%%%%%%%%%%%%%%%%%%%%%%%%%%%
Note that
\begin{align*}
&(\lambda^2-e^{2i\phi})b_3 = \frac 14(\lambda^2-e^{2i\phi}) - 2e^{i\phi}t^{-1}
\z -\frac 12 (\theta_{\infty}\lambda +(2\theta_0+\theta_{\infty})e^{i\phi})
t^{-1},
\\
&y(\lambda^2-e^{2i\phi})(b_1+ib_2)= ((y-1)(\lambda+e^{i\phi})-2e^{i\phi}y)
t^{-1}\z 
%\\
%&\phantom{-----------}
 -\frac 12 (\theta_0+\theta_1+\theta_{\infty})
(\lambda+e^{i\phi})t^{-1},
\\
&(\lambda^2-e^{2i\phi})(b_1-ib_2)= ((y-1)(\lambda+e^{i\phi})+2e^{i\phi})
t^{-1}\z 
\\
&\phantom{---------------} +\Bigl(\frac y2 (\theta_0-\theta_1+\theta_{\infty})
(\lambda+e^{i\phi})- \theta_0(\lambda-e^{i\phi}) \Bigr) t^{-1}.
\end{align*}
Let loops $\hat{l}_0,$ $\hat{l}_1$ and a point $\hat{p}_{\mathrm{st}}$ 
in the $\lambda$-plane be the images of $l_0$, $l_1$ and $p_{\mathrm{st}}$
under \eqref{3.5}. The loops $\hat{l}_0,$ $\hat{l}_1$ start from 
$\hat{p}_{\mathrm{st}}$ and surround $\lambda=-e^{i\phi}$, $\lambda=e^{i\phi}$, 
respectively; and $\arg \hat{p}_{\mathrm{st}}=\pi/2$ (cf. Figure \ref{loops0}).
%%%%%%%%%%%%%%%%%%%%%%%%%%%%%%%%%%%%%%%%%%%%%%%%%%%%%%%%%%%%%%%%%%%%
%%%%%%%%%%%%%%%%%%%%%%%%%%%%%%%%%%%%%%%%%%%%%%%%%%%%%%%%%%%%%%%%%%%
%%%%%%%%% Figure 3.0 %%%%%%%%%%%%%%%%%%%%%%%%%%%%%%%%
%%%%%%%%%%%%%%%%%%%%%%%%%%%%%%%%%%%%%%%%%%%%
{\small
\begin{figure}[htb]
\begin{center}
\unitlength=0.77mm
%%%%%%%%%%%%%%%%%%%%%%%%%%%%%%%%%%%%%
%%%%%%%%%%%%%%%%%%%%%%%%%%%%
\begin{picture}(55,50)(-10,-5)
\thicklines
\put(16,40){\circle*{1}}
\put(16,40){\line(-2,-5){14.13}}
\put(16,40){\line(1,-2){13.57}}
\put(0,0){\circle{10}}
\put(0,0){\circle*{1}}
\put(31.8, 8.2){\circle{10}}
\put(31.8, 8.2){\circle*{1}}
%%%%%%%%%%%%%%%%%%%%%%%%%%%%
\thinlines
\put(12,25){\line(-2,-5){3}}
\put(12,25){\vector(1,3){0}}
\put(6,20.2){\line(2,5){3}}
\put(6,20.2){\vector(-1,-3){0}}
\qbezier(-2,-7)(0,-8)(2,-7)
\put(2,-7){\vector(3,2){0}}
\put(22,23.5){\line(-1,2){3}}
\put(23,30){\line(1,-2){3}}
\put(22,23.5){\vector(1,-2){0}}
\put(23,30){\vector(-1,2){0}}
\qbezier(34.7,1.9)(36.8,2.6)(37.6,4.7)
\put(37.6,4.7){\vector(1,4){0}}
%%%%%%%%%%%%%%%%%%
\put(14,45){\makebox{$\hat{p}_{\mathrm{st}}$}}
\put(38.5,8.5){\makebox{$e^{i\phi}$}}
\put(-17,0){\makebox{$-e^{i\phi}$}}
\put(28.8,19){\makebox{$\hat{l}_{1}$}}
\put(0.5,15){\makebox{$\hat{l}_{0}$}}
\end{picture}
%%%%%%%%%%%%%%%%%%%%%%%%%%%%%%%%%
\end{center}
\caption{Loops $\hat{l}_0$ and $\hat{l}_1$ on the $\lambda$-plane}
\label{loops0}
\end{figure}
}
%%%%%%%%%%%%%%%%%%%%%%%%%%%%%%%%%%%%%%%%%%%%%%%%%%
\par
Then \eqref{3.6} admits the matrix solution $Y(t,\lambda)=\exp(-\varpi(t,\phi)
\sigma_3) \Xi(e^{i\phi}t, (e^{-i\phi}\lambda +1)/2)$ (cf. \eqref{1.2}) with
the properties:
\par
(i) $Y(t,\lambda)$ has the asymptotic representation
%%%%%%% (3.8) %%%%%%%%%%%%%%%%%%
\begin{equation}\label{3.8}
Y(t,\lambda) =(I+O(\lambda^{-1})) \exp\Bigl(\frac 14 (t\lambda-2\theta_{\infty}
\log\lambda) \sigma_3 \Bigr)
\end{equation}
as $\lambda\to\infty$ through the sector $|\arg\lambda -\pi/2|<\pi,$ the
branch of $\log\lambda$ being taken in such a way that $\im(\log\lambda)
\to \pi/2$ as $\lambda\to \infty$ through this sector;
\par
(ii) the isomonodromy deformation yields the same
monodromy data $M^0,$ $M^1,$ $S_1,$ $S_2$ as in Section \ref{sc2}, where 
$M^0$ and $M^1$ are defined by $Y^{(\hat{l}_{\nu})}(t,\lambda)=Y(t,\lambda)
M^{\nu}$ for $\nu=0,1$ with $Y^{(\hat{l}_{\nu})}(t,\lambda)$ denoting  
the analytic continuation of $Y(t,\lambda)$ 
along the loop $\hat{l}_{\nu}$, and $S_k$ are such that
$Y_{k+1}(t,\lambda)=Y_k(t,\lambda)S_k$ with $Y_k(t,\lambda)$ 
($Y_2(t,\lambda)=Y(t,\lambda)$) having the same 
asymptotic representation as \eqref{3.8} in the sector 
$|\arg \lambda-\pi/2-(k-2)\pi|<\pi$; 
\par
(iii) system \eqref{3.6} has the isomonodromy property if and only if
$(y,\z,\hat{u})$ with
$\hat{u}=u\exp(-2\varpi(t,\phi))$ satisfies
%%%%%% (3.9) %%%%%%%%%%%%%%%%%%%%%
\begin{equation}\label{3.9}
\begin{split}
&ty_t= e^{i\phi}ty - 2\z (y-1)^2 -\frac{(y-1)} 2 \Bigl( (\theta_0-\theta_1+
\theta_{\infty})y -  (3\theta_0+\theta_1+\theta_{\infty}) \Bigr),
\\
&t\z_t= y\z \Bigl(\z+\frac 12 (\theta_0-\theta_1+\theta_{\infty}) \Bigr)
 - \frac 1y (\z+\theta_0) \Bigl(\z+ \frac 12(\theta_0+\theta_1+\theta_{\infty})
 \Bigr),
\\
&\frac{t\hat{u}_t}{\hat{u}}= -2\z -\theta_0
+y \Bigl(\z+\frac 12 (\theta_0-\theta_1+\theta_{\infty}) \Bigr)
 + \frac 1y \Bigl(\z+ \frac 12(\theta_0+\theta_1+\theta_{\infty}) \Bigr)
\end{split}
\end{equation}
$(y_t=dy/dt)$, and then $y(e^{i\phi}t)$ is parametrised by 
%%%% \in \mathcal{Y}^*(\mathrm{P}_{\mathrm{V}})$ 
$(M^0,M^1)\in \mathcal{M}_{(\theta_0,\theta_1,\theta_{\infty})}.$
%%%%%%% Remark 3.3 %%%%%%%%%%%%%%%%%%%%%%%%%%
\begin{rem}\label{rem3.2}
In what follows we denote $y(e^{i\phi}t)$ by $y(t)$ for brevity,
and set
%%%%%%% (3.10) %%%%%%%
\begin{equation}\label{3.10}
\z= -\frac{(y_t -e^{i\phi}y)t}{2(y-1)^2} - \frac{1}{4}
(\theta_0-\theta_1+\theta_{\infty}) +\frac{\theta_0+\theta_1}{2(y-1)},
\end{equation}
which is the first equation of \eqref{3.9}.
\end{rem}
%%%%%%%%%%%%%%%%%%%%%%%%%%%
%%%%%%%%%%%%%%%%%%%%%%%%%
%%%%%%%% Remark 3.4 %%%%%%
\begin{rem}\label{rem3.6}
Let $\mathbf{s}$ be a substitution given by $e^{i\phi} \mapsto -e^{i\phi},$ 
$y \mapsto y^{-1},$ $(\theta_0,\theta_1)\mapsto (\theta_1,\theta_0)$. It is 
easy to see that $\mathbf{s}(\z)=-\z-(\theta_0+\theta_1+\theta_{\infty})/2.$
Then system \eqref{3.6} is invariant under the extension of $\mathbf{s}$: 
$(\mathbf{s}, Y \mapsto y^{\sigma_3/2} Y).$ 
\end{rem}
%%%%%%%%%%%%%%%%%%%%%%%%%%%%%%%%%%%%
The uniqueness of the Riemann-Hilbert correspondence for system \eqref{3.6}
will be used in the justification procedure in Section \ref{sc6}.
%%%%%%%%%%%%%%%%%%%%%%%%%%%%%%%%%%%%
%%%% Proposition 3.a %%%%%%%%%%%%%%%%
\begin{prop}\label{prop3.a}
If $M^0\not=I$ and $M^1\not=I$, a canonical solution $Y_2(t,\lambda)$ of 
\eqref{3.6} corresponding to $(M^0,M^1)$ is uniquely determined.
\end{prop}
%%%%%%%%%%%%%%%%%%%%%%%%%%
\begin{proof}
Let $Y_2(\lambda)$ and $\tilde{Y}_2(\lambda)$ be canonical solutions defining
$(M^0,M^1).$ Suppose that $Y_2(\lambda)$ and $\tilde{Y}_2(\lambda)$ admit the
same asymptotic representation \eqref{3.8} as $\lambda\to\infty$ in the sector
$ |\arg\lambda-\pi/2|<\pi,$ and set, around $\lambda=-e^{i\phi}$, i.e.,
$\lambda_-:=\lambda+e^{i\phi}=0,$
\begin{equation*}
Y_2(\lambda)=G_0(I+O(\lambda_-))\lambda_-^{(\theta_0/2)\sigma_3+L}C_0,
\quad
\tilde{Y}_2(\lambda)=\tilde{G}_0(I+O(\lambda_-))
\lambda_-^{(\theta_0/2)\sigma_3+\tilde{L}}\tilde{C}_0,
\end{equation*}
where $G_0,$ $\tilde{G}_0,$ $C_0,$ $\tilde{C}_0 \in SL_2(\mathbb{C})$ and
$L=(l_{ij})$ and $\tilde{L}=(\tilde{l}_{ij})$ are upper-triangular (respectively,
lower-triangular) matrices with vanishing diagonal entries.
Then we have $M^0=C_0^{-1}e^{\pi i\theta_0\sigma_3}e^{2\pi iL}C_0
=\tilde{C}_0^{-1}e^{\pi i\theta_0\sigma_3}e^{2\pi i\tilde{L}}\tilde{C}_0,$
and hence $M^0\not=\pm I$ implies either of the cases: (i) $\theta_0\not\in 
\mathbb{Z}$ and $L=\tilde{L}=0;$ (ii) $\theta_0 \in \mathbb{Z}_{\ge 0}$ and
$l_{12}\tilde{l}_{12}\not=0$ (respectively, $\theta_0 \in \mathbb{Z}_{\le 0}$ 
and $l_{21}\tilde{l}_{21}\not=0$). The relation $(C_0\tilde{C}_0^{-1})^{-1}
e^{\pi i\theta_0\sigma_3}e^{2\pi iL}C_0\tilde{C}_0^{-1}=e^{\pi i\theta_0\sigma_3}
e^{2\pi i\tilde{L}}$ leads to 
$C_0\tilde{C}_0^{-1}=\mathrm{diag}[\alpha,\alpha^{-1}]$
or $\mathrm{diag}[\alpha,\alpha^{-1}]+L_*$ with some $\alpha\not=0$ 
and some upper-triangular (respectively, lower-triangular) matrix $L_*$ with
vanishing diagonal entries. By this fact
$$
Y_2(\lambda)\tilde{Y}_2(\lambda)^{-1}=G_0(I+O(\lambda_-))
\lambda_-^{(\theta_0/2)\sigma_3}e^{2\pi iL}C_0\tilde{C}_0^{-1}
e^{-2\pi i\tilde{L}}\lambda_-^{-(\theta_0/2)\sigma_3}
(I+(\lambda_-))\tilde{G}_0^{-1}
$$
is holomorphic around $\lambda=-e^{i\phi}$. Similarly 
$Y_2(\lambda)\tilde{Y}_2(\lambda)^{-1}$ is also holomorphic
around $\lambda=e^{i\phi}$ if $M^1\not=\pm I.$ Observing that
$Y_2(\lambda)\tilde{Y}_2(\lambda)^{-1}=I+O(\lambda^{-1})$ as $\lambda\to \infty$
through $|\arg \lambda-\pi/2|<\pi$, we conclude 
$Y_2(\lambda)=\tilde{Y}_2(\lambda)$ by the Phragm\'en-Lindel\"of reasoning.
\end{proof}
%%%%%%%%%%%%%%%%%%%%%%%%%%%%%%%%%%%%%%%%%%%
%%%%% ssc 3.4 %%%%%%%%%%%%%%%%%%%%%
\subsection{Characteristic roots, turning points and Stokes curves}\label{ssc3.4}%%%%%%%%%%%%%%%%%
In the remaining part of this section, Sections \ref{sc4} and \ref{sc5} we
are concerned with the direct monodromy problem for system \eqref{3.6}, 
that is, calculation of the monodromy. Then $y,$ $\mathfrak{z}$ and $u$ may 
be treated as arbitrary
complex parameters. It is also possible, in place of $\mathfrak{z}$, to choose
$y^*=y_t$ related to $\mathfrak{z}$ via \eqref{3.10}, and then let us suppose
that $y,$ $y^*$ and $u$ are arbitrary complex parameters. 
\par
To calculate the monodromy data for system \eqref{3.6} we need to know the 
characteristic roots of $\mathcal{B}(t,\lambda)$ and their turning points.
The characteristic roots $\pm \mu=\pm \mu(t,\lambda)$ are given by
$$
\mu^2 = - \det \mathcal{B}(t,\lambda) =b_3^2 +(-ib_2+b_1)(ib_2+b_1)=
b_1^2 + b_2^2 + b_3^2.
$$
Using \eqref{3.10}, we obtain
%%%%% (3.11) %%%%%%%%%
\begin{align}\label{3.11}
4( e^{2i\phi}-\lambda^2)^2\mu^2 &= \frac 14 (e^{2i\phi}-\lambda^2)
( e^{2i\phi}a_{\phi}-\lambda^2+4\theta_{\infty} t^{-1}\lambda)
\\
\notag
&\phantom{---} +2(\theta_1^2-\theta_0^2)e^{i\phi}t^{-2}\lambda +
2(\theta_0^2+\theta_1^2) e^{2i\phi}t^{-2},
\end{align}
%%%%%%%%
where $a_{\phi}=a_{\phi}(t)$ is given by
%%%%%% (3.12) %%%%%%%%%
\begin{align}\label{3.12}
a_{\phi} = & 1- \frac{4(e^{-2i\phi} (y^*)^2- y^2)}{y(y-1)^2} + 4e^{-i\phi}
(\theta_0+\theta_1)\frac{y+1}{y-1} t^{-1}
\\
\notag
& \phantom{---} + e^{-2i\phi}\frac{y-1}{y} ((\theta_0-\theta_1 +\theta_{\infty}
)^2 y -(\theta_0-\theta_1 -\theta_{\infty})^2 )t^{-2}.
\end{align}
%%%%%%%%%%%
%%% as $t\to\infty$. 
%%% outside the exceptional set $\bigcup_{\sigma\in Z(1)}
%%% \{ |e^{i\phi}t -\sigma |<\delta_1 \}$, $Z(1)=\{\sigma; y(\sigma)=1\}.$ 
%%% Note that this is also valid in the 
%%% cheese-like strip containing the ray $\arg x=\phi$:
%%% $$
%%% S'(\phi, t'_{\infty}, \kappa_0, \delta_1) = \{e^{i\phi}t\,; \,\, \re t
%%% >t'_{\infty},\, |\im t|<\kappa_0 \} \setminus \bigcup_{\sigma\in Z(1)}
%%% \{|e^{i\phi}t-\sigma| \le \delta_1 \},
%%% $$
%%% $t'_{\infty}$ and $\kappa_0$ $(>2\delta_1)$ being given numbers. 
%%%%%%%%%%%%%%%%%%%%
By \eqref{3.11}, as long as $y\not=0,1,\infty$ and $y^*\not=\infty,$ 
the turning points, that is, the zeros of $\mu$ are given by the following.
%%%%%%%%%%%%%%%%%%%%%%%%%%%%%%%%
%%%%%% Proposition 3.2 %%%%%%%%%%
\begin{prop}\label{prop3.5}
For each $t$, let the square roots of $a_{\phi}=a_{\phi}(t)$ be denoted by
$\pm a_{\phi}^{1/2}$, where $\re a_{\phi}^{1/2} \ge 0$ and $\im a_{\phi}^{1/2}
>0$ if $\re a_{\phi}^{1/2}=0.$ 
Then, for $|\phi|<\pi/2,$ the turning points are
\begin{align*}
\lambda_1(t) &= e^{i\phi} a_{\phi}^{1/2} + 2\theta_{\infty}t^{-1} +O(t^{-2}),
\quad & 
\lambda_2(t) &= - e^{i\phi} a_{\phi}^{1/2} + 2\theta_{\infty}t^{-1} +O(t^{-2}),
\\
\lambda_1^0(t) &= e^{i\phi} +O(t^{-2}),
\quad &
\lambda_2^0(t) &= - e^{i\phi} +O(t^{-2})
\end{align*}
as $t\to\infty$, and these are simple. Furthermore
$$
\mu^2 = \frac{(\lambda-\lambda_1)(\lambda-\lambda_2)
(\lambda-\lambda_1^0)(\lambda-\lambda_2^0)}{16(\lambda-e^{i\phi})^2(\lambda
+e^{i\phi})^2}.
$$
\end{prop}
%%%%%%%%%%%%%%%%%%%%%%%%%%%%%%
%%%%%%% Remark 3.5 %%%%%%%%%%%%%%
\begin{rem}\label{rem3.3b}
For a solution $y(t)$ of (P$_{\mathrm{V}}$), let us consider $a_{\phi}(t)$
with $(y,y^*)=(y(t),y_t(t))$. 
On the positive real axis all the solutions corresponding
to the monodromy data such that $m^0_{11}m^0_{21}m^1_{11}m^1_{12} \not=0$ are 
given by \cite[Theorems 3.1 and 4.1]{Andreev-Kitaev}. By using the expressions
of these solutions it is easy to verify $a_0(t)$ $(=a_{\phi}(t)|_{\phi=0}) 
 \ll t^{-\varepsilon}$ as 
$t \to \infty$ (for solutions of \cite[Theorem 4.1]{Andreev-Kitaev}, as 
$t\to \infty$ along a suitable path avoiding poles). Then 
$\re a_{\phi}(t)^{1/2} \ll |t^{-\varepsilon}|+o(1)$ uniformly in $t$ 
for sufficiently small $|\phi|$, which implies that, as long as 
$m^0_{11}m^0_{21}m^1_{11}m^1_{12} \not=0$, every corresponding 
solution fulfills $0 \le \re a_{\phi}(t)^{1/2} <1$ if $|\phi|$ is sufficiently 
small. On the other hand, for a general solution in \cite[Theorem 2.18]{S-2018} 
along the imaginary axis, $a_{\pi/2}(t)= 1+O(t^{-\varepsilon}).$ 
\end{rem}
%%%%%%%%%%%%%%%%%%%%%%%%%%
%%%%%%% Remark 3.6 %%%%%%%%%%%%
\begin{rem}\label{rem3.3c}
To the monodromy data such that $m^0_{11}m^0_{21}m^1_{11}m^1_{12}=0$ correspond
truncated solutions in sectors containing the positive real axis \cite{Andreev},
\cite{Andreev-Kitaev-2}. Let us consider $a_{\phi}(t)$
with $(y,y^*)=(y(t),y_t(t))$. Then we have $a_{\phi}(t) \ll t^{-1}$
for $\phi$ in some intervals containing $\phi=0$. In the case $m^1_{11}=0$ the
solution $y(x) \sim -1+ cx^{-1/2} e^{ix/2}$ in $0 \le \arg x \le \pi$
\cite[Proposition 5]{Andreev-Kitaev-2}, \cite[Corollary 5.2]{Andreev-Kitaev}
satisfies $a_{\phi}(t) \ll t^{-1}$ for $0\le \phi \le \pi.$ If $m^0_{11}=
m^1_{11}=0,$ then $y(x) \sim -1+ 4(\theta_0+\theta_1-1)x^{-1}$ in $|\arg x|
<\pi$ \cite[Proposition 2]{Andreev-Kitaev-2}, \cite[Corollary 5.3]
{Andreev-Kitaev} satisfies $a_{\phi}(t)\ll t^{-2}$ for $|\phi|<\pi.$ For
$m^1_{12}=0$ or $m^0_{21}=0$, truncated solutions such that 
$a_{\phi}(t) \ll t^{-1}$ for $|\phi|<\pi/2$ are given by \cite{Andreev}.
\end{rem}
%%%%%%%%%%%%%%%%%%%%%%%%%%%
%%%%%%%%%%%%%%%%%%%%%%%%%%%%%%%%%
%%%%%%%%%%%%%%%%%%%%%%%
By \eqref{3.11} the characteristic root $\mu=\mu(t,\lambda)$ is written in the
form
%%%%%%%%%%%%%%%%%%%%%%%%%%%%%%%%%%%
%%%%% (3.13) %%%%%%%%%%%%%%%
\begin{equation}\label{3.13}
\mu= \frac 14 \sqrt{ \frac{e^{2i\phi}a_{\phi}-\lambda^2}{e^{2i\phi}-\lambda^2}}
+ \frac {\theta_{\infty} \lambda t^{-1} }{2 \sqrt{(e^{2i\phi}-\lambda^2)
(e^{2i\phi}a_{\phi} -\lambda^2) }} +g_2(t,\lambda) t^{-2}
\end{equation}
%%%%%%%%%%%%%%%%%%%%%%%%%%
as $t \to \infty.$
Here $g_2(t,\lambda)$ has branch points at $\lambda^0_{1,2},$ $\lambda_{1,2}$,
$\pm e^{i\phi}$ and $\pm e^{i\phi} a_{\phi}^{1/2}$, but it fulfills 
$g_2(t,\lambda) \ll 1$ if  
$|\lambda^2-e^{2i\phi}a_{\phi}|^{-1} +|\lambda^2-e^{2i\phi} |^{-1}\ll 1.$
The algebraic function $\mu(t,\lambda)$ is given on the Riemann surface
consisting of two copies of $\lambda$-plane $\mathbb{P}_+$ and $\mathbb{P}_-$
glued along the cuts 
$[\lambda_1, \lambda^0_1],$ $[\lambda^0_2, \lambda_2]$ (cf. Figure 
\ref{stokes0}, (a)).
In \eqref{3.13}, each square root is fixed in such a way that 
$$
 \sqrt{\frac {e^{2i\phi}a_{\phi} -\lambda^2}{e^{2i\phi}-\lambda^2}}\to 1, \quad
 \lambda^{-2} \sqrt{(e^{2i\phi}a_{\phi} -\lambda^2)
(e^{2i\phi}-\lambda^2)} \to - 1
$$
as $\lambda \to \infty$ on $\mathbb{P}_{+}.$ Then, 
for $a_{\phi}^{1/2}$ as in Proposition \ref{prop3.5},
$a^{-1/2}_{\phi}\sqrt{(e^{2i\phi}a_{\phi}-\lambda^2)/(e^{2i\phi}-\lambda^2)}$, 
$e^{-2i\phi} a^{-1/2}_{\phi} \sqrt{(e^{2i\phi}a_{\phi}-\lambda^2)
(e^{2i\phi}-\lambda^2)} \to 1$ as $\lambda\to 0$ on $\mathbb{P}_+$.
\par
A Stokes curve is defined by
$$
\re \int^{\lambda}_{\lambda_*} \mu(t,\tau) d\tau =0,
$$
where $\lambda_*$ is a turning point \cite{F}. This curve connects 
$\lambda_*$ to another turning point, $\pm e^{i\phi}$ or $\infty.$ 
The Stokes graph consists of Stokes curves, turning points and singular points.
\par
Our WKB analysis will be carried out under the supposition 
$a_{\phi}(t)\to A_{\phi}$ as $t\to \infty$ (cf. \eqref{4.1}), where
$A_{\phi}$ is a unique solution of the Boutroux equations \eqref{2.1}. 
Suppose that $0<|\phi|<\pi/2$.  
Let us consider the {\it limit} Stokes graph for $t=\infty$, in which the 
limit turning points $\lambda_{1,\,2}(\infty)$ are also denoted by 
the same symbols $\lambda_{1,\,2}$ as for $t\not=\infty$. 
Set $\mathbb{P}_+^{\infty} \cup \mathbb{P}_-^{\infty}=\lim_{t\to
\infty} \mathbb{P}_+ \cup \mathbb{P}_-$, which is a two-sheeted Riemann
surface glued along the cuts $[-e^{i\phi},\lambda_2],$ $[\lambda_1, e^{i\phi}]$
with $\lambda_{1,\,2}=\lambda_{1,\,2}(\infty)=\pm e^{i\phi}A_{\phi}^{1/2}$ and
$\lambda^0_{1,\,2}(\infty)=\pm e^{i\phi}$.
As long as the turning points do not coalesce,  
the limit Stokes curve for $t=\infty$ reflects the Boutroux equations.  
Then this limit
Stokes graph on $\mathbb{P}_+^{\infty}$ is considered to be as in 
Figure \ref{stokes0} (b), (c) (cf. Proposition \ref{propA.17}), 
in which, if $0<\phi <\pi/2,$ the limit
Stokes curves connect $\lambda_1$ to $e^{i\phi}$, $\lambda_2$ and $ i \infty,$ 
and $\lambda_2$ to $-e^{i\phi}$ $\lambda_1$ and $- i \infty.$ 
%%%%%%%%%%%%%%%%%%%%%%%%%%%%%%%%%%%%%%%%%%%%%%%%%%%%%%%%%%%%
%%%%%%%%%%%%%%%%%%%%%%%%%%%%%%%%%%%%%%%%%%%%%%%%%%%%%
%%%%%%%%%%%%%%  Figure 3.1 %%%%%%%%%%%%%%%%%%
%%%%%%%%%%%%%%%%%%%%%%%%%%%%%%%%%%%%%%%
%%%%%%%%%%%%%%%%%%%%%%%%%%%%%%%%%%%%%%%
{\small
\begin{figure}[htb]
\begin{center}
\unitlength=0.75mm
%%%%%%%%%%%%%%%%%%%%%%%%%%%%%%%%%%
%%%%%%%%%%%%%%%%%%%%%%%%%%%%%%%%%%
%%%%%%%%%%%%%%%%%%
\begin{picture}(55,55)(-25,-30)
\put(-22,-5){\circle*{1}}
\put(-7,-8){\circle*{1}}
\put(7,8){\circle*{1}}
\put(22,5){\circle*{1}}
\put(-23,-8){\circle*{1}}
\put(23,8){\circle*{1}}

\put(-22,-4.5){\line(5,-1){15}}
\put(-22,-5.5){\line(5,-1){15}}
\put(7,7.5){\line(5,-1){15}}
\put(7,8.5){\line(5,-1){15}}
\put(-25,-1){\makebox{$\lambda^0_2$}}
\put(-4,-9){\makebox{$\lambda_2$}}
\put(-1,8){\makebox{$\lambda_1$}}
\put(20,-1){\makebox{$\lambda^0_1$}}
\put(-29,-14){\makebox{$-e^{i\phi}$}}
\put(22,10){\makebox{$e^{i\phi}$}}
\put(-25,-34){\makebox{(a) $[\lambda_1,\lambda_1^0]$, 
$[\lambda_2^0,\lambda_2]$}}
\end{picture}
%%%%%%%%%%%%%%%%%%
\quad\quad
%%%%%%%%%%%%%%%%%%%%%%%%%%%%%%%%%%
\begin{picture}(60,55)(-30,-30)
\put(19,-6){\circle*{1.5}}
\put(8,-8){\circle*{1.5}}
\put(-8,8){\circle*{1.5}}
\put(-19,6){\circle*{1.5}}

\thicklines
 \qbezier (8,-8) (6,-4) (0, 0)
\qbezier (-8,8) (-6,4) (0, 0)
\qbezier (8,-8) (3,-16) (2, -24)
\qbezier (8,-8) (17,-8) (19, -6)
\qbezier (-8,8) (-3,16) (-2, 24)
\qbezier (-8,8) (-17,8) (-19, 6)
\put(21,-11.5){\makebox{$e^{i\phi}$}}
\put(-25,7.5){\makebox{$-e^{i\phi}$}}
\put(0,-10){\makebox{$\lambda_1$}}
\put(-5,7){\makebox{$\lambda_2$}}
\put(-17,-34){\makebox{(b) $-\pi/2<\phi <0$}}
\end{picture}
%%%%%%%%%%%%%%%%%%
\quad\quad
%%%%%%%%%%%%%%%%%%%%%%%%%%%%%%%
\begin{picture}(60,55)(-30,-30)
\put(-19,-6){\circle*{1.5}}
\put(-8,-8){\circle*{1.5}}
\put(8,8){\circle*{1.5}}
\put(19,6){\circle*{1.5}}

\thicklines
 \qbezier (-8,-8) (-6,-4) (0, 0)
\qbezier (8,8) (6,4) (0, 0)
\qbezier (-8,-8) (-3,-16) (-2, -24)
\qbezier (-8,-8) (-17,-8) (-19, -6)
\qbezier (8,8) (3,16) (2, 24)
\qbezier (8,8) (17,8) (19, 6)
\put(-31,-11.5){\makebox{$-e^{i\phi}$}}
\put(20,7.5){\makebox{$e^{i\phi}$}}
\put(-5,-10){\makebox{$\lambda_2$}}
\put(1,7){\makebox{$\lambda_1$}}
\put(-17,-34){\makebox{(c) $0 <\phi <\pi/2$}}
\end{picture}
%%%%%%%%%%%%%%%%%%
%%%%%%%%%%%%%%%%%%%%%%%%%%%%%%
%%%%%%%%%%%%%%%%%%%%%%%%%%%%%%%%%
\end{center}
\caption{Cuts on $\mathbb{P}_+$ and the limit Stokes graph on $\mathbb{P}_+
^{\infty}$}
\label{stokes0}
\end{figure}
}
%%%%%%%%%%%%%%%%%%%%%%%%%%%%%%%%%%%%%%%%%%%%%%%%%%%%%%%%
%%%%%%%%%%%%%%%%%%%%%%%%%%%%%%%%%%%%%%%%%%%%%%%%%%%%%%%%%%%%%%%%%%%%%%%5
%%%%%%%%%%%%%%%%%%%%%%%%%%%%%%%%%%%%%%%%%%%%
\par
An unbounded domain $\mathcal{D} \subset \mathbb{P}_+^{\infty} \cup 
\mathbb{P}_-^{\infty}$
is called a canonical domain if, for each $\lambda \in \mathcal{D}$, there
exist contours $\mathcal{C}_{\pm}(\lambda)\subset \mathcal{D}$ ending at 
$\lambda$ such that
%%%%%%%%%%%%%%%%%%%%%%%%%%%%
\begin{align}\label{3.a}
\re \int^{\lambda}_{\lambda_-} \mu(\tau) d\tau \to -\infty
\quad \biggl (\text{respectively,} \,\,\, 
\re \int^{\lambda}_{\lambda_+} \mu(\tau) d\tau \to +\infty \biggr )
\end{align}
%%%%%%%%%%%%%%%%%%%%%%%%%%%
as $\lambda_- \to \infty$ along $\mathcal{C}_-(\lambda)$ 
(respectively, as $\lambda_+ \to \infty$ along $\mathcal{C}_+(\lambda)$)
(see \cite{F}, \cite[p. 242]{FIKN}). The interior of a canonical domain 
contains exactly one Stokes curve, and
its boundary consists of Stokes curves.  
%%%%%%%%%%%%%%%%%%%%%%%%%%%%%%%%%%%%%%%
%%%%%% ssc 3.5 %%%%%%
\subsection{WKB-solution}\label{ssc3.5}
%%%%%%%%%%%%%%%%%%%%%%
The following is a WKB-solution of \eqref{3.6} in a canonical domain.
%%%%%%%%%%%%%%%%%%%%%%%%%%%%%%%%%%%
%%%%%%%% Proposition 3.3 %%%%%%%%%
\begin{prop}\label{prop3.6}
In the canonical domain whose interior contains a Stokes curve issuing from
the turning point $\lambda_1$ or $\lambda_2$, system \eqref{3.6} with $\hat{u}
\equiv 1$ admits a solution expressed by
$$
\Psi_{\mathrm{WKB}}(\lambda) = T(I+O(t^{-\delta})) \exp\biggl(\int
^{\lambda}_{\tilde{\lambda}_*} \Lambda(\tau) d\tau \biggr)
$$
outside suitable neighbourhoods of zeros and poles of $b_1\pm ib_2$ as long as
$|\lambda\pm e^{i\phi} |\gg t^{-2+2 \delta},$ %% $|\lambda-\lambda_{\iota}^0|,$
$|\lambda-\lambda_{\iota}| \gg t^{-2/3 +(2/3)\delta}$ $(\iota =1,2)$,  
$0 < \delta <1$ being arbitrary. 
Here $\tilde{\lambda}_*$ is a base point near $\lambda_1$ or
$\lambda_2$, and $\Lambda(\tau)$ and $T$ are given by
$$
\Lambda(\lambda) = t\mu \sigma_3 -\diag T^{-1}T_{\lambda}, \quad
T= \begin{pmatrix}
    1  &  \dfrac{b_3-\mu}{b_1+ib_2} \\
  \dfrac{\mu - b_3}{b_1-ib_2}   &  1
\end{pmatrix}.
$$
\end{prop}
%%%%%%%%%%%%%%%%%%%%%%%%%%%%
%%%% Remark 3.7 %%%%%%%%%%%%%
\begin{rem}\label{rem3.4}
In this proposition
\begin{align*}
\diag T^{-1}T_{\lambda} &= \frac 1{2\mu(\mu+b_3)} \bigl( i(b_1b'_2 -b'_1b_2)
\sigma_3 + (b_3 \mu' - b'_3\mu)I ) \quad\quad
(b'_{\iota}=\partial b_{\iota} /\partial\lambda)
\\
&= \frac 14 \Bigl(1-\frac{b_3}{\mu} \Bigr) \frac{\partial}{\partial \lambda}
\log \frac{b_1+ib_2}{b_1-ib_2} \sigma_3  +\frac 12 \frac{\partial}{\partial
\lambda} \log \frac{\mu}{\mu +b_3} I.
\end{align*} 
\end{rem}
%%%%%%%%%%%%%%%%%%%%%
\begin{proof}
By $Y=T\tilde{Y}$ system \eqref{3.6} with $\hat{u}\equiv 1$ becomes
%%%%%% (3.15) %%%%%%
\begin{equation}\label{3.15}
\tilde{Y}_{\lambda} =(t\mu \sigma_3 -T^{-1}T_{\lambda})\tilde{Y}.
\end{equation}
To remove the off-diagonal part 
$R=T^{-1}T_{\lambda} -\diag T^{-1}T_{\lambda}$ set
$$
T_1 =\frac 1{2t\mu} \begin{pmatrix}  0 & R_{12}  \\  -R_{21} &  0 
\end{pmatrix},
\quad R_{12} =\frac{\mu +b_3}{2\mu} \frac{\partial}{\partial\lambda}
\Bigl( \frac{b_3-\mu}{b_1+ib_2} \Bigr), 
\quad R_{21} =\frac{\mu +b_3}{2\mu} \frac{\partial}{\partial\lambda}
\Bigl( \frac{\mu-b_3}{b_1-ib_2} \Bigr), 
$$
which fulfills $[t\mu\sigma_3, T_1] = R.$ Now we would like to find $X$
such that the transformation $\tilde{Y}=(I+T_1)(I+X) Z$ takes 
\eqref{3.15} to
$$
Z_{\lambda} =\Lambda Z =(t\mu\sigma_3 -\diag T^{-1}T_{\lambda})Z,
$$
that is,
$$
(T_1)_{\lambda}(I+X) +(I+T_1)X_{\lambda} +(I+T_1)(I+X)
\Lambda = (\Lambda -R)(I+T_1)(I+X).
$$
It follows that
%%%%%%% (3.16) %%%%%%%%%
\begin{equation}\label{3.16}
X_{\lambda} =[\Lambda,X] +(I+T_1)^{-1} Q(I+X)
\end{equation}
with $Q =-(T_1)_{\lambda} -R(I+T_1) +[\Lambda, T_1] = -(T_1)
_{\lambda} -T^{-1}T_{\lambda} T_1 +T_1 \diag T^{-1}T_{\lambda}.$
Then $\|Q\|$ is estimated as follows:
\par
(1) Near $\lambda= \mp e^{i\phi},$ we have $b_3,$ 
$|b_1 \pm ib_2| \ll |\lambda \pm e^{i\phi} |^{-1},$ 
$|b_1b'_2 -b'_1b_2 | \ll |\lambda \pm e^{i\phi}|^{-2},$
$\mu \asymp |\lambda \pm e^{i\phi}|^{-1/2},$ and hence $\| R\| \ll |\lambda
\pm e^{i\phi}|^{-1} \asymp \mu^2,$ $\| \diag T^{-1}T_{\lambda} \| \ll
|\lambda \pm e^{i\phi} |^{-1},$ $\|T_1\| \ll t^{-1} |\lambda \pm e^{i\phi}|
^{-1/2}$ and 
$\| Q \| \ll t^{-1}|\lambda \pm e^{i\phi}|^{-3/2} $.  
\par
(2) Near $\lambda= \lambda_{\iota}$ $(\iota=1,2)$ we have 
$b_3,$ $|b_1 b'_2-b'_1 b_2 | \ll 1,$
$\mu \asymp |\lambda - \lambda_{\iota}|^{1/2},$ and hence $\| R\| \ll |\lambda
-\lambda_{\iota}|^{-1},$ $\| \diag T^{-1}T_{\lambda} \| 
\ll |\lambda -\lambda_{\iota} |^{-1},$ 
$\|T_1\| \ll t^{-1} |\lambda -\lambda_{\iota}|^{-3/2},$ and
$\| Q \| \ll t^{-1} |\lambda-\lambda_{\iota}|^{-5/2} $. 
\par
(3) Near $\lambda=\infty,$ observe that $\mu=1/4+O(\lambda^{-1})$ on
$\mathbb{P}_+^{\infty}$ and $b_3=1/4 +O(\lambda^{-1})$. 
Then $\pm(b_3-\mu)/(b_1\pm ib_2)
= -(b_1\mp ib_2)/(\mu+b_3) \ll \lambda^{-1},$ which means $T=I+O(\lambda^{-1}).$
It is easy to see that $\|T_1 \| \ll t^{-1}\lambda^{-2},$ 
$\| Q\|\ll t^{-1}\lambda^{-3}$ near $\lambda=\infty$ on $\mathbb{P}_+^{\infty}$.
Near $\infty$ on $\mathbb{P}_-^{\infty},$ $\|T_1\| \ll t^{-1}\lambda^{-1},$ 
$\|Q\| \ll t^{-1}\lambda^{-2},$ since $\mu + 1/4 \ll t^{-1}\lambda^{-1}.$
\par\noindent
The $\sigma_3$-component of $\diag T^{-1}T_{\lambda}$ is  
$D_3(\lambda)=\tfrac14 (1-{b_3}/{\mu})(\log((b_1+ib_2)/(b_1-ib_2)))_{\lambda}
\sigma_3$ (cf.~Remark \ref{rem3.4}), which satisfies 
$\| D_3(\lambda)\|\ll |\lambda \mp e^{i\phi}|^{-1/2}$ near 
$\lambda=\pm e^{i\phi};$ $\ll |\lambda-\lambda_{\iota}|^{-1/2}$ near 
$\lambda=\lambda_{\iota};$ and $\ll \lambda^{-2}$ near $\lambda=\infty.$
\par
Every solution of the integral equation
\begin{align*}
X(\lambda) &= \int_{\mathcal{C}(\lambda)} \exp\Bigl(\int^{\lambda}_{\xi}
\Lambda(\tau)d\tau \Bigr) (I+T_1(\xi))^{-1} Q(\xi) (I+X(\xi))\exp\Bigl(
-\int^{\lambda}_{\xi}\Lambda(\tau) d\tau \Bigr) d\xi
\\
 &= \int_{\mathcal{C}(\lambda)} e^{(tM(\lambda,\xi)-J(\lambda,\xi))\sigma_3}
 (I+T_1(\xi))^{-1} Q(\xi) (I+X(\xi)) 
e^{-(tM(\lambda,\xi)-J(\lambda,\xi))\sigma_3} d\xi
\end{align*}
with
$$
M(\lambda,\xi)=\int^{\lambda}_{\xi} \mu(\tau)d\tau,\quad
J(\lambda,\xi)=\int^{\lambda}_{\xi} D_3(\tau)d\tau
$$
solves \eqref{3.16}, where the set of contours $\mathcal{C}(\lambda)$
ending in $\lambda$ is chosen in such a way that for 
(1,2)- (respectively, (2,1)-) entry is $\mathcal{C}_{-}(\lambda)$ 
(respectively, $\mathcal{C}_+(\lambda)$) (cf. \eqref{3.a}), and
that those for (1,1)- and (2,2)-entries are paths joining $\lambda_{\iota}$
to $\lambda$. In the canonical domain $J(\lambda,\xi)$ is
bounded uniformly as long as $\lambda$ and $\xi$ are outside neighbourhoods
of $\pm e^{i\phi}$ and $\lambda_{\iota}$.
In the canonical domain, as long as 
$|\lambda \pm e^{i\phi}|\gg t^{-2(1-\delta)}$, 
$|\lambda-\lambda_{\iota}|\gg t^{-(2/3)(1-\delta)}$, we have $\|T_1\| \ll
t^{-\delta}$, and
\begin{equation*}
\|X \| \ll \|Q\| \ll t^{-1}(|\lambda \pm e^{i\phi} |^{-1/2}
+|\lambda-\lambda_1|^{-3/2} +|\lambda-\lambda_2|^{-3/2} +1)
\\
\ll t^{-\delta},
\end{equation*}
which implies the proposition (cf.~the proof of \cite[Theorem 7.2]{FIKN}). 
\end{proof}
%%%%%%%%%%%%%%%%%%%%%%%%%%%%%%%%%%
%%%%%%%%%%%%%%%%%%%%%%%%%%%%%%%%%%
%%%%% ssc 3.6 %%%%%%
\subsection{Local solutions around turning points}\label{ssc3.6}
%%%%%%%%%%%%%%%%%%%%%%%%%%%
For $\iota=1$ or $2$, if $|\lambda-\lambda_{\iota}| \ll t^{-2/3},$ the 
WKB-solution given above fails in expressing the asymptotic behaviour. 
Consider the system
%%%%%%%%%%%%%%%%%%%%%%%%%%%%%%%%%%%%%
\begin{equation}\label{3.18}
\frac{dW}{d\zeta} = \begin{pmatrix}  0 & 1 \\ \zeta & 0 \end{pmatrix} W,
\end{equation}
which admits canonical matrix solutions 
$W_{\nu}(\zeta)$ $(\nu= 0, \pm1, \pm2, \ldots)$ such that
$$
W_{\nu}(\zeta) =\zeta^{-(1/4)\sigma_3}(\sigma_3+\sigma_1)(I+O(\zeta^{-3/2}))
\exp\bigl( \tfrac 23\zeta^{3/2} \sigma_3 \bigr)
$$
as $\zeta \to \infty$ through the sector
$\Sigma_{\nu}:$ $|\arg \zeta -(2\nu -1)\pi/3 |<2\pi/3,$ and that 
$W_{\nu+1}(\zeta)=W_{\nu}(\zeta)G_{\nu}$ with
$$
G_1= \begin{pmatrix}   1  &  -i \\  0  & 1 \end{pmatrix},  \quad
G_2= \begin{pmatrix}   1  &  0 \\  -i  & 1 \end{pmatrix},  \quad
G_{\nu+1} =\sigma_1 G_{\nu} \sigma_1.
$$
In particular 
$$
W_1(\zeta)=\begin{pmatrix}  \mathrm{Bi}(\zeta) & \mathrm{Ai}(\zeta) \\
       \mathrm{Bi}_{\zeta}(\zeta) & \mathrm{Ai}_{\zeta}(\zeta) \end{pmatrix},
$$
where $\mathrm{Ai}(\zeta)$ and $\mathrm{Bi}(\zeta)$ are the Airy functions
\cite{AS, HTF} such that
$\mathrm{Ai}(\zeta) \sim \zeta^{-1/4}\exp(-\frac 23\zeta^{3/2})$ as $\zeta
\to\infty$ in $|\arg\zeta|<\pi$ and that $\mathrm{Bi}(\zeta)=\omega^{-1/4}
\mathrm{Ai}(\omega^{-1}\zeta)$ with $\omega=e^{2\pi i/3}.$
Then we have the following local solution around each turning point.
%%%%%%%%%%%%%%%%%%%%%%%%%%%%%
%%%%%%% Proposition 3.4 %%%%%%%%
\begin{prop}\label{prop3.7}
For each turning point $\lambda_{\iota}$ $(\iota=1, 2)$ write $c_k=b_k(\lambda
_{\iota}),$ $c'_k=(b_k)_{\lambda}(\lambda_{\iota})$ $(k=1,2,3),$ and suppose
that $c_k,$ $c'_k$ are bounded and $c_1 \pm ic_2 \not=0.$ 
Let $W(\zeta)$ be a given matrix solution of \eqref{3.18}, and let
$\lambda-\lambda_{\iota} = (2\kappa)^{-1/3}t^{-2/3}(\zeta +\zeta_0)$
with $\kappa=c_1c'_1 +c_2c'_2 +c_3c'_3,$ $|\zeta_0| \ll t^{-1/3}.$  
Then system \eqref{3.6} with $\hat{u}\equiv 1$ admits a matrix solution given by
\begin{equation*}
 \Phi_{\iota}(\lambda)= T_{\iota} (I+O(t^{-\delta'}) ) 
\begin{pmatrix}
1  &  0  \\  0 &  \hat{t}^{-1}  \end{pmatrix} W(\zeta), \quad  
 T_{\iota} = \begin{pmatrix}
  1  &   - \dfrac{c_3}{c_1 + ic_2}   \\   - \dfrac{c_3}{c_1 -ic_2}  &  1
\end{pmatrix}
\end{equation*}
with $\hat{t} = 2(2\kappa)^{-1/3} (c_1-i c_2) t^{1/3}$  
as long as $|\zeta| \ll t^{1/3-\delta'/3},$ that is, 
$|\lambda-\lambda_{\iota} |\ll t^{-1/3 -\delta'/3}$, 
$0<\delta'< 1$ being arbitrary. 
\end{prop}
%%%%%%%%%%%%%%%%%%%%%%%
%%%%%%%%%%%%%%%%%%%%%%%%%%%%%%%%%%%
%%%%%%%%%%%%%%%%%%%%%%%%%%%%%%%%%%%%
\begin{proof}
Since $\mu^2=b_1^2 +b_2^2 +b_3^2,$ we have $c_1^2 +c_2^2 +c_3^2 =\mu(\lambda
_{\iota})^2 =0.$ Write $\mathcal{B}(t,\lambda) =\mathcal{B}_0(t) +\mathcal{B}_1
(t,\lambda)$ with
$$
\mathcal{B}_0(t)= \mathcal{B}(t,\lambda_{\iota}) 
=\begin{pmatrix}
c_3  &  c_1 -i c_2    \\ c_1+i c_2  &  -c_3  
\end{pmatrix},  \quad
\mathcal{B}_1(t,\lambda) =\begin{pmatrix}
\delta_3  &  \delta_1 -i\delta_2  \\ \delta_1+i\delta_2  & -\delta_3  
\end{pmatrix}, 
$$
where $\delta_k=b_k -c_k$ $(k=1,2,3).$ Set $\eta=\lambda-\lambda_{\iota}.$ 
Observing that 
$\delta_k=\sum_{j\ge 1}c_{k,j}\eta^j$ with $c_{k,1}=c'_k,$ 
and that 
\begin{align*}
&T_{\iota}^{-1}\mathcal{B}_0 (t) T_{\iota} = \begin{pmatrix}
0  &  2(c_1-i c_2)  \\ 0 & 0 
\end{pmatrix},  
\\
&T_{\iota}^{-1}\mathcal{B}_1(t,\lambda) T_{\iota} 
= \begin{pmatrix}
ic_3^{-1}(c_2\delta_1 -c_1\delta_2) & 
(c_1+ic_2)^{-1}(c_1\delta_1+ c_2\delta_2 -c_3\delta_3)  \\ 
 (c_1-ic_2)^{-1}(c_1\delta_1+ c_2\delta_2 +c_3\delta_3)  & 
- ic_3^{-1}(c_2\delta_1 -c_1\delta_2) 
\end{pmatrix},  
\end{align*}
we have
\begin{align*}
 T_{\iota}^{-1} & \mathcal{B}(t,\lambda) T_{\iota}
 = \begin{pmatrix} 0  &  \gamma_0  \\ 0 & 0 \end{pmatrix}
 +   \begin{pmatrix}
 f(\eta) & g(\eta)  \\  h(\eta)  & -f(\eta) \end{pmatrix} \eta,
\end{align*}
where $f(\eta),$ $g(\eta)$ and $h(\eta)$ are analytic around $\eta=0$ and
\begin{align*}
& \gamma_0= 2(c_1-i c_2),  \quad 
f(0)= ic_3^{-1}(c'_1c_2 -c_1c'_2),   
\\
&g(0)= (c_1+ic_2)^{-1}(c_1c'_1+ c_2c'_2 -c_3c'_3), \quad
 h(0)=(c_1-ic_2)^{-1}(c_1c'_1+ c_2c'_2 +c_3c'_3). 
\end{align*}
Note that 
$$
\begin{pmatrix} f(\eta)\eta  &  \gamma_0+g(\eta)\eta \\
              h(\eta)\eta  &  -f(\eta)\eta      \end{pmatrix}(I+ H(\eta))
=(I+ H(\eta)) \begin{pmatrix} 0  &  \gamma_0 \\
              h_*(\eta)\eta  &  0     \end{pmatrix},
$$
in which $H(\eta)=L(\eta)\eta$ with
\begin{align*}
L(\eta)&=
\begin{pmatrix}  l(\eta)  &   0 \\  l_*(\eta)  & -l(\eta)
\end{pmatrix}
:=\frac{1}{2\gamma_0+g(\eta)\eta}\Biggl(
\begin{pmatrix}   0  &  0  \\    -2f(\eta) &  0 \end{pmatrix}
+ g(\eta)\sigma_3 \Biggr), 
\\
h_*(\eta)&=\frac{h(\eta)(1+l(\eta)\eta)-f(\eta)l_*(\eta)\eta}
{1-l(\eta)\eta} =h(\eta)+O(\eta).
\end{align*}
By
$$
Y=T_{\iota}(I+H(\eta))Z=T_{\iota}(I+L(\eta)\eta)Z, \quad L(0)
=\begin{pmatrix} q & 0 \\
 p & -q \end{pmatrix}
%%%% \quad p=-\frac{i(c'_1c_2-c_1c'_2)}{2c_3(c_1-ic_2)}, \quad
%%% q=  -\frac{c_1c'_1 +c_2c'_2 - c_3c'_3}{4c_3^2}
$$
system \eqref{3.6} is changed into
\begin{align*}
\frac{dZ}{d\eta} &= t \mathcal{B}_{\iota}(t,\eta) Z,
\\
 \mathcal{B}_{\iota}(t,\eta)&= \begin{pmatrix}  0  & \delta_0 
        \\ h_*(\eta)\eta   & 0  \end{pmatrix} -t^{-1}(I+L(\eta)\eta)^{-1}
(L(\eta)\eta)'
\\
& = \begin{pmatrix}  0  & 2(c_1-i c_2) 
 \\
\kappa (c_1-ic_2)^{-1} \eta h_0(\eta)  &  0    \end{pmatrix}
+ \begin{pmatrix}  -q  &  0  \\ -p & q  
\end{pmatrix} t^{-1} + t^{-1}\sum_{j\ge 1}B^*_j \eta^j, 
\end{align*}
where $h_0(\eta)=h_*(\eta)h(0)^{-1}=1+\eta h_1(\eta)=1+O(\eta).$
Let $\varphi(\xi)$ be such that $\varphi(0)=1$ and that
$\varphi(\xi) h_0(\xi\varphi(\xi))((\xi\varphi(\xi))')^2
=\varphi(\xi) h_0(\xi\varphi(\xi))(\varphi(\xi)+\xi\varphi'(\xi))^2=1.$
Then the change of variables $\eta=\xi\varphi(\xi)$ takes this system
to
\begin{equation*}
\frac{dZ}{d\xi} = t \biggl(\hat{\mathcal{B}}_{\iota,0}(t,\xi) 
 +(1+O(\xi))t^{-1} \Bigl( L(0) + \sum_{j\ge 1}\hat{B}^*_j \xi^j\Bigr)\biggr)Z
\end{equation*}
with
\begin{align*}
 \hat{\mathcal{B}}_{\iota,0}(t,\xi) 
& = \begin{pmatrix}  0  & 2(c_1-i c_2)(\xi\varphi(\xi))' 
 \\
\kappa(c_1-ic_2)^{-1}\xi\varphi(\xi) h_0(\xi\varphi(\xi))(\xi\varphi(\xi))'
  &  0    \end{pmatrix}
\\
& = \begin{pmatrix}  0  & 2(c_1-i c_2)(\xi\varphi(\xi))' 
 \\
\kappa(c_1-ic_2)^{-1}\xi/(\xi\varphi(\xi))'  &  0    \end{pmatrix}.
\end{align*}
Apply the additional transformation $Z=\diag [1, 1/ (\xi\varphi(\xi))' ]V$, 
and set $\xi=\eta+O(\eta^2)=\beta z$ with $\beta=(2\kappa)^{-1/3} t^{-2/3}$
and $\hat{t}=2(c_1-ic_2)\beta t\asymp t^{1/3}$. Then
%%%%%%%%%%%%%%%%%%%%%%%%%%%%%%%%%%%%%%%
\begin{equation}\label{3.e}
\frac{dZ}{dz} = \Biggl( \begin{pmatrix} -q\beta  &  \hat{t}  \\
\hat{t}^{-1}z- p\beta   &  q\beta  \end{pmatrix} + O(\hat{t}^{-4}z ) \Biggr)Z
\end{equation}
%%%%%%%%%%%%%%%%%%%%%%%%%%%%%%%%%%%%%%%
as long as $|\beta z|\asymp t^{-2/3}z$ is sufficiently small. The further 
change of variables
$$
Z=\begin{pmatrix}  1 & 0 \\ 0 & \hat{t}^{-1}  \end{pmatrix}
\begin{pmatrix}  1 & 0 \\ q\beta & 1  \end{pmatrix} V,
\quad \zeta=z-\zeta_0, \quad  \zeta_0=p\beta \hat{t} -q^2\beta^2 \ll t^{-1/3}
$$
yields
%%%%%% (3.17) %%%%%%
\begin{equation}\label{3.17}
\frac{dV}{d\zeta} = \Biggl( \begin{pmatrix}  0 & 1 \\ \zeta & 0 \end{pmatrix}
+ \Delta(t,\zeta)  \Biggr) V, \quad \Delta(t,\zeta) \ll {t}^{-1}(|\zeta|+
|t^{-1/3}|).
%%%%% |{t}^{-2/3}\zeta^2| 
\end{equation}
Suppose that $V=P(t,\zeta)W$ reduces \eqref{3.17} to \eqref{3.18}. 
Then $P=P(t,\zeta)$ satisfies
\begin{equation}\label{3.b}
\frac{dP}{d\zeta}=\begin{pmatrix} 0 & 1 \\ \zeta & 0 \end{pmatrix} P
-P\begin{pmatrix} 0 & 1 \\ \zeta & 0 \end{pmatrix} +\Delta(t,\zeta)P.
\end{equation}
Note that, in each sector $\Sigma_{\nu}$ $(\nu=0,\pm 1)$, the function 
$P_{\nu}(\zeta)=I+X_{\nu}(\zeta)$ such that 
%%%%%%%%%%%%%%%%%%%%%%%%%%%%%%%%%%%%%
\begin{equation}\label{3.c}
X_{\nu}(\zeta) = \int_{\mathcal{C}_{\nu}(\zeta)} W_{\nu}(\zeta)W_{\nu}(\xi)^{-1}
\Delta(t,\xi)(I+X_{\nu}(\xi)) W_{\nu}(\xi) W_{\nu}(\zeta)^{-1} d\xi
\end{equation}
%%%%%%%%%%%%%%%%%%%%%%%%%%%%%%%%%
solves \eqref{3.b} in $\Sigma_{\nu}$.
Here $\mathcal{C}_{\nu}(\zeta)$ is a set of contours $\gamma(\zeta)$ ending at 
$\zeta$ for each term of the integrand chosen according to the multiplier 
$g(\xi,\zeta)$ caused by $W_{\nu}(\zeta)W_{\nu}(\xi)^{-1}$, which 
belongs to $G_0\cup G_+ \cup G_-$ up to $(1+O(|\zeta|^{-3/2}+|\xi|^{-3/2})),$
where
$$
G_0=\{1,\,\,  \xi^{\pm 1/2},\,\, \zeta^{\pm 1/2},\,\, (\xi\zeta)^{\pm 1/2},\,\,
(\xi/\zeta)^{\pm 1/2} \}, \quad
G_{\pm}=\{\rho \exp(\pm\tfrac 43(\zeta^{3/2}-\xi^{3/2}))\,|
\,\, \rho \in G_0  \}.
$$
The choice of $\gamma(\zeta)$ is described as follows: 
(i) for $g(\xi,\zeta) \in G_0$ let $\gamma(\zeta)$ be a segment joining $0$ to
$\zeta$; and
(ii) if, say, $\zeta\in \Sigma_1$, for $g(\xi,\zeta)
\in G_+$ (respectively, $\in G_-$) choose $\gamma(\zeta)$ to be a line
issuing from $\zeta$ and tending to $+\infty$ (respectively, $-\infty$).
Then it is easy to see that $\int_{\gamma(\zeta)}g(\xi,\zeta)d\xi \ll \zeta^2$. 
For $\nu=0,$ $\pm 1,$ from \eqref{3.c} 
we derive $\| X_{\nu}(\zeta)\| \ll \zeta^2\Delta(t,\zeta) \ll
|t^{-1}\zeta^3|+|t^{-4/3}\zeta^2| \ll t^{-\delta'}$ as long as 
$|\zeta| \ll t^{1/3-\delta'/3}$ with $0<\delta'<1,$ 
and obtain the solutions $P_{\nu}(\zeta)=I+X_{\nu}(\zeta)
=I+O(t^{-\delta'})$ of \eqref{3.b} in $\Sigma_{\nu} \cap \{|\zeta|\ll t^{1/3
-\delta'/3}\}$. It remains to show that
there exists $Q(t,\zeta)=I+O(t^{-\delta'})$ in $|\zeta|\ll t^{1/3-\delta'/3}$
such that $V=Q(t,\zeta)W$ reduces \eqref{3.17} to \eqref{3.18} by using
$P_{\nu}(\zeta)$'s.
This uniform reduction is constructed by a parallel argument as in
\cite[Section 6.5]{Wasow} in which 
$\epsilon^{-2/3}x,$ $\diag[1,\epsilon^{-1/3}]z$, $\Sigma_{j}$ and 
``\,\,$\sim 1$\,'' correspond to our $\zeta,$ $W$, $\Sigma_{j-1}$ and
``\,\,$=I+O(t^{-\delta'})$\,'', respectively. Thus, for 
given $W(\zeta)$, we obtain a solution as in the proposition.
\end{proof}
%%%%%%%%%%%%%%%%%%%%%%%%%%%%%%%%%%%%%%%
%%%%%%%%%% Remark 3.8 %%%%%%%%%%%%%%%%%%%
\begin{rem}\label{rem3.100}
(1) Let $E(t):=\{\zeta\,|\,\, |\zeta|\ll t^{1/3-\delta'/3},\,\,
|\exp(\pm \frac 23 \zeta^{3/2})| \ll 1\}$. For every $g(\xi,\zeta)\in G_0\cup 
G_+\cup G_-$, we also have $\int_0^{\zeta} g(\xi,\zeta)d\zeta\ll \zeta^2$ 
uniformly in $E(t)$. Thus we may easily derive a solution of Proposition 
\ref{prop3.7} in $E(t)$ without using the reasoning of \cite{Wasow}. 
Since $E(t)$ contains all Stokes curves $\zeta=re^{\frac13(2k-1)i},$ $k=0,\pm1$, 
$r>0,$ this solution restricted to $E(t)$ is enough for our use, in which
matchings are along Stokes curves.
\par
(2) In the proof above, without the change of variables $\eta=\xi\varphi(\xi)$,
we have $\Delta(t,\zeta) \ll (|t^{-2/3}\zeta^2| +|t^{-1}\zeta| +|t^{-4/3}|)$ 
caused by the $\eta^2$ term of $h_*(\eta)$, and the estimates in the
proposition are $|\zeta|\ll t^{1/6-\delta'/6}$ and $|\lambda-\lambda_{\iota}|
\ll t^{-1/2-\delta'/6}.$ 
%% Moreover system \eqref{3.e} may be written in the form
%% $$
%% \epsilon \frac{dV}{dx}=\Biggl(\begin{pmatrix} 0 & 1 \\ x  & 0 \end{pmatrix}
%%  + \epsilon \sum_{j\ge 0}B_j^{**} x^j \Biggr)V
%% $$
%% with $x=\epsilon^{2/3} z,$ $\epsilon=t^{-1}$. By the uniform reduction
%% theorem \cite[Theorem 6.5.-1]{Wasow} the domain is extended to $|\zeta|
%% <\varepsilon_0|t|^{2/3}$, $|\lambda-\lambda_{\iota}|< \varepsilon'_0,$
%% for sufficiently small $\varepsilon_0$ and $\varepsilon'_0.$
\end{rem}
%%%%%%%%%%%%%%%%%%%%%%%%%%%%%%%%%%
%%%%%%%%%%%%%%%%%%%%%%%%%%%%%%
%%%%%%%%%%%%%%%%%%%%%%%%%%%%%%%
%%%% Section 4 %%%%%
\section{Monodromy matrices}\label{sc4}
%%%%%%%%%%%%%%%%%%%%%%%%%%%%%%%%%%
%%%%%%%%%%%%%%%%%%%%%%%%%%%%%%%%%%
%%%%%%%%%%%%%%%%%%%%%%%%%%%%%%%%%%%%%%
We would like to find the monodromy matrices $M^0,$ $M^1$ with respect to the 
matrix solution \eqref{3.8}. Let $M^0_*$ and $M^1_*$ be the monodromy matrices
in the case where \eqref{3.8} solves \eqref{3.6} with $\hat{u}\equiv 1.$
Since the gauge transformation $Y=\hat{u}^{\sigma_3/2}Y_* $
reduces \eqref{3.6} to the system with $\hat{u}\equiv 1,$ the monodromy matrices
$M^0,$ $M^1$ in the case of general $\hat{u}$ are given by
$$
M^0=\hat{u}^{\sigma_3/2} M^0_* \hat{u}^{-\sigma_3/2}, \quad
M^1=\hat{u}^{\sigma_3/2} M^1_* \hat{u}^{-\sigma_3/2}. 
$$
In this section we calculate $M^0_*$ and $M^1_*$ by using connection matrices
and $S^*_1$, $S^*_2$ such that
$$
S_1=\hat{u}^{\sigma_3/2} S_1^* \hat{u}^{-\sigma_3/2}, \quad 
S_2=\hat{u}^{\sigma_3/2} S_2^* \hat{u}^{-\sigma_3/2}. 
$$
%%%%%%%%%%%%%%
%%%% ssc 4.1 %%%%
\subsection{Stokes graph}\label{ssc4.1}
%%%%%%%%%%%%%%%%%%%%%%%%%%%
Recall the Riemann surface $\mathbb{P}_+^{\infty}
\cup \mathbb{P}_-^{\infty}$ glued along the cuts $[-e^{i\phi}, \lambda_2]$ and 
$[\lambda_1, e^{i\phi}]$ with $\lambda_{2,\,1}=\mp e^{i\phi}A_{\phi}^{1/2}$ 
for $t=\infty.$
As in Figure \ref{stokes} (a), (b), symbols are assigned to 
the Stokes curves of the limit Stokes graph on $\mathbb{P}_+^{\infty}$. 
For $-\pi/2< \phi<0,$ let $\l^{\infty}_2,$ $\hat{\l}^{\infty}_1,$ $\l_0$ be 
Stokes curves on $\mathbb{P}_+^{\infty}$ 
connecting $\lambda_2$ to $i \infty,$ $\lambda_1$ to $-i \infty,$ 
$\lambda_2$ to $\lambda_1$, respectively,
and, for $0<\phi<\pi/2$, $\l^{\infty}_1,$ $\hat{\l}^{\infty}_2,$ $\l_0$ 
connecting
$\lambda_1$ to $i \infty,$ $\lambda_2$ to $-i\infty$, $\lambda_2$ to
$\lambda_1$, respectively. 
%%%%%%%%%%%%%%%%%%%%%%%%%%%%%%%%%%%%%%%%%%%%%%%%%%
%%%%%%%%%%%%%%%%%%%%%%%%%%%%
%%%%%%% Figure 4.1 %%%%%%%%%%%
%%%%%%%%%%%%%%%%%%%%%%%%%%%%
%%%%%%%%%%%%%%%%%%%%%%%%%%%%
{\small
\begin{figure}[htb]
\begin{center}
\unitlength=0.73mm
%%%%%%%%%%%%%%%%%%%%%%%%%%%%%%%%%%
%%%%%%%%%%%%%%%%%%%%%%%%%%%%%%%%%%
%%%%%%%%%%%%%%%%%%%%%%%%%%%%%%%%%%
\begin{picture}(60,65)(-30,-35)
\put(18,-7){\circle*{1.5}}
\put(6,-8){\circle*{1.5}}
\put(-6,8){\circle*{1.5}}
\put(-18,7){\circle*{1.5}}
\thicklines
\qbezier (6,-8) (4.5,-4) (0, 0)
\qbezier (6,-8) (1,-20) (1, -30)
\qbezier (6,-8) (13,-8) (18, -7)
\qbezier (-6,8) (-4.5,4) (0, 0)
\qbezier (-6,8) (-1,20) (-1, 30)
\qbezier (-6,8) (-13,8) (-18, 7)
\put(20,-5.5){\makebox{$e^{i\phi}$}}
\put(-32,3){\makebox{$-e^{i\phi}$}}
\put(-1.5,-9.0){\makebox{$\lambda_1$}}
\put(-11.5,11){\makebox{$\lambda_2$}}
\put(1,24){\makebox{$\mathbf{c}_2^{\infty}$}}
\put(3,-27){\makebox{$\hat{\mathbf{c}}_1^{\infty}$}}
\put(0,2){\makebox{$\mathbf{c}_0$}}

\put(-22,-38){\makebox{(a) $-\pi/2 <\phi <0$}}
\end{picture}
%%%%%%%%%%%%%%%%%%
%%%%%%%%%%%%%%%%%%
%%%%%%%%%%%%%%%%%%
\qquad\quad
%%%%%%%%%%%%%%%%%%%%%%%%%%%%%%%%%%
\begin{picture}(60,60)(-30,-35)
\put(-18,-7){\circle*{1.5}}
\put(-6,-8){\circle*{1.5}}
\put(6,8){\circle*{1.5}}
\put(18,7){\circle*{1.5}}
\thicklines
\qbezier (-6,-8) (-4.5,-4) (0, 0)
\qbezier (-6,-8) (-1,-20) (-1, -30)
\qbezier (-6,-8) (-13,-8) (-18, -7)
\qbezier (6,8) (4.5,4) (0, 0)
\qbezier (6,8) (1,20) (1, 30)
\qbezier (6,8) (13,8) (18, 7)
\put(-30,-7.5){\makebox{$-e^{i\phi}$}}
\put(20,4.5){\makebox{$e^{i\phi}$}}
\put(-4,-10.0){\makebox{$\lambda_2$}}
\put(6,11){\makebox{$\lambda_1$}}
\put(-7,24){\makebox{$\mathbf{c}_1^{\infty}$}}
\put(1,-27){\makebox{$\hat{\mathbf{c}}_2^{\infty}$}}
\put(-5,2){\makebox{$\mathbf{c}_0$}}

\put(-22,-38){\makebox{(b) $0 <\phi <\pi/2$}}
\end{picture}
%%%%%%%%%%%%%%%%%%
%%%%%%%%%%%%%%%%%%%%%%%%%%%%%%%%
%%%%%%%%%%%%%%%%%%%%%%%%%%%%%%%%%%
\end{center}
\caption{Limit Stokes graph}
\label{stokes}
\end{figure}
}
%%%%%%%%%%%%%%%%%%%%%%%%%%%%%%%%%%%%%%%%%%%%%%%%%%%%%%%%%%%
%%%%%%%%%%%%%%%%%%%%%%%%%%%%%%%%%%%%%%%%%%%%%%%%%%%%%%%%%%%
%%%%%%%%%%%%%%%%%%%%%%%%%%%%%%%%%%%%%%%%%%%%%%%%%%%5
%%%%%%%%%%%%%%%%%%%%%%%%%%%
%%%% ssc 4.2 %%%%%%%%%%%%
\subsection{Connection matrices}\label{ssc4.2}
%%%%%%%%%%%%%%%%%%%%%%%%%%%%
Recall the canonical solutions $Y(t,\lambda)=Y_2(t,\lambda)$ and 
$Y_3(t,\lambda)$ in the sectors $|\arg\lambda-\pi/2|<\pi$ and
$|\arg\lambda-3\pi/2|<\pi$, respectively, near $\lambda=\infty$.
For \eqref{3.6} with $\hat{u}\equiv 1$, the Stokes matrix $S^*_2$ is given by
$Y_3=YS^*_2$ (cf. \eqref{1.3}). Set $\mathcal{D}^0_{-}=\mathbb{C}\setminus
[e^{i\phi},+\infty]$ with $0<\arg\lambda <2\pi.$ Let us consider the case, say,
$0<\phi<\pi/2.$ 
For the Stokes curve $\gamma_2:=(-\hat{\mathbf{c}}^{\infty}_2)\cup \mathbf{c}_0
\cup {\mathbf{c}}^{\infty}_1$ on $\mathbb{P}_+^{\infty}$,
the projection $\mathrm{pr}(\gamma_2)$ to $\mathcal{D}^0_-$ is 
a path issuing from $e^{3\pi i/2}\infty$ and ending in $e^{\pi i/2}\infty$ as 
in Figure \ref{comploops} (a). In what follows let us identify
$\mathrm{pr}(\gamma_2)$ with $\gamma_2$ and denote by the same symbol.  
Note that $Y=Y_2$ (respectively, $Y_3$) is analytic in the 
domain $\mathcal{D}_-(\pi/2):=\{|\arg\lambda-\pi/2|<\pi/4,\,\,|\lambda|>2025\}
\subset \mathcal{D}^0_-$ (respectively, $\mathcal{D}_-(3\pi/2):
=\{|\arg\lambda-3\pi/2|<\pi/4,\,\,|\lambda|>2025\}\subset\mathcal{D}^0_-$),
which is simply connected. 
The loop $\hat{l}_0$ of Figure \ref{loops0} with $\hat{p}_{\mathrm{st}}
\in \mathcal{D}_-(\pi/2)$ is decomposed into $\hat{l}_0=\gamma_2\circ \gamma_0$,
where $\gamma_0$ is an arc lying in $\{|\lambda|>2025\}\cap\mathcal{D}^0_-$ 
issuing from $\mathcal{D}_-(\pi/2)$ and ending in $\mathcal{D}_-(3\pi/2)$ 
as in Figure \ref{comploops} (a).
%%%%%%%%%%%%%%%%%%%%%%%%%%%%%%%%%%%%%%
%%%%%%%%%%%%%%%%%%%%%%%%%%%%%%%%%%%%
%%%%%%%%% Figure 4.2 %%%%%%%%%%%%
%%%%%%%%%%%%%%%%%%%%%%%%%%%%%%
%%%%%%%%%%%%%%%%%%%%%%%%%%%%%%%
{\small
\begin{figure}[htb]
\begin{center}
\unitlength=0.68mm
%%%%%%%%%%%%%%%%%%%%%%%%%%%%%%%%%%
%%%%%%%%%%%%%%%%%%%%%%%%%%%%%%%%%%
%%%%%%%%%%%%%%%%%%%%%%%%%%%%%%%%%%
\begin{picture}(60,85)(-30,-46)
\put(-18,-7){\circle*{2}}
\put(-6,-8){\circle*{1.5}}
\put(6,8){\circle*{1.5}}
\put(18,7){\circle*{2}}
\put(-41.2,-6){\vector (1,-1){0}}
\put(5.7,-0.2){\vector (2,3){0}}
\qbezier (18,8) (26,8) (38,8)
\qbezier (18,6) (26,6) (38,6)
\qbezier (-40.9,6) (-44.4,0) (-41.2,-6)
\qbezier (-1.3,-7.4) (3.1,-1.8) (5.7,-0.2)
\qbezier (16,20) (26,30) (26,30)
\qbezier (-16,20) (-26,30) (-26,30)
\qbezier (-16,20) (0,30) (16,20)
\qbezier (16,-20) (26,-30) (26,-30)
\qbezier (-16,-20) (-26,-30) (-26,-30)
\qbezier (-16,-20) (0,-30) (16,-20)
\thicklines
\qbezier (-7,30)  (-70,0)     (-9,-30)
\qbezier (-6,-8) (-4.5,-4) (0, 0)
\qbezier (-6,-8) (-1,-20) (-1, -30)
\qbezier[6] (-6,-8) (-13,-8) (-18, -7)
\qbezier (6,8) (4.5,4) (0, 0)
\qbezier (6,8) (1,20) (1, 30)
\qbezier[6] (6,8) (13,8) (18, 7)
\put(-27,-4.5){\makebox{$-e^{i\phi}$}}
\put(20,0){\makebox{$e^{i\phi}$}}
\put(-12,-13.5){\makebox{$\lambda_2$}}
\put(7,11){\makebox{$\lambda_1$}}
 \put(-5,33){\makebox{$Y_2$}}
 \put(16,33){\makebox{$\mathcal{D}_-(\pi/2)$}}
 \put(-7,-36){\makebox{$Y_3$}}
 \put(-43,-36){\makebox{$\mathcal{D}_-(3\pi/2)$}}
 \put(0,-11){\makebox{$\gamma_2$}}
 \put(-44,10){\makebox{$\gamma_0$}}
\put(-22,-50){\makebox{(a) $\mathcal{D}^0_-
=\mathbb{C}\setminus [e^{i\phi},+\infty]$}}
\end{picture}
%%%%%%%%%%%%%%%%%%
%%%%%%%%%%%%%%%%%%%%%%%%%%%%%%%%%%
\qquad\qquad\qquad
%%%%%%%%%%%%%%%%%%%%%%%%%%%%%%%%%%
\begin{picture}(60,67)(-30,-46)
\put(-18,-7){\circle*{2}}
\put(-6,-8){\circle*{1.5}}
\put(6,8){\circle*{1.5}}
\put(18,7){\circle*{2}}
\put(-5.7,0.2){\vector (-2,-3){0}}
\put(41.2,6){\vector (-1,1){0}}
\qbezier (-18,-8) (-28,-8) (-38,-8)
\qbezier (-18,-6) (-28,-6) (-38,-6)
\qbezier (40.9,-6) (44.4,0) (41.2,6)
\qbezier (1.3,7.4) (-3.3,2.0) (-5.7,0.2)
\qbezier (16,20) (26,30) (26,30)
\qbezier (-16,20) (-26,30) (-26,30)
\qbezier (-16,20) (0,30) (16,20)
\qbezier (16,-20) (26,-30) (26,-30)
\qbezier (-16,-20) (-26,-30) (-26,-30)
\qbezier (-16,-20) (0,-30) (16,-20)
\thicklines
\qbezier (7,-30)  (70,0)     (9,30)
\qbezier (-6,-8) (-4.5,-4) (0, 0)
\qbezier (-6,-8) (-1,-20) (-1, -30)
\qbezier[6] (-6,-8) (-13,-8) (-18, -7)
\qbezier (6,8) (4.5,4) (0, 0)
\qbezier (6,8) (1,20) (1, 30)
\qbezier[6] (6,8) (13,8) (18, 7)
\put(-27,-2.5){\makebox{$-e^{i\phi}$}}
\put(21,4.5){\makebox{$e^{i\phi}$}}
\put(-3,-10.0){\makebox{$\lambda_2$}}
\put(7,11){\makebox{$\lambda_1$}}
 \put(2,33){\makebox{$Y_2$}}
 \put(-37,33){\makebox{$\mathcal{D}_+(\pi/2)$}}
 \put(1,-36){\makebox{$Y_1$}}
 \put(16,-36){\makebox{$\mathcal{D}_+(-\pi/2)$}}
 \put(-3,10){\makebox{$\gamma_1$}}
 \put(40,-11){\makebox{$\gamma_0$}}
\put(-22,-50){\makebox{(b) $\mathcal{D}^0_+=\mathbb{C}\setminus [-\infty,-e^{i\phi}]$}}
\end{picture}
%%%%%%%%%%%%%%%%%%
%%%%%%%%%%%%%%%%%%%%%%%%%%%%%%%%
%%%%%%%%%%%%%%%%%%%%%%%%%%%%%%%%%%
%%%%%%%%%%%%%%%%%%%%%%%%%%%%%%%%%%
\end{center}
\caption{Composite loops in the case $0<\phi<\pi/2$} 
\label{comploops}
\end{figure}
}
%%%%%%%%%%%%%%%%%%%%%%%%%%%%%%%%%
%%%%%%%%%%%%%%%%%%%%%%%%%%%%%%%%%%%
The analytic continuation of $Y_3$ along $\gamma_2$ results in
$Y_3^{\gamma_2}\Gamma^{\infty}_{\infty\,2}=Y$ in $\mathcal{D}_-(\pi/2),$
where $\Gamma^{\infty}_{\infty\,2}\in SL_2(\mathbb{C})$ is a connection matrix.
On the other hand the relation $YS^*_2=Y_3$ is also written in the form
$Y^{\gamma_0}S^*_2=Y_3$ in $\mathcal{D}_-(3\pi/2)$, where $Y^{\gamma_0}$ 
denotes the analytic continuation of $Y$ along $\gamma_0.$ Then the analytic 
continuation of both sides along $\gamma_2$ yields $Y^{\hat{l}_0} S^*_2
=Y^{\gamma_2\circ\gamma_0} S^*_2=Y_3^{\gamma_2} =Y 
(\Gamma^{\infty}_{\infty\,2})^{-1}$, i.e., $Y^{\hat{l}_0}=Y
(\Gamma^{\infty}_{\infty\,2})^{-1}(S^*_2)^{-1}=YM^0_*$ in 
$\mathcal{D}_-(\pi/2),$ which implies
%%%% (4.a) %%%%%%%%%%%%%%
\begin{equation}\label{4.a}
\Gamma^{\infty}_{\infty\,2} M^0_*=(S_2^*)^{-1}.
\end{equation}
In the domain $\mathcal{D}^0_+=\mathbb{C}\setminus[-\infty,-e^{i\phi}]$ with
$|\arg\lambda|<\pi$, for the Stokes curve 
$\gamma_1=(-\mathbf{c}^{\infty}_1)\cup(-\mathbf{c}_0)\cup 
\hat{\mathbf{c}}^{\infty}_2$ on $\mathbb{P}_+^{\infty}$, 
considering the projection to $\mathcal{D}^0_+$ 
issuing from $e^{\pi i/2}\infty$ and ending in $e^{-\pi i/2}\infty,$
we analogously obtain
%%%% (4.b) %%%%%%%%%%%%%%
\begin{equation}\label{4.b}
 M^1_* \Gamma^{\infty}_{\infty\,1} =(S_1^*)^{-1},
\end{equation}
where $\Gamma^{\infty}_{\infty\,1}$ is a connection matrix such that
$Y^{\gamma_1}\Gamma^{\infty}_{\infty\,1}=Y_1$ in $\mathcal{D}_+(-\pi/2)
:=\{|\arg\lambda+\pi/2|<\pi/4,\,\,|\lambda|>2025\}\subset\mathcal{D}^0_+$. 
Our concern is finding
the connection matrices $\Gamma^{\infty}_{\infty\,1}$ and
$\Gamma^{\infty}_{\infty\,2}$. 
%%%%%%%%%%%%%%%%%%%%%%%%%%%%%%%%%%%%%%
%%%%%%%%%%%%%%%%%%%%%%%%%%%%%%%%%%%%
\par
Let us calculate $\Gamma^{\infty}_{\infty 2}$ in the case $0<\phi <\pi/2.$ 
To discuss according to the justification scheme of Kitaev \cite{Kitaev-1}, 
suppose that $a_{\phi}=a_{\phi}(t)$ is given by \eqref{3.12} with 
$(y,y^*)=(y(t), y^*(t))$ not necessarily solving (P$_{\mathrm{V}}$), and that
%%%%% (4.1) %%%%%%%%%%%
\begin{equation}\label{4.1}
a_{\phi}(t)=A_{\phi} + t^{-1}B_{\phi}(t), \quad  B_{\phi}(t)=O(1)
\end{equation}
%%%%%%%%%%%%%%%%%%%%%%%
for $t\in S_*(\phi,t'_{\infty},\kappa_0,\delta_1)$ with given $\kappa_0,$
given small $\delta_1$ and sufficiently large $t'_{\infty}$.
Here $A_{\phi}$ is a unique solution of the Boutroux equations \eqref{2.1},
and  
$$ 
S_*(\phi,t'_{\infty},\kappa_0,\delta_1)=\{t\,| \,\, \re t>t'_{\infty}, 
|\im t|<\kappa_0, |y^*(t)|+|y(t)|+|y(t)|^{-1}+|y(t)-1|^{-1}<\delta_1^{-1}\}.
$$
Let the limit Stokes graph be as in Figure \ref{stokes} (b).  
In what follows we use the following notation:  
\par
(1) for the WKB-solution of Proposition \ref{prop3.6}, write $\Lambda(\tau)$ 
in the component-wise form $\Lambda(\tau) =\Lambda_3(\tau) + \Lambda_I(\tau)$ 
with $\Lambda_3(\tau) \in \mathbb{C}\sigma_3,$ $\Lambda_I(\tau)\in \mathbb{C}I;$
\par
(2) $c_0=(c_1-ic_2)/c_3$ and $d_0=(d_1-i d_2)/d_3$, where $c_k=b_k(\lambda_2),$ 
$d_k=b_k(\lambda_1)$ for $k=1,2,3$. 
%%% \par
%%% (3) for large $t$, $\lambda_1$, $\lambda_2$, $\mathbf{c}^{\infty}_1,$ 
%%% $\hat{\mathbf{c}}_2^{\infty},$ $\mathbf{c}_0$ also denote turning points and
%%% curves approaching the limiting ones in Figure \ref{stokes} (ii).
\par
%%%%%%%%%%%%%%%%%%%%%%%%%%%%%%%%%%%%%%
In Propositions \ref{prop3.6} and \ref{prop3.7}, set $\delta=\delta'=1/4-
\varepsilon$, $0<\varepsilon<1/4$ being arbitrary. Then both propositions
apply to the annulus
$$
\mathcal{A}_{\varepsilon}: \quad |t|^{-\tfrac 23+\tfrac 23
\left(\tfrac14-\varepsilon\right)}
\ll |\lambda-\lambda_{\iota}|\ll |t|^{-\tfrac 5{12}-\tfrac13
\left(\tfrac14-\varepsilon\right)} 
\ll |t|^{-\tfrac 13-\tfrac 13\left(\tfrac 14-\varepsilon\right)} \quad
(\iota=1,2),
$$
in the choice of which evaluation of $\eta^{1/2}$ in, say, {\bf (b)} below is
also taken into account. 
%%%%%%%%%%%%% Remark 4.1 %%%%%%%%%%%%%%%%%%%%%%%%%%
\begin{rem}\label{rem4.1}
The annulus $|t|^{-14/27 -(2/3)\varepsilon'} \ll 
|\lambda-\lambda_{\iota}|\ll |t|^{-14/27+(1/3)\varepsilon'}$ of the early 
version, which is contained in $\mathcal{A}_{\varepsilon}$ with
$\varepsilon=1/36+\varepsilon'$ for
$\delta=\delta'=2/9-\varepsilon'=1/4-\varepsilon$, 
is also available in our calculation below, 
though the derivation of $-\frac{14}{27}-\frac 23\varepsilon'$,
$-\frac{14}{27}+\frac 13\varepsilon'$
is based on an inaccurate argument. 
\end{rem}
%%%%%%%%%%%%%%%%%%%%%%%%%%%%%%%%%%%
In what follows we set $\delta=1/4-\varepsilon.$
\par
Consider the Stokes graph on the
plane $\mathbb{C}\setminus [e^{i\phi}, +\infty]$ containing the sector
$0<\arg \lambda <2\pi,$ $|\lambda|>2025.$ 
The connection matrix along 
$\gamma_2=(- \hat{\mathbf{c}}^{\infty}_2)\cup  \mathbf{c}_0 \cup
\mathbf{c}^{\infty}_1$ 
consists of the following matrices 
$\Gamma_{\mathrm{a}}$, $\ldots$, $\Gamma_{\mathrm{i}}$.
\par
%%%%%% (a) %%%%%%%%%%%%%%%%%%%%%
{\bf (a)} Let $\Psi_{\infty}(\lambda)$ be the WKB solution along 
$\mathbf{c}^{\infty}_1$ with a base point 
$\tilde{\lambda}_1 \in \mathbf{c}^{\infty}_1,$ $|\tilde{\lambda}_1
-\lambda_1|\asymp t^{-1}$ near $\lambda_1$, 
and set $Y(\lambda)=\Psi_{\infty}(\lambda) \Gamma_{\mathrm{a}}$.
Then we have  
\begin{align*}
 \Gamma_{\mathrm{a}}
% =\Psi_{\infty}(\lambda)^{-1}Y(\lambda)
% \\
&= \exp\Bigl(-\int^{\lambda}
_{\tilde{\lambda}_1} \Lambda(\tau)d\tau \Bigr) T^{-1}(I+O(t^{-\delta}
+|\lambda|^{-1})) \exp\bigl(\tfrac 14 (t\lambda-2\theta_{\infty} \log \lambda)
\sigma_3\bigr)
\\
&=  C_3(\tilde{\lambda}_1) c_I(\tilde{\lambda}_1)(I+O(t^{-\delta}) )
%\\
%&\phantom{--}\times
\exp\biggl( -\lim_{\substack{ 
\lambda \to \infty \\[0.05cm] \lambda \in \mathbf{c}^{\infty}_1 }} 
\Bigl( \int^{\lambda}_{\lambda_1} \Lambda_3(\tau) d\tau - \frac 14 (t\lambda
-2\theta_{\infty}\log\lambda )\sigma_3 \Bigr)\biggr)
\end{align*}
with $ C_3(\tilde{\lambda}_1) =\exp (\int^{\tilde{\lambda}_1}_{\lambda_1}
\Lambda_3(\tau) d\tau ),$  
$c_I(\tilde{\lambda}_1) =\exp (-\int_{\tilde{\lambda}_1}^{\infty}
\Lambda_I(\tau) d\tau ).$
%%%%%%%%%%%%%%%%%%%%%%%%%%%%%%%%%
\par
%%%%%% (b) %%%%%%%%%%%%%%%%%%%%%%%
{\bf (b)} 
For $\Psi_{\infty}(\lambda)$ and $\Phi_1(\lambda)$ (cf. Proposition
\ref{prop3.7}) in the annulus $\mathcal{A}_{\varepsilon}$
set $\Psi_{\infty}(\lambda)=\Phi_1(\lambda)\Gamma_{\mathrm{b}}$
along $\l^{\infty}_1.$ We may suppose that the curve 
$(2\kappa)^{1/3} (\lambda-\tilde{\lambda}_1) =t^{-2/3}(\zeta+O(t^{-1/3}))$ 
with $\lambda \in \l^{\infty}_1$ enters the sector $\Sigma_1:$ 
$|\arg \zeta -\pi/3| <2\pi/3,$ and that $\Sigma_1$ does not intersect the
cut $[e^{i\phi},\lambda_1].$ 
Write $K^{-1}=2(2\kappa)^{-1/3}(d_1-id_2),$ where $d_k=b_k(\lambda_1)$ as
in {\bf (2)}. Then, 
by Propositions \ref{prop3.6} and
\ref{prop3.7},
\begin{align*}
\Gamma_{\mathbf{b}} &= \Phi_1(\lambda)^{-1} \Psi_{\infty}(\lambda)
\\
&= W(\zeta)^{-1} \begin{pmatrix} 1 & 0 \\ 0 & Kt^{-1/3} \end{pmatrix}^{\!\! -1}
(I+O(t^{-\delta}))
\begin{pmatrix}  1 & -{d_3}/({d_1+id_2}) \\ -{d_3}/({d_1-id_2}) & 1
\end{pmatrix}^{\!\! -1}
\\
& \phantom{---} \times
\begin{pmatrix}  1 & ({b_3-\mu})/({b_1+ib_2}) \\ 
({\mu-b_3})/({b_1-ib_2}) & 1
\end{pmatrix} (I+O(t^{-\delta}))
\exp\Bigl( \int^{\lambda}_{\tilde{\lambda}_1} \Lambda(\tau) d\tau \Bigr)
\\
&= W(\zeta)^{-1} 
\begin{pmatrix} 1 & {d_3}/({d_1+id_2}) \\ {\mu t^{1/3}}/({2K(d_1-id_2)})
& {\mu t^{1/3}}/({2Kd_3}) \end{pmatrix} (I+O(t^{-\delta}))
\exp\Bigl( \int^{\lambda}_{\tilde{\lambda}_1} \Lambda(\tau) d\tau \Bigr)
\end{align*}
for $\lambda \in \mathcal{A}_{\varepsilon} \cap \l^{\infty}_2$, 
where $(\mu-b_3)/(b_1\pm ib_2)=(\mu-d_3)/(d_1\pm id_2)+O(\eta),$ 
$\eta=\lambda-\tilde{\lambda}_1.$ Using
\begin{align*}
\mu &=(b_1^2+b_2^2+b_3^2)^{1/2}=(2(d_1d_1'+d_2d_2'+d_3d_3')(\eta+O(\eta^2)))
^{1/2} \\
&=(2\kappa)^{1/2}\eta^{1/2}(1+O(\eta))
=(2\kappa)^{1/3} t^{-1/3} \zeta^{1/2}(1+O(\eta))
\\
&=2K (d_1-id_2)t^{-1/3}\zeta^{1/2}(1+O(\eta)),
\end{align*}
we have
$$
\Gamma_{\mathrm{b}} = \exp\Bigl(\int^{\lambda}_{\tilde{\lambda}_1} \Lambda(\tau)
d\tau -\frac 23 \zeta^{3/2}\sigma_3 \Bigr) \zeta^{1/4} (I+O(t^{-\delta}))
\begin{pmatrix} 1 & 0 \\ 0 & -d_0 \end{pmatrix}.
$$
By Remark \ref{rem3.4}, $\Lambda_3(\lambda)= [(2\kappa)^{1/2} t\eta^{1/2}
(1+O(\eta))+O(\eta^{-1/2}) ] \sigma_3$ and 
$\Lambda_I(\lambda)=[-(\log\eta)_{\eta}/4+O(\eta^{-1/2})]I$
for $\eta=\lambda-\tilde{\lambda}_1$, $\lambda \in \mathcal{A}_{\varepsilon}
\cap \l^{\infty}_1 .$ Hence
$$
\Gamma_{\mathrm{b}}= (I+O(t^{-\delta})) \exp\Bigl( -\int^{\tilde{\lambda}_1}
_{\lambda_1} \Lambda_3(\tau)d\tau +O(\eta^{1/2})
\Bigr) (\tilde{\zeta}_1)^{1/4} \begin{pmatrix} 1 & 0 \\ 0 & -d_0
\end{pmatrix}
$$
with suitably chosen $\tilde{\zeta}_1 \asymp \tilde{\eta}_1
= \tilde{\lambda}_1-\lambda_1.$ Since $\eta^{1/2} \ll t^{-1/4+\varepsilon /6}
\ll t^{-\delta}$ in $\mathcal{A}_{\varepsilon},$
$$
\Gamma_{\mathrm{b}}=(\tilde{\zeta}_1)^{1/4} (I+O(t^{-\delta})) C_3(\tilde
{\lambda}_1)^{-1}   \begin{pmatrix} 1 & 0 \\ 0 & -d_0
\end{pmatrix}.
$$
%%%%%%%%%%%%%%%%%%%%%%%%%%%%%%%%%
\par
%%%%%%%% (c) %%%%%%%%%%%%%%%%%%%%%%%%%
{\bf (c)} Let $\Phi_{1-}(\lambda)$ be the solution by Proposition \ref{prop3.7} 
near $\mathbf{c}^{}_0$ in an annulus 
and set $\Phi_1(\lambda)=\Phi_{1-}(\lambda)
G_{\mathrm{c}}.$ Then, by Proposition \ref{prop3.7},
\begin{equation*}
G_{\mathrm{c}}= \Phi_{1-}(\lambda)^{-1}\Phi_1(\lambda) =(\Phi_1(\lambda)
G_1)^{-1}\Phi_1(\lambda)=G_1^{-1}=\begin{pmatrix} 1 & i \\ 0 & 1 \end{pmatrix}.
\end{equation*}
%%%%%%% (d) %%%%%%%%%
\par
{\bf (d)} Let $\Psi_{\infty 1}(\lambda)$ be the WKB solution along 
$\mathbf{c}^{}_0$ with the base point $\tilde{\lambda}'_1 \in  
\mathbf{c}_0$ near $\lambda_1$, 
and set $\Phi_{1-}(\lambda)=\Psi_{\infty 1}(\lambda) \Gamma_{\mathrm{d}}$.
By the same argument as in step {\bf (b)}, we have
\begin{equation*}
\Gamma_{\mathrm{d}}=(\tilde{\zeta}'_1)^{-1/4}(I+O(t^{-\delta}))\tilde{C}_3(\tilde
{\lambda}'_1)   \begin{pmatrix} 1 & 0 \\ 0 & -d_0^{-1} \end{pmatrix}   
\end{equation*}
with $\tilde{\zeta}'_1 \asymp  \tilde{\lambda}'_1-\lambda_1,$
$\tilde{C}_3(\tilde{\lambda}
_1')=\exp(\int^{\tilde{\lambda}'_1}_{\lambda_1} \Lambda_3(\tau)d\tau ).$ 
%%%%%%%% (e) %%%%%%%%%%%%%%%
\par
{\bf (e)} Let $\Psi_{\infty 2}(\lambda)$ be the WKB solution along 
$\mathbf{c}^{}_0$ with the base point 
$\lambda_2' \in \mathbf{c}_0$ near $\lambda_2$,  
and set $\Psi_{\infty 1}(\lambda)=\Psi_{\infty 2}(\lambda) \Gamma_{\mathrm{e}}$.
Then  
$$
\Gamma_{\mathrm{e}}= (I+O(t^{-\delta}))\tilde{C}_3(\tilde{\lambda}'_1)^{-1} 
\tilde{C}'_3(\lambda_2')\tilde{c}_I(\tilde{\lambda}'_1, \lambda_2')
\exp\Bigl(-\int_{\lambda_2}^{\lambda_1} \Lambda_3(\tau) d\tau \Bigr)
$$
with $\tilde{c}_I(\tilde{\lambda}'_1,\lambda_2')=\exp(\int^{\lambda_2'}_{\tilde
{\lambda}'_1} \Lambda_I(\tau)d\tau ),$ $\tilde{C}'_3(\lambda_2')=\exp
(\int_{\lambda_2}^{\lambda'_2}\Lambda_3(\tau) d\tau).$
\par
%%%%%%%%%%% (f) %%%%%%%%%%%%%%%%%%%
{\bf (f)} Let $\Phi_{ 2}(\lambda)$ be the solution by Proposition \ref{prop3.7} 
along $\mathbf{c}^{}_0$ near $\lambda_2$,
and set $\Psi_{\infty 2}(\lambda)=\Phi_{ 2}(\lambda) \Gamma_{\mathrm{f}}$.
Then
\begin{equation*}
\Gamma_{\mathrm{f}}=({\zeta}'_2)^{1/4} (I+O(t^{-\delta})) 
\tilde{C}'_3(\tilde{\lambda}'_2)^{-1} \begin{pmatrix} 1 & 0 \\ 0 & -c_0
\end{pmatrix}, 
\end{equation*}
where ${\zeta}'_2 \asymp  {\lambda}'_2-\lambda_2 $ and $c_k=b_k(\lambda_2).$ 
\par
%%%%%%%%%%%%%%%%%% (g) %%%%%%%%%%%%%%%%%%%%
{\bf (g)} Let $\Phi_{2-}(\lambda)$ be the solution by Proposition \ref{prop3.7}
near $\hat{\mathbf{c}}_2^{\infty}$ in the 
annulus $\mathcal{A}_{\varepsilon}$ around $\lambda_2$, and set 
$\Phi_2(\lambda)=\Phi_{2-}(\lambda)G_{\mathrm{g}}.$
Then 
$$
G_{\mathrm{g}}= \Phi_{2-}(\lambda)^{-1}\Phi_2(\lambda) =(\Phi_2(\lambda)G_0^{-1})^{-1}\Phi_2(\lambda)=G_0
=\begin{pmatrix} 1 & 0 \\ -i & 1  \end{pmatrix}.
$$
%%%%%%%%%%%%%%%%% (h) %%%%%%%%%%%%%%%%%%%%%
\par
{\bf (h)} For the WKB-solution $\hat{\Psi}_{\infty}(\lambda)$ along $\hat{\l}
^{\infty}_2$ set $\Phi_{2-}(\lambda) = \hat{\Psi}_{\infty}(\lambda) 
{\Gamma}_{\mathrm{h}}$ for $\lambda \in \mathcal{A}_{\varepsilon} \cap
\hat{\l}^{\infty}_2.$ 
Then
$$
{\Gamma}_{\mathrm{h}} = (\tilde{\zeta}'_2)^{-1/4} (I+O(t^{-\delta})) 
\hat{C}_3(\tilde{\lambda}'_2) \begin{pmatrix} 1 &  0 \\ 0 & -c_0^{-1} 
\end{pmatrix}
$$
with $\tilde{\zeta}'_2 \asymp \tilde{\lambda}'_2-\lambda_2$, 
$ \hat{C}_3(\tilde{\lambda}'_2) =\exp\Bigl
(\int^{\tilde{\lambda}'_2}_{\lambda_2}
\Lambda_3(\tau) d\tau \Bigr).$
%%%%%%% (i) %%%%%%%%%%%%%%%%%
\par
{\bf (i)} Set $\hat{\Psi}_{\infty}(\lambda)=Y_3(t,\lambda){\Gamma}_{\mathrm{i}}.$
Then
\begin{align*}
\hat{\Gamma}_{\infty}&= \hat{C}_3(\tilde{\lambda}'_2)^{-1} 
\hat{c}_I(\tilde{\lambda}'_2)(I+O(t^{-\delta}) )
\\
\notag
& \phantom{--}\times  \exp\biggl( \lim_{\substack{ 
\lambda \to \infty \\[0.05cm] \lambda \in \hat{\l}^{\infty}_2 }} 
\Bigl( \int^{\lambda}_{\lambda_2} \Lambda_3(\tau) d\tau - \frac 14 (t\lambda
-2\theta_{\infty}\log\lambda )\sigma_3 \Bigr)\biggr)
\end{align*}
with
$\hat{c}_I(\tilde{\lambda}'_2) =\exp \Bigl(\int_{\tilde{\lambda}'_2}^{\infty}
\Lambda_I(\tau) d\tau \Bigr). $
\par
%%%%%%%%%%%%%%%%%%%%%%%%%%%%%%%%%%%%%%%%%%%%%%%%%%%%%
Then, collecting the matrices above and using a symmetric property 
about the scalar part, we have
\begin{align*}
\Gamma^{\infty}_{\infty 2}& = {\Gamma}_{\mathrm{i}}{\Gamma}_{\mathrm{h}}
G_{\mathrm{g}} \Gamma_{\mathrm{f}} \Gamma_{\mathrm{e}}\Gamma_{\mathrm{d}} 
G_{\mathrm{c}}
\Gamma_{\mathrm{b}} \Gamma_{\mathrm{a}}
\\
&=(I+O(t^{-\delta}))\exp(\hat{J}_2\sigma_3) 
   \begin{pmatrix} 1 & 0 \\ 0 & -c_0^{-1}  \end{pmatrix}
   \begin{pmatrix} 1 & 0 \\ -i & 1  \end{pmatrix}
   \begin{pmatrix} 1 & 0 \\ 0 & -c_0  \end{pmatrix}
\\
& \phantom{--}\times 
  \exp( -J_0\sigma_3) 
   \begin{pmatrix} 1 & 0 \\ 0 & -d_0^{-1}  \end{pmatrix}
   \begin{pmatrix} 1 & i \\ 0 & 1  \end{pmatrix}
   \begin{pmatrix} 1 & 0 \\ 0 & -d_0  \end{pmatrix}
 \exp(-J_1 \sigma_3)
\\
&= (I+O(t^{-\delta}))
\begin{pmatrix} e^{\hat{J}_2 -J_0 -J_1} & -i d_0 e^{\hat{J}_2-J_0+J_1}  \\
i c^{-1}_0 e^{-\hat{J}_2-J_0-J_1 }   &
   e^{J_1-\hat{J}_2}(e^{J_0} + c^{-1}_0d_0 e^{-J_0} ) 
  \end{pmatrix}.
\end{align*}
Here 
%%%%%%%%%%%%%%%%%%%%%%%%%%%%%%%%%%%%
\begin{align*}
J_{j}=& \lim_{\substack{\lambda \to \infty \\ \lambda \in 
\mathbf{c}^{\infty}_j}}
\Bigl( \int^{\lambda}_{\lambda_{j}} \Lambda_3(\tau) d\tau -\frac 14(t\lambda -2
\theta_{\infty} \log \lambda) \sigma_3 \Bigr)\quad (j=1,2),
\\
\hat{J}_{2}=& \lim_{\substack{\lambda \to \infty \\ \lambda\in\hat{\mathbf{c}}
^{\infty}_2}}
\Bigl( \int^{\lambda}_{\lambda_2} \Lambda_3(\tau) d\tau -\frac 14(t\lambda -2
\theta_{\infty} \log \lambda) \sigma_3 \Bigr),
\quad
J_0= J_2-J_1 =\int^{\lambda_1}_{\lambda_2}\Lambda_3(\tau) d\tau.
\end{align*}
%%%%%%%%%%%%%%%%%%%%%%%%%%%%%%%%%%%%
\par
The matrix $\Gamma^{\infty}_{\infty 1}$ is calculated by using the same Stokes
graph of Figure \ref{stokes} (b) on the plane 
$\mathcal{D}_+^0=\mathbb{C}\setminus [-\infty, -e^{i\phi}]$
containing the sector $-\pi <\arg\lambda <\pi,$ $|\lambda|>2025.$ Note that,
in this case, the curve $\hat{\mathbf{c}}^{\infty}_2\subset \mathcal{D}_+^0$ 
tends to $e^{-\pi i/2}\infty$, and denote it by 
$\hat{\mathbf{c}}^{\infty}_{2*}.$ The calculation 
of $\Gamma^{\infty}_{\infty 1}$ begins with setting 
$Y_1(\lambda)=\hat{\Psi}^*_{\infty}
(\lambda)\Gamma^*_{\mathrm{a}},$ where $\hat{\Psi}^*_{\infty}(\lambda)$ is the
WKB solution along $\hat{\mathbf{c}}^{\infty}_{2*}$ tending to $e^{-\pi i/2}
\infty$. Note that $\hat{\Psi}^*_{\infty}(\lambda)$ has the same asymptotic 
form as of $\hat{\Psi}_{\infty}(\lambda)$. Repeating step by step matchings 
along $\gamma_1= (-\mathbf{c}_1^{\infty})\cup \mathbf{c}_0 
\cup \hat{\mathbf{c}}^{\infty}_{2*}$, we have
\begin{align*}
\Gamma^{\infty}_{\infty 1} 
&=(I+O(t^{-\delta}))\exp({J}_1\sigma_3) 
\begin{pmatrix} 1 & 0 \\ 0 & -d_0^{-1}\end{pmatrix} 
\begin{pmatrix}  1 &  -i \\ 0 & 1  \end{pmatrix}
\begin{pmatrix}  1 &  0 \\ 0 & -d_0  \end{pmatrix}
\\
&\phantom{--} \times
 \exp( J_0\sigma_3) 
\begin{pmatrix} 1 & 0 \\ 0 & -c_0^{-1}\end{pmatrix} 
\begin{pmatrix}  1 &  0 \\ i & 1  \end{pmatrix}
\begin{pmatrix}  1 &  0 \\ 0 & -c_0  \end{pmatrix}
 \exp(-\hat{J}_2^* \sigma_3)
\\
&= (I+O(t^{-\delta}))
\begin{pmatrix} 
   e^{J_1-\hat{J}^*_2}(e^{J_0} +c_0^{-1}d_0 e^{-J_0} ) 
& i d_0e^{{J}_1 -J_0 +\hat{J}^*_2}  
\\
 -i c_0^{-1} e^{-J_0-J_1-\hat{J}^*_2} & e^{\hat{J}^*_2-J_0-J_1 } \end{pmatrix}
\end{align*}
with 
$$
\hat{J}^*_2= \lim_{\substack{\lambda \to \infty \\
\lambda \in \hat{\mathbf{c}}^{\infty}_{2 *} }}
\Bigl(\int^{\lambda}_{\lambda_2} \Lambda_3(\tau) d\tau -\frac 14 (t\lambda
-2\theta_{\infty} \log\lambda) \sigma_3 \Bigr).
$$
Thus we have $\Gamma^{\infty}_{\infty 2}$ and $\Gamma^{\infty}_{\infty
1}$ for $0<\phi <\pi/2.$
In the case $-\pi/2<\phi<0$, using the Stokes graph of Figure \ref{stokes} (a), 
we have similarly
\begin{align*}
\Gamma^{\infty}_{\infty 2}&= (I+O(t^{-\delta}))
\begin{pmatrix} 
e^{\hat{J}^*_1-J_2}( e^{J_0} +c_0^{-1}d_0 e^{-J_0} ) &  
-i d_0 e^{J_2 -J_0 +\hat{J}^*_1}
\\
ic_0^{-1} e^{-J_2-J_0 -\hat{J}^*_1}    & 
e^{J_2-J_0 -\hat{J}^*_1}     \end{pmatrix},
\\
\Gamma^{\infty}_{\infty 1}&= (I+O(t^{-\delta}))
\begin{pmatrix} e^{J_2 -J_0 -\hat{J}_1} & id_0 e^{-J_0+J_2+\hat{J}_1} 
\\  -i c^{-1}_0 e^{-J_2-J_0-\hat{J}_1}  &
   e^{-J_2+\hat{J}_1}(e^{J_0} +c_0^{-1}d_0 e^{-J_0} ) 
  \end{pmatrix},
\end{align*}
where 
\begin{align*}
& \hat{J}_1= \lim_{\substack{\lambda \to \infty \\
\lambda \in \hat{\mathbf{c}}^{\infty}_{1 } }}
\Bigl(\int^{\lambda}_{\lambda_1} \Lambda_3(\tau) d\tau -\frac 14 (t\lambda
-2\theta_{\infty} \log\lambda) \sigma_3 \Bigr),
\\
& \hat{J}^*_1= \lim_{\substack{\lambda \to \infty \\
\lambda \in \hat{\mathbf{c}}^{\infty}_{1 *} }}
\Bigl(\int^{\lambda}_{\lambda_1} \Lambda_3(\tau) d\tau -\frac 14 (t\lambda
-2\theta_{\infty} \log\lambda) \sigma_3 \Bigr)
\end{align*}
with a curve $\hat{\mathbf{c}}^{\infty}_{1*}$ tending to $e^{3\pi i/2}\infty.$
\par
Let $M^0=(m_{ij}^0)$ and $M^1=(m_{ij}^1)$, and suppose that 
$0<\phi <\pi/2.$
The relations $\Gamma^{\infty}_{\infty 2} M^0_*=(S_2^*)^{-1}$, 
$M^1_* \Gamma^{\infty}_{\infty 1}= (S^*_1)^{-1}$ yield 
\begin{align*}
& e^{\hat{J}_2-J_0-J_1} m_{11}^0 - id_0 e^{\hat{J}_2-J_0+J_1} \hat{u}m_{21}^0=1,
\\
& ic^{-1}_0 e^{-\hat{J}_2-J_0-J_1} m_{11}^0 + (e^{J_0}+ c^{-1}_0 d_0e^{-J_0})
 e^{-\hat{J}_2+J_1}\hat{u} m_{21}^0=0,
\\
& id_0 e^{\hat{J}_2^*-J_0+J_1} m_{11}^1 + e^{\hat{J}^*_2-J_0-J_1}\hat{u}^{-1} 
m_{12}^1=0,
\\
& (e^{J_0}+ c^{-1}_0 d_0e^{-J_0})e^{-\hat{J}_2^*+J_1} m_{11}^1 
- ic^{-1}_0 e^{-\hat{J}_2^* -J_0-J_1}\hat{u}^{-1} m_{12}^1 =1
\end{align*}
up to the factor $1+O(t^{-\delta})$,
from which it follows that
\begin{align*}
&  \hat{u} m^0_{21}=-ic_0^{-1}e^{-\hat{J}_2-J_0-J_1}, 
\quad
  m^0_{11}= e^{J_1-\hat{J}_2}(e^{J_0}+c^{-1}_0d_0e^{-J_0}),
\\
 & \hat{u}^{-1}m^1_{12}=-id_0  e^{-J_0+J_1+\hat{J}^*_2}, \quad
m^1_{11}=   e^{-J_0-J_1+ \hat{J}^*_2}.
\end{align*}
Therefore entries of $M^0$ and $M^1$ satisfy
$$ 
\frac{m^0_{11}m^1_{11}}{m^0_{21}m^1_{12}}
= -(1+ c_0d^{-1}_0 e^{2J_0})  , \quad
m^0_{11}=e^{J_2-\hat{J}_2}(1+c^{-1}_0 d_0 e^{-2J_0}) .
$$
Note that $m^0_{11}m^1_{11}+m^0_{21}m^1_{12}
=e^{-\pi i\theta_{\infty}}$ by \eqref{3.1}. 
From the first equation we have
$$
e^{\pi i\theta_{\infty}} m^0_{21}m^1_{12} = -c_0^{-1}d_0 e^{-2J_0}(1+O(t
^{-\delta})),
$$
and from the second equation
$$
e^{J_2-\hat{J}_2} =\frac{m^0_{11}(1+O(t^{-\delta}))}{1+c_0^{-1}d_0e^{-2J_0}}
 =\frac{m^0_{11}(1+O(t^{-\delta}))}{1-m^0_{21} m^1_{12} (m^0_{11} m^1_{11}
+m^0_{21}m^1_{12})^{-1} } =\frac 1{e^{\pi i\theta_{\infty}} m^1_{11}}
(1+O(t^{-\delta})).
$$
In the case $-\pi/2 <\phi<0$, we have
\begin{align*}
&(e^{J_0}+c_0^{-1}d_0e^{-J_0})e^{\hat{J}^*_1-J_2}m^0_{11}-id_0e^{\hat{J}^*_1-J_0
+J_2}\hat{u}m^0_{21}=1,
\\
&ic_0^{-1}e^{-\hat{J}^*_1-J_0-J_2}m^0_{11}+ e^{-\hat{J}_1^*-J_0+J_2}\hat{u}
m^0_{21}=0,
\\
&e^{-\hat{J}_1-J_0+J_2}m^1_{11}-ic_0^{-1}e^{-\hat{J}_1-J_0-J_2}\hat{u}^{-1}
m^1_{12}=1,
\\
&id_0e^{\hat{J}_1-J_0+J_2}m^1_{11}+(e^{J_0}+c_0^{-1}d_0e^{-J_0})e^{\hat{J}_1
-J_2}\hat{u}^{-1}m^1_{12}=0
\end{align*} 
yielding
\begin{align*}
&\hat{u}m^0_{21}=-ic_0^{-1}e^{-\hat{J}^*_1-J_0-J_2}, \quad m^0_{11}
=e^{-\hat{J}^* _1-J_0+J_2},
\\
&\hat{u}^{-1}m^1_{12}=-id_0 e^{\hat{J}_1-J_0+J_2}, \quad m^1_{11}=
(e^{J_0}+c_0^{-1}d_0 e^{-J_0})e^{\hat{J}_1-J_2},
\end{align*}
from which we derive the first equality of the case above and
$m^0_{11}=e^{J_2-\hat{J}_2} (1+O(t^{-\delta})).$
Then, for $0<|\phi|<\pi/2$, 
%%%%% (4.4) %%%%%%%%%%%%%
\begin{equation}\label{4.2}
\begin{split}
&\mathfrak{m}_{\phi}=(1+O(t^{-\delta}))e^{J_2-\hat{J}_2}, 
\\
& \frac{m^0_{21}m^1_{12}}{m^0_{11}m^1_{11}
+m^0_{21}m^1_{12}} = e^{\pi i\theta_{\infty}}m^0_{21}m^1_{12}
= -(1+O(t^{-\delta})) c^{-1}_0d_0 e^{2J_1-2J_2}
\end{split}
\end{equation}
with $\mathfrak{m}_{\phi}$ as in Theorem \ref{thm2.1}. Here the contour
of the integral $J_2-\hat{J}_2$ on $\mathbb{C}\setminus [e^{i\phi},+\infty]$
corresponds to the cycle $\mathbf{b}$ as in Section \ref{ssc5.1}.
%%%%%%%%%%%%%%%%%%%%%%%%%%%%%%%%%%
%%%%% Section 5 %%%%%%%%%%%%
\section{Asymptotics of monodromy data}\label{sc5}
%%%%%%%%%%%%%%%%%%%%%%%%%%%%%%%%%%%%%%%%
To calculate the integrals $J_{1,2}$ and $\hat{J}_{2}$ we make a further
change of variables
%%%%%%%% (5.2) %%%%%%
\begin{equation}\label{5.2}
\lambda=\lambda(z)=  e^{i\phi} z.
\end{equation} 
%%%%%%%%%%%%%%%%%%
%%%%%% Proposition 5.1 %%%%%%%
\begin{prop}\label{prop5.1}
By \eqref{5.2}, the turning points $\lambda_1(t),$ $\lambda_2(t),$ 
$\lambda^0_1(t),$ $\lambda^0_2(t)$ are mapped to
\begin{align*}
z_1(t) &=a_{\phi}^{1/2} +2 e^{-i\phi} \theta_{\infty}t^{-1} + O(t^{-2}), \quad 
& z_2(t) &=-a_{\phi}^{1/2} +2 e^{-i\phi} \theta_{\infty}t^{-1}+ O(t^{-2}),    
\\
z_1^0(t) &= 1 +O(t^{-2}), \quad &  z_2^0(t) &= -1 +O(t^{-2}), 
\end{align*} 
respectively.
\end{prop}
%%%%%%%%%%%%%%%%%%%%%%%%%%%
By \eqref{5.2} the characteristic root \eqref{3.13} becomes 
%%%%%%%% (5.3) %%%%%%%
\begin{equation}\label{5.3}
\mu=\mu(z) =\frac 14 \sqrt{\frac{a_{\phi}-z^2}{1-z^2}}
 +\frac{e^{-i\phi}\theta_{\infty}z}{2w} t^{-1} +\tilde{g}_2(t,z) t^{-2},
\end{equation}
%%%%%%%%%%%%%%%%%%%%%%%%%%%%%%%%%%%%%%%%%%%5
where 
$$
w=w(z)=w(a_{\phi},z) =\sqrt{(1-z^2)(a_{\phi}-z^2)}
$$
and $\tilde{g}_2(t,z) \ll 1$ if $|z^2 - 1|^{-1}+|z^2-a_{\phi}|^{-1} \ll 1.$
The image of $\mathbb{P}=\mathbb{P}_+ \cup \mathbb{P}_-$ under the map 
\eqref{5.2} is given by 
$\Pi^* =\Pi^*_+\cup \Pi^*_-$ with $\Pi_{\pm}^*=z(\mathbb{P}_{\pm})$ glued 
along the cuts $[z_2^0(t),z_2(t)]$ and $[z_1(t),z_1^0(t)]$, $\mathbb{P}$ being
defined in Subsection \ref{ssc3.4}.
Then $\mu=\mu(z)$ is an algebraic function on $\Pi^*.$ 
The elliptic curve defined by $w(z)$ is given by the two-sheeted Riemann
surface $\Pi_{a_{\phi}}=\Pi_{a_{\phi},+}\cup \Pi_{a_{\phi},-}$ glued
along the cuts $[-1, -a_{\phi}^{1/2}]$, $[a_{\phi}^{1/2}, 1]$ satisfying
$z_{2,\,1}(t)=\mp a_{\phi}^{1/2}+O(t^{-1})$, 
$z_{2,\,1}^0(t)=\mp 1+O(t^{-2}).$ Each square root in \eqref{5.3} is such that
$\sqrt{(a_{\phi}-z^2)/(1-z^2)}\to 1,$ 
$z^{-2}w(z)\to -1$ as $z\to \infty$ on $\Pi_{a_{\phi},+}$, and then 
$a_{\phi}^{-1/2}\sqrt{(a_{\phi}-z^2)/(1-z^2)}$,\,\, $a_{\phi}^{-1/2}w(z)\to 1$ 
as $z\to 0$ on $\Pi_{a_{\phi},+}$. 
The elliptic curve $\Pi_{A_{\phi}}=\Pi_+\cup\Pi_-$ 
of Section \ref{sc2} is the image of 
$\mathbb{P}_+^{\infty} \cup \mathbb{P}_-^{\infty}=\lim_{t\to\infty}
\mathbb{P}_+\cup \mathbb{P}_-$ under the map \eqref{5.2}.
\par
%%%%%%%%%%%%%%%%%%%%%%%%%%%%%%%%%%%%%%%%%%%%%%%%%%%%%%%%%%%%%%%%%%%
%%%%%%%% Figure 5.1 %%%%%%%%%%%%%%%%%%%%%%%%%%%%%%%%%%
%%%%%%%%%%%%%%%%%%%%%%%%%%%%%%%%%%%%%%%%%%%
{\small
\begin{figure}[htb]
\begin{center}
\unitlength=0.8mm
%%%%%%%%%%%%%%%%%%%%%%%%%%%%%%%%%%
%%%%%%%%%%%%%%%%%%%%%%%%%%%%%%%%%%
%%%%%%%%%%%%%%%%%%
\begin{picture}(80,48)(0,2)
\put(0,17){\makebox{$-1$}}
\put(22,3){\makebox{$-a_{\phi}^{1/2}$}}
\put(73,25){\makebox{$1$}}
\put(52,39){\makebox{$a_{\phi}^{1/2}$}}
\thinlines
\put(10,25.5){\line(2,-1){17}}
\put(10,24.5){\line(2,-1){17}}
\put(70,25.5){\line(-2,1){17}}
\put(70,24.5){\line(-2,1){17}}
\qbezier(35,33.5) (40,37) (46,38)
\qbezier(9,31.5) (14,32.7) (19,30.5)
\put(46,38){\vector(4,1){0}}
\put(9,31.5){\vector(-4,-1){0}}
\put(31,31){\makebox{$\mathbf{a}$}}
\put(20,29){\makebox{$\mathbf{b}$}}
\put(65,8){\makebox{$\Pi_{a_{\phi},+}$}}
\thicklines
\put(10,25){\circle*{1}}
\put(27,16.5){\circle*{1}}
\put(53,33.5){\circle*{1}}
\put(70,25){\circle*{1}}
\qbezier(24,18) (25,25) (40,32.5)
\qbezier(40,32.5) (56,40) (56.5,31.7)
\qbezier[15](40,18) (54,25.5) (56,30)
\qbezier[15](24,16) (26,11) (40,18)
\qbezier(6,27) (10,32.5) (25.9,24.8)
\qbezier(29,23.2) (35.5,18) (34,13)
\qbezier(6,27) (3.5,21) (14,15)
\qbezier (14,15)(30,7)(34,13)
\end{picture}
\end{center}
\caption{Cycles $\mathbf{a},$ $\mathbf{b}$ on $\Pi_{a_{\phi}}$}
\label{cycles2}
\end{figure}
}
%%%%%%%%%%%%%%%%%%%%%%%%%%%%%%%%%%%%%%%%%%%%%%%
%%%%%%%%%%%%%%%%%%%%%%%%%%%%%%%%%%%%%%%%%%%%%%%%
Recall the cycles $\mathbf{a}$ and $\mathbf{b}$ on $\Pi_{A_{\phi}}$ descried 
in Figure \ref{cycles1}. 
We remark that $\mathbf{a}$ and $\mathbf{b}$ may also be regarded as 
those on $\Pi_{a_{\phi}}$ as in Figure \ref{cycles2}, if $t$ is sufficiently 
large and if the distance 
between
$\mathbf{a}\cup \mathbf{b}$ and $\{\pm 1 \} \cup \{ \pm a_{\phi}^{1/2} \}$ 
is $\gg 1$.
We use the same symbols $\mathbf{a}$ and $\mathbf{b}$ as in 
Figure \ref{cycles1}, provided that
$A^{1/2}_{\phi}=\lim_{t\to\infty} a^{1/2}_{\phi}(t),$ which will not
cause confusions.
%%%%%%%%%%%%%%%%%%%%%%%%%%%%%%%%%%%%%%%%%%%%%
%%%%%% ssc 5.1 %%%%%%%%%%%
\subsection{Expressions in terms of elliptic integrals}\label{ssc5.1}
%%%%%%%%%%%%%%%%%%%%%%%%%%%%%%%%%%%%%%%%%%%%%%%
Let $J^{\mu}_{\iota}$ and $\hat{J}^{\mu}_{\iota}$ be such that 
$J^{\mu}_{\iota}\sigma_3= J_{\iota}\sigma_3 |_{\Lambda_3(\tau) \mapsto  t\mu
(\tau)\sigma_3}$ and
$\hat{J}^{\mu}_{\iota}\sigma_3= \hat{J}_{\iota}\sigma_3 |_{\Lambda_3(\tau)
 \mapsto  t\mu(\tau)\sigma_3}$, respectively.
Note that $z_1(\infty)=A^{1/2}_{\phi},$ $z_2(\infty)=-A^{1/2}_{\phi},$ and 
that $z_{1,\,2}(\infty)-z_{1,\,2}(t)\ll t^{-1}$ by Proposition \ref{prop5.1}.
Recall that $\lambda_{1,\,2}=\lambda_{1,\,2}(\infty)$ in these integrals.
By use of $\int^{\lambda}
_{\lambda_{\iota}} (t\mu(\tau) -\frac 14 (t-2\theta_{\infty}\tau^{-1}))d\tau$,
we have
\begin{equation*}
J^{\mu}_1-J^{\mu}_2 
=\int^{\lambda_2}_{\lambda_1}t\mu(\tau)d\tau=e^{i\phi} t \int^{-A_{\phi}^{1/2}}
_{A_{\phi}^{1/2}} \mu(\lambda(z)) dz 
 = -e^{i\phi}t \int^{z_1(t)}_{z_2(t)}
\mu(\lambda(z)) dz + t(I_++I_-),
\end{equation*}
where
$$
|I_{\pm}| \ll \biggl| \int^{\pm a_{\phi}^{1/2}+2 e^{-i\phi}\theta_{\infty}t^{-1}
+O(t^{-2})}_{\pm A_{\phi}^{1/2}} \mu(\lambda(z)) dz \biggr| \ll t^{-3/2}
$$
with $A_{\phi}^{1/2}-a_{\phi}^{1/2}\ll t^{-1}.$
Hence
\begin{align*}
J^{\mu}_1-J^{\mu}_2 
&=  -\frac{e^{i\phi}}2 t \int_{\mathbf{a}}\mu(\lambda(z)) dz +O(t^{-1/2})
\\
&= -\frac{e^{i\phi}}8 t \int_{\mathbf{a}}\Bigl( \sqrt{\frac{a_{\phi}-z^2}
{1-z^2} }+ \frac{2e^{-i\phi}\theta_{\infty}t^{-1}z}w \Bigr)dz
-\frac{e^{i\phi}}2 t^{-1} \int_{\mathbf{a}} \tilde{g}_2(t,z)dz +O(t^{-1/2})
\\
&= -\frac{e^{i\phi}}8 t \int_{\mathbf{a}} \sqrt{\frac{a_{\phi}-z^2}{1-z^2} }dz
 +O(t^{-1/2}).
\end{align*}
Furthermore we have
\begin{align*}
 J^{\mu}_2 -\hat{J}^{\mu}_2 &= {e^{i\phi}} t \int_{\mathbf{b}} 
\mu(\lambda(z))dz +O(t^{-1/2})  
\\
&= \frac{e^{i\phi}}4 t\int_{\mathbf{b}} 
\Bigl( \sqrt{\frac{a_{\phi}-z^2}{1-z^2} }+
 \frac{2e^{-i\phi}\theta_{\infty}t^{-1}z}w \Bigr)dz
+{e^{i\phi}} t^{-1} \int_{\mathbf{b}} \tilde{g}_2(t,z)dz +O(t^{-1/2})
\\
&= \frac{e^{i\phi}}4 t\int_{\mathbf{b}}  \sqrt{\frac{a_{\phi}-z^2}{1-z^2} }dz
- \frac{\theta_{\infty}\pi i}2 +O(t^{-1/2}).
\end{align*} 
The relation $\sqrt{(a_{\phi}-z^2)/(1-z^2)} = (1/w)(a_{\phi}-z^2)
= -(w/z)' +a_{\phi}/w -z^{-2} a_{\phi}/w$ yields the following.
%%%%%%%%%%%%%%%%%%%%%%%%%%%%%%%%%%%%%%%%%%%
%%%%%%% Proposition 5.2 %%%%%
\begin{prop}\label{prop5.2} We have
\begin{align*}
J^{\mu}_1- J^{\mu}_2 &=
-\frac{e^{i\phi}} 8 t \int_{\mathbf{a}} \sqrt{\frac{a_{\phi}-z^2}{1-z^2}}dz
+O(t^{-1/2})
\\
&= -\frac{e^{i\phi}} 8 a_{\phi} t \int_{\mathbf{a}} \Bigl( \frac{1}{w}
-\frac{1}{z^2w} \Bigr)dz +O(t^{-1/2}),
\\
J^{\mu}_2-\hat{J}^{\mu}_2  &=
\frac{e^{i\phi}} 4 t \int_{\mathbf{b}} \sqrt{\frac{a_{\phi}-z^2}{1-z^2}}dz
- \frac{\theta_{\infty}\pi i}2 +O(t^{-1/2})
\\
&=  \frac{e^{i\phi}} 4 a_{\phi} t \int_{\mathbf{b}} \Bigl( \frac{1}{w}
-\frac{1}{z^2w} \Bigr)dz - \frac{\theta_{\infty}\pi i}2 +O(t^{-1/2}).
\end{align*}
\end{prop}
%%%%%%%%%%%%%%%%%%%%%%%%%%%%%%%%%%%
To calculate $J_{1,2},$ $\hat{J}_2$ it is necessary to know the integrals
of 
%%%%%%% 5.3a %%%%%%%%%%%%
\begin{equation}\label{5.3a}
\diag T^{-1}T_{\lambda} |_{\sigma_3} =\frac{e^{-i\phi}} 4 \Bigl( 1-\frac{b_3}
{\mu} \Bigr) \frac {d}{dz} \log\frac{b_1+ib_2}{b_1-ib_2}
=\frac{i e^{-i\phi} (b_1b_2' -b_1'b_2)}{2\mu(\mu+b_3)}.
\end{equation}
%%%%%%%%%%%%%%%%%%%%%%%%%
By \eqref{5.2}, $b_k$ are written in the form
%%%%%%%%%% 5.3aa %%%%%%%%%%%%%%%%%%%%
\begin{equation}\label{5.3aa}
\begin{split}
&(z^2-1)b_3= \frac 14 ( z^2-1-4\z_0)+O(t^{-1}),
\\
y &(z^2-1)(b_1+ib_2) =\frac 12 (y-1)\z_0(z+1) -y\z_0 +O(t^{-1}),
\\
 &(z^2-1)(b_1-ib_2) =\frac 12 (y-1)\z_0(z+1) +\z_0 +O(t^{-1})
\end{split}
\end{equation}
with $\z_0= -(e^{-i \phi}y^*-y)(y-1)^{-2}.$
Let $z_{\pm}$ be such that 
$$
b_1(z_+) +ib_2(z_+)=0, \quad   b_1(z_-) -ib_2(z_-)=0.
$$
It is easy to see that
%%%%%%%% (5.4) %%%%%%%%%
\begin{equation}\label{5.4}
z_+ = \frac{y+1}{y-1} +O(t^{-1}),  \quad  z_- =- \frac{y+1}{y-1} +O(t^{-1}),
\end{equation}
%%%%%%%%%%%%%%%%%%%
and that $\mu(z_{\pm})^2 =b_3(z_{\pm})^2.$ By \eqref{5.3a}, 
$\| \diag T^{-1}T_{\lambda}\| \ll |z\mp 1|^{-1/2}$ near $z=\pm 1.$ 
Furthermore,
$\diag T^{-1}T_{\lambda}$ has poles at $z_{\pm} \in \Pi_-$ and is holomorphic
around $z_{\pm} \in \Pi_+.$
From \eqref{3.11} combined with \eqref{5.3} it follows that
$$
(1-z^2)\mu = \frac 14 w(z) \Bigl(1+ \frac{2e^{-i\phi}\theta_{\infty}zt^{-1}}
{a_{\phi}-z^2}+ O(t^{-2}) \Bigr).
$$
Note that $b_3(z_{\pm})=e^{-i\phi}y^{-1}y^*/4+O(t^{-1})$ and $\mu(z_{\pm})^2
= e^{-2i\phi}y^{-2}(y^*)^2 /16 +O(t^{-1}).$ When $z_{\pm} \in \Pi_-$,
%%%%%%%%%% (5.5) %%%%%%%%%%%%%%%%
\begin{equation}\label{5.5}
((z_{\pm})^2-1) b_3(z_{\pm}) = -((z_{\pm})^2-1) \mu(z_{\pm})=\frac 14 w(z_{\pm})
(1+O(t^{-1})).
\end{equation}
%%%%%%%%%%%%%%%%%%%%%%%%%%%
The relations 
\begin{align*}
(z^2-1)b_3\Bigl(\frac 1{z-z_+} -\frac 1{z-z_-}\Bigr) =& \frac 14(z_+-z_-)
+\frac{((z_+)^2-1)b_3(z_+)}{z-z_+} -\frac{((z_-)^2-1)b_3(z_-)}{z-z_-}
\\
=& \frac 14 \Bigl(z_+ - z_- +\frac{w(z_+)}{z-z_+} -\frac{w(z_-)}{z-z_-}\Bigr)
+O(t^{-1})
\end{align*}
and 
$$
\diag T^{-1}T_{\lambda} |_{\sigma_3} =\frac{e^{-i\phi}}4 \Bigl(1-\frac{b_3}
{\mu} \Bigr) \Bigl(\frac 1{z-z_+} -\frac 1{z-z_-} \Bigr)
$$
yield
$$
\diag T^{-1}T_{\lambda} |_{\sigma_3} =\frac{e^{-i\phi}}4 \Bigl(\frac 1{z-z_+}
 -\frac 1{z-z_-} + \Bigl(z_+ - z_- +\frac{w(z_+)}{z-z_+} -\frac{w(z_-)}{z-z_-}
\Bigr)\frac 1{w(z)} +O(t^{-1}) \Bigr).
$$
Hence we have
\begin{align*}
& \Bigl(\int_{\l^{\infty}_1} -\int_{\l^{\infty}_2}\Bigr) 
\diag T^{-1}T_{\lambda} |_{\sigma_3} d\tau 
=\int^{\lambda_2}_{\lambda_1} \diag T^{-1}T_{\lambda} |_{\sigma_3} d\tau
= e^{i\phi} \int^{z_2}_{z_1} \diag T^{-1}T_{\lambda} |_{\sigma_3} dz
\\
&=\frac 14 \log \frac{(z_2 - z_+)(z_1-z_-)}{(z_2-z_-)(z_1 - z_+)}
\\
&\phantom{--}
-\frac 18  \int_{\mathbf{a}} \Bigl(\frac {z_+-z_-} w + \frac{w(z_+)}
{(z-z_+) w} - \frac{w(z_-)}{(z-z_-) w}\Bigr) dz +\pi i r_{\mathbf{a}}+O(t^{-1}), 
\end{align*}
and 
\begin{align*}
& \Bigl(\int_{\l^{\infty}_2} -\int_{\hat{\l}^{\infty}_2}\Bigr) 
\diag T^{-1}T_{\lambda}|_{\sigma_3} d\tau
= {e^{i\phi}} \int_{\mathbf{b}} \diag T^{-1}T_{\lambda}|_{\sigma_3} dz
\\
&= \frac 14  \int_{\mathbf{b}} \Bigl(\frac {z_+-z_-} w + \frac{w(z_+)}
{(z-z_+) w} - \frac{w(z_-)}{(z-z_-) w}\Bigr) dz +2\pi i r_{\mathbf{b}} 
+O(t^{-1}), 
\end{align*}
where $\pi i r_{\mathbf{a}}$, $2\pi i r_{\mathbf{b}}$ with $r_{\mathbf{a}},$
$r_{\mathbf{b}}=0,1$ are the contributions from the poles $z_{\pm}$ in
deforming the contours.
Since $(b_1(z)-ib_2(z))/(b_1(z)+ib_2(z))=y(z-z_-)/(z-z_+),$
$$
c_0^2=-\frac{c_1-ic_2}{c_1+ic_2}=- \frac{y(z_2-z_-)}{z_2-z_+},  
\quad d_0^2=-\frac{y(z_1-z_-)}{z_1-z_+}. 
$$
These combined with Proposition \ref{prop5.2} and \eqref{4.2} yield the 
following.
%%%%%%%%%%%%%%%%%%%%%%%%%%%%%%%%%%%%%%
%%%%%%% Proposition 5.3 %%%%%%
\begin{prop}\label{prop5.3}
Set
$$
W_0(z)=\frac{e^{i\phi}}4 a_{\phi}\Bigl(\frac 1w -\frac 1{z^2w}\Bigr),
\quad
W_1(z)=\frac 14 \Bigl(\frac{z_+-z_-}w +\frac{w(z_+)}{(z-z_+)w}
 - \frac{w(z_-)}{(z-z_-)w} \Bigr)
$$
with $w=w(z)=w(a_{\phi},z).$ Then
\begin{align*}
J_2 -\hat{J}_2& = \frac{e^{i\phi}}4 t \int_{\mathbf{b}} \sqrt{\frac
{a_{\phi}-z^2}{1-z^2} } dz -\int_{\mathbf{b}}W_1(z)dz -\frac {\theta_{\infty}
\pi i}2 -2\pi i r_{\mathbf{b}} +O(t^{-1/2})
\\
&= \int_{\mathbf{b}} (tW_0(z) -W_1(z) ) dz -\frac{\theta_{\infty}\pi i}2
- 2\pi i r_{\mathbf{b}} +O(t^{-1/2}),
\\
2(J_1-J_2) & + \log(c_0^{-1}d_0)  
= - \frac{e^{i\phi}}4 t \int_{\mathbf{a}} \sqrt{\frac{a_{\phi}-z^2}{1-z^2}} dz
 +\int_{\mathbf{a}}W_1(z)dz +2\pi i r_{\mathbf{a}}  +O(t^{-1/2})
\\
&=- \int_{\mathbf{a}} (tW_0(z) -W_1(z) ) dz +2\pi i r_{\mathbf{a}}+O(t^{-1/2}).
\end{align*}
\end{prop}
%%%%%%%%%%%%%%%%%%%%%%%%%%%%%%%%%
%%%% Corollary 5.4 %%%%%%%%%%%
\begin{cor}\label{cor5.4}
We have
\begin{align*}
\log \mathfrak{m}_{\phi} & = \frac{e^{i\phi}}4 t \int_{\mathbf{b}} \sqrt{\frac
{a_{\phi}-z^2}{1-z^2} } dz -\int_{\mathbf{b}}W_1(z)dz -\frac {\theta_{\infty}
\pi i}2 +O(t^{-\delta})
\\
&= \int_{\mathbf{b}} (tW_0(z) -W_1(z) ) dz -\frac{\theta_{\infty}\pi i}2
+O(t^{-\delta}),
\\
\log(m^0_{21}m^1_{12}) &    
= - \frac{e^{i\phi}}4 t \int_{\mathbf{a}} \sqrt{\frac{a_{\phi}-z^2}{1-z^2}} dz
 +\int_{\mathbf{a}}W_1(z)dz -(\theta_{\infty}+1)\pi i + O(t^{-\delta})
\\
&=- \int_{\mathbf{a}} (tW_0(z) -W_1(z) ) dz -(\theta_{\infty}+1)\pi i 
+ O(t^{-\delta}).
\end{align*}
\end{cor}
%%%%%%%%%%%%%%%%%%%%%%
%%%%%% ssc 5.2 %%%%%%%%%%%%%%%%%%%%%%%%%%%%%%%%%%%%%%%%%%
\subsection{Expressions in terms of the $\vartheta$-function}\label{ssc5.2}
%%%%%%%%%%%%%%%%%%%%%%%%%%%%%%%%%%%%%%%%%%%
Under the supposition $a_{\phi}\not=0, 1,$ write
$$
\sn(u)=a_{\phi}^{1/2} \mathrm{sn}(u; a_{\phi}^{1/2}).
$$
Then $z=\sn(u)$ satisfies $(dz/du)^2=w(z)^2=(1-z^2)(a_{\phi}-z^2).$
Setting $z=\sn(u)$ we have, for a given $z_0$ on the elliptic curve 
$\Pi_{a_{\phi}}=\Pi_{a_{\phi},+} \cup \Pi_{a_{\phi},-}$ for $w(z)=w(a_{\phi},z)$,
$$
\frac{dz}{(z-z_0)w(z)}=\frac{du}{\sn(u)-\sn(u_0)}, \quad z_0=\sn(u_0).
$$
Let $u^{\pm}_0$ be such that $z_0^{(\pm)}=\sn(u^{\pm}_0),$ where 
$z^{(+)}_0 =(z_0, w(z_0^{(+)})),$ 
$z^{(-)}_0 =(z_0, - w(z_0^{(+)})).$ Since $\sn'(u^{\pm}_0)=
\pm w(z_0^{(+)}),$
\begin{align*}
\frac{1}{\sn(u)-z_0} &= \frac 1{w(z_0^{(+)})} (\zeta(u-u^{+}_0)
 -\zeta(u-u^{-}_0))
-\frac 1{w(z_0^{(+)})} 
\Bigl(\frac{w'(z_0^{(+)})}2 -\zeta(u_0^{+}-u_0^{-}) \Bigr)
\\
&= \frac 1{w(z_0^{(+)})}\frac d{du}\log\frac{\sigma(u-u^{+}_0)}
{\sigma(u-u^{-}_0)}
-\frac 1{w(z_0^{(+)})} \Bigl(\frac{w'(z_0^{(+)})}2 -\frac{\sigma'}{\sigma}
(u_0^{+}-u_0^{-}) \Bigr).
\end{align*}
This function may be written in terms of the $\vartheta$-function
$$
\vartheta(z,\tau) = \sum_{n\in \mathbb{Z}} e^{\pi i\tau n^2+2\pi izn},
\quad \im\tau >0
$$
coinciding with $\vartheta_3$ of Jacobi and having the properties:
$$
\vartheta(z\pm 1,\tau) =\vartheta(z,\tau),
\quad \vartheta(-z,\tau)=\vartheta(z,\tau), 
\quad \vartheta(z\pm\tau, \tau)=e^{-\pi i(\tau\pm 2z)} \vartheta(z,\tau)
$$
\cite{HC, WW}.
Let us write the fundamental periods of $\Pi_{a_{\phi}}$ as
$$
\omega_{\mathbf{a}}=\omega_{\mathbf{a}}(t)=\int_{\mathbf{a}}\frac{dz}
{w(z)},  \quad
\omega_{\mathbf{b}}= \omega_{\mathbf{b}}(t)=\int_{\mathbf{b}} \frac{dz}
{w(z)},  \quad w(z)=w(a_{\phi},z). 
$$
Then $\sn(u)$ with the modulus $k=a^{1/2}_{\phi}$ has the periods $\omega
_{\mathbf{a}}=4K,$ $\omega_{\mathbf{b}}=2iK'$, and
\begin{align*}
& d \log\frac{\sigma(u-u^{+}_0)}{\sigma(u-u^{-}_0)} =
\frac{2\zeta(\omega_{\mathbf{a}}/2)}{\omega_{\mathbf{a}}} 
(u^{-}_0 -u^{+}_0)du
+ d \log\frac{\vartheta(F(z^{(+)}_0, z) +\nu, \tau)}
{\vartheta(F(z^{(-)}_0, z) +\nu, \tau)},  
\\
& \frac{\sigma'}{\sigma}(u^{+}_0 -u^{-}_0)  = - 
\frac{2\zeta(\omega_{\mathbf{a}}/2)}{\omega_{\mathbf{a}}} 
(u^{-}_0 -u^{+}_0) +\frac{i\pi}
{\omega_{\mathbf{a}}}+\frac{1}{\omega_{\mathbf{a}}} \frac{\vartheta'}
{\vartheta}(F(z^{(-)}_0,z^{(+)}_0) + \nu,\tau)
\end{align*}
\cite{HC, WW}, where
%%%%%%%%%%%%%%%%%%%%%%%%%
%%%% (5.6) %%%%%%%%%%%%%
\begin{equation}\label{5.6}
\tau=\frac{\omega_{\mathbf{b}}}{\omega_{\mathbf{a}}}, \quad  \nu=\frac 12
(1+\tau), \quad F(z_*, z)=\frac 1{\omega_{\mathbf{a}}} \int^z_{z_*}
\frac{dz}{w(z)} =\frac 1{\omega_{\mathbf{a}}} (u-u_*)
\end{equation}
with $z=\sn(u),$ $z_*=\sn(u_*).$ Thus we have 
\begin{align*}
\frac{dz}{(z-z_0)w(z)}
&= \frac{1}{w(z_0^{(+)})} d \log\frac{\vartheta(F(z^{(+)}_0, z) +\nu, \tau)}
{\vartheta(F(z^{(-)}_0, z) +\nu, \tau)}
\\
& -\frac{1}{w(z_0^{(+)})} \Bigl(
\frac{w'(z_0^{(+)})}{2} -\frac{1}{\omega_{\mathbf{a}}} \Bigl(
 {i\pi}+  \frac{\vartheta'}{\vartheta}(F(z^{(-)}_0,z^{(+)}_0) + \nu,\tau)
\Bigr) \Bigr) \frac{dz}{w(z)} ,
\end{align*}
which yields
%%%%%%%%%%%%%%
%%%%% (5.7) %%%
\begin{align}\label{5.7}
& \int_{\mathbf{a}} \frac{dz}{(z-z_0)w(z)}
= - \frac{w'(z_0^{(+)})} {2w(z_0^{(+)})}\omega_{\mathbf{a}}
+\frac 1{w(z_0^{(+)})}\Bigl( {i\pi}
+  \frac{\vartheta'}{\vartheta}(F(z^{(-)}_0,z^{(+)}_0) + \nu,\tau)\Bigr) ,
\\  \label{5.8}
& \int_{\mathbf{b}} \frac{dz}{(z-z_0)w(z)}
= \frac{2\pi i} {w(z_0^{(+)})} F(z^{(-)}_0,z^{(+)}_0) + \tau
 \int_{\mathbf{a}} \frac{dz}{(z-z_0)w(z)}.
\end{align}
The integrals
\begin{align*}
& \int_{\mathbf{a}} \frac{dz}{z^2 w(z)} = \frac 12 (1+a^{-1}_{\phi})\omega
_{\mathbf{a}} -\frac 2{a_{\phi}\omega_{\mathbf{a}}} \Bigl(\frac{\vartheta''}
{\vartheta} -\Bigl(\frac{\vartheta'}{\vartheta}\Bigr)^{\!\! 2}\Bigr)
(\tau/2, \tau),
\\
& \int_{\mathbf{b}} \frac{dz}{z^2 w(z)} = \frac {4\pi i}
{a_{\phi}\omega_{\mathbf{a}}} +\tau \int_{\mathbf{a}} \frac{dz}{z^2 w(z)}  
\end{align*}
follow from $(\partial/\partial z_0)\eqref{5.7}$, 
$(\partial/\partial z_0)\eqref{5.8}$ with $z_0=0.$ Furthermore, 
we obtain 
$$
\int_{\mathbf{a}} \frac{w(z_0) dz}{(z-z_0)w(z)} = - z_0 \omega_{\mathbf{a}}
+ 2\frac{\vartheta'}
{\vartheta}\Bigl( \frac 12 F(z^{(-)}_0, z^{(+)}_0)-\frac 14, \tau \Bigr), 
$$
combining \eqref{5.7} with the relation
$$
\Bigl(\frac{z_0}2 -\frac{w'(z_0)}4 \Bigr)\omega_{\mathbf{a}}+ \frac 12
\frac{\vartheta'}{\vartheta}(F(z^{(-)}_0, z^{(+)}_0) +\nu, \tau) +\frac{\pi i}2
= \frac{\vartheta'}{\vartheta}\Bigl( \frac 12 F(z^{(-)}_0, z^{(+)}_0)-\frac 14,
 \tau \Bigr), 
$$
which is derived by comparing the poles $z_0=\pm 1,$ $\pm a_{\phi}^{1/2}$ 
and $\infty^+$ $(\in \Pi_{a_{\phi},+})$ on $\Pi_{a_{\phi}}$
(cf. \cite[p.513]{Kapaev-1}, \cite[(3.5)]{Kitaev-3}).\footnote[2]
%%%%%%%%%%%%%%%%% [2] %%%%%%%%%%%%%%%%%%%%%%%%%%%
{For $z_0=\infty^+$ by using $\int^{\infty}_0w(z)^{-1}dz=\omega_{\mathbf{b}}/2$, 
it is shown that $F(z_0^{(-)},z_0^{(+)})=\tau-1/2$ and the residues of 
the poles on both sides coincide. Similarly for each of  
$z_0=\pm 1$ and $z_0=\pm a_{\phi}^{1/2}$ the poles on the left-hand side are
cancelled out, since $F(z^{(-)}_0,z^{(+)}_0)=2\sqrt{2}(1-a_{\phi})^{-1/2}
(\omega_{\mathbf{a}})^{-1}t(1+o(1))$, say, for $z_0=1+t^2\to 1$.
The difference of both sides is a constant, and setting $z_0=0$ we obtain
this relation.} 
%%%%%%%%%%%%%%%%%%%%%%%%%%%%%%%%%%%%%%%%%%%%%%%
In what follows let us adopt the convention that the path of the integral
$\omega_{\mathbf{a}}F(z^{(-)}_0,z^{(+)}_0)
=\int^{z^{(+)}_0}_{z^{(-)}_0} w(z)^{-1}dz$ passes through $a_{\phi}^{1/2}$,
i.e. the left end of the cut $[a^{1/2}_{\phi},1]$. Then 
\begin{align*}
\int^{z_0^{(+)}}_{z_0^{(-)}} \frac{dz}{w(z)}=& 
2 \int^{z_0^{(+)}}_{a^{1/2}_{\phi}}
\frac {dz}{w(z)}= -2\int^{-z^{(+)}_0}_{-a^{1/2}_{\phi}} \frac{dz}{w(z)} 
\\
=& -2\biggl(\int^{-z^{(+)}_0}_{a^{1/2}_{\phi}} +\int^{a^{1/2}_0}
_{-a^{1/2}_{\phi}}\biggr) \frac{dz}{w(z)} 
=-\int^{-z^{(+)}_0}_{-z^{(-)}_0}  \frac{dz}{w(z)} -\omega_{\mathbf{a}},
\end{align*}
and hence $-F(-z^{(-)}_0,-z^{(+)}_0)=F(z^{(-)}_0,z^{(+)}_0)+1.$
From these relations with $z_0=z_+,$ $z_-$ such that $z_++z_- =O(t^{-1})$
(cf. \eqref{5.4}), we derive the following. 
%%%%%%%%%%%%%%%%%%%%%%%%%%%%%%%%%%%%%%%%
%%%%%%%%% Proposition 5.5 %%%%%%%%%%%%%%%%
\begin{prop}\label{prop5.5}
For $W_0(z)$ and $W_1(z)$ in Proposition $\ref{prop5.3}$,
\begin{align*}
&\int_{\mathbf{a}} W_0(z)dz = \frac{e^{i\phi}}{8} \Bigl( (a_{\phi}-1)\omega
_{\mathbf{a}} + 
\frac 4{\omega_{\mathbf{a}}} \Bigl(\frac{\vartheta''}
{\vartheta} -\Bigl(\frac{\vartheta'}{\vartheta}\Bigr)^{\!\! 2}\Bigr)
(\tau/2, \tau) \Bigr),
\\
&\int_{\mathbf{b}} W_0(z)dz -\tau \int_{\mathbf{a}} W_0(z)dz = 
- \frac {e^{i\phi} \pi i}{\omega_{\mathbf{a}}}, 
\\
&\int_{\mathbf{a}} W_1(z)dz = \frac{\vartheta'}
{\vartheta}\Bigl( \frac 12 F(z^{(-)}_+, z^{(+)}_+)-\frac 14, \tau \Bigr)+ 
O(t^{-1}),
\\
&\int_{\mathbf{b}} W_1(z)dz -\tau \int_{\mathbf{a}} W_1(z)dz = 
\pi i F(z_+^{(-)}, z_+^{(+)}) +\frac{\pi i}2 + O(t^{-1}),
\end{align*}
where $z_+=(y+1)/(y-1)+O(t^{-1})$, $z_+^{(+)}=(z_+,  w(z_+^{(+)}))$, 
$z_+^{(-)}=(z_+, - w(z_+^{(+)})).$
\end{prop}
%%%%%%%%%%%%%%%%%%%%%%%%%%%%%%%%5
%%%%%% ssc 5.3 %%%%%%
%%%%%%%%%%%%%%%%%%%%%%%%%%%%%%%%%%%%%%
%%%%%%%%%%%%%%%%%%%%%%%%%%%%%%%%%%%%%
%%%%%%%% ssc 5.4 %%%%%
\subsection{Expression of $B_{\phi}(t)$}
\label{ssc5.4}
%%%%%%%%%%%%%%%%%%%
Recall that
$ a_{\phi}(t)=A_{\phi} + t^{-1} B_{\phi}(t),$ $ B_{\phi}(t)=O(1)$
%%%%%%%%%%%%%%%%%%%%%%%
in the domain $S_*(\phi, t'_{\infty}, \kappa_0, \delta_1)$,
where $A_{\phi}$ is a unique solution of the Boutroux equations \eqref{2.1}.
The quantity $B_{\phi}(t)$ is written in terms of
%%%%%%%%%%%%%%%%%%%%%%
\begin{equation*}
\Omega_{\mathbf{a}, \mathbf{b}}  = \int_{\mathbf{a},\mathbf{b}} 
\frac{dz}{w(A_{\phi}, z)},   \quad
% \Omega_{\mathbf{b}} & = \int_{\mathbf{b}} \frac{dz}{w(A_{\phi}, z)}, 
% \\
\mathcal{E}_{\mathbf{a}, \mathbf{b}} = \int_{\mathbf{a},\mathbf{b}} \sqrt{\frac{A_{\phi}-z^2}
{1-z^2} } dz, 
 \quad
\end{equation*}
where $w(A_{\phi}, z) =\sqrt{(A_{\phi}-z^2)(1-z^2)},$ and $\mathbf{a},$
$\mathbf{b}$ are basic cycles on $\Pi_{A_{\phi}}$. Observing that 
$$
\sqrt{\frac{a_{\phi}-z^2}{1-z^2}} - \sqrt{\frac{A_{\phi}-z^2}{1-z^2}}
=\frac 1{\sqrt{1-z^2}}(\sqrt{a_{\phi}-z^2} -\sqrt{A_{\phi} -z^2})
=\frac{t^{-1}B_{\phi}}{2w(A_{\phi},z)}(1+O(t^{-1}B_{\phi})),
$$
and combining this with Corollary \ref{cor5.4} and Proposition \ref{prop5.5},
we obtain
$$
\frac{e^{i\phi}}4 \Bigl(t\mathcal{E}_{\mathbf{a}} +\frac{\Omega_{\mathbf{a}}}
2 B_{\phi}(1+O(t^{-1}B_{\phi})) \Bigr) = \int_{\mathbf{a}} W_1(z)dz
-\log (m^0_{21}m^1_{12}) -\pi i(\theta_{\infty}+1) +O(t^{-\delta})
$$
with $\int_{\mathbf{a}} W_1(z)dz $ as in Proposition \ref{prop5.5}.
%%%%%%%%%%%%%%%%%%%%%%%%%%%%%%
%%% Proposition 5.6 %%%%%
%%%%%%%%%%%%%%%%%%%%%%%%%%%
\begin{prop}\label{prop5.7} 
In $S_*(\phi, t'_{\infty},\kappa_0, \delta_1)$, $B_{\phi}(t)$ is bounded, and 
%% if $\log \mathfrak{m}_{\phi}$ and $\log \mathfrak{m}_0$ are bounded, 
$$
\frac{e^{i\phi}}4 \Bigl(t\mathcal{E}_{\mathbf{a}} +\frac{\Omega_{\mathbf{a}}}
2 B_{\phi} \Bigr) =  \frac{\vartheta'}{\vartheta}\Bigl(\frac 12
F(z_+^{(-)}, z_+^{(+)})-\frac 14, \tau\Bigr) 
 -\log (m^0_{21}m^1_{12}) -\pi i(\theta_{\infty}+1) +O(t^{-\delta}).
$$
\end{prop}
%%%%%%%%%%%%%%%%%%%%%
%%%% Remark 5.1 %%%%%
\begin{rem}
In the argument above, the substitution $(\mathbf{a},
\mathbf{b}) \mapsto (\mathbf{b},-\mathbf{a})$ yields
$$
 \frac{e^{i\phi}}4 \Bigl(t\mathcal{E}_{\mathbf{b}} +\frac{\Omega_{\mathbf{b}}}
 2 B_{\phi}(1+O(t^{-1}B_{\phi})) \Bigr) = \int_{\mathbf{b}} W_1(z)dz+\frac
 {\theta_{\infty}}2 \pi i +\log \mathfrak{m}_{\phi} +O(t^{-\delta})
 $$
with
 $$
 \int_{\mathbf{b}} W_1(z)dz =\frac{\vartheta'}{\vartheta}\Bigl(\frac 12
 \hat{F}(z_+^{(-)}, z_+^{(+)})+\frac {\hat{\tau}}4, \hat{\tau}\Bigr) +O(t^{-1}),
 $$
in which $\hat{F}$ denotes $F$ corresponding to 
$\hat{\tau}=(-\omega_{\mathbf{a}})/
\omega_{\mathbf{b}}.$ Since $\re \int_{\mathbf{a},\mathbf{b}} W_1(z)dz$ are 
bounded in $S_*(\phi,t'_{\infty},\kappa_0, \delta_1),$ the Boutroux equations
\eqref{2.1} with $A_{\phi}$ are equivalent to the boundedness of  
$\re (e^{i\phi}\Omega_{\mathbf{a}}B_{\phi})$ and 
$\re (e^{i\phi} \Omega_{\mathbf{b}} B_{\phi}) $, 
namely, the boundedness of $B_{\phi}(t).$ 
\end{rem}
% The following is obtained by using the boundedness of $B_{\phi}(t).$ 
%%%%%%%%%%%%%%%%%%%%%%%%%%%
%%%%% Proposition 5.7 %%%%
\begin{prop}\label{prop5.8}
We have
$$
 \int^{z^{(+)}_+}_{z^{(-)}_+} \frac {dz}{w(A_{\phi},z)} 
= \int^{z^{(+)}_+}_{z^{(-)}_+} \frac {dz}{w(z)} + O(t^{-1}), \quad
w(z)=w(a_{\phi},z).  
$$
\end{prop}
%%%%%%%%%%%%%%%%%%%%
To show this proposition, we note the following lemma, which is verified by
combining
$$
3 w(A_{\phi},z) =(zw(A_{\phi},z))'+(A_{\phi}+1)\sqrt{\frac{A_{\phi}-z^2}
{1-z^2}} -A_{\phi}(A_{\phi}-1)\frac 1{w(A_{\phi},z)}
$$
with
$$
J_{\mathbf{a}} \Omega_{\mathbf{b}}- J_{\mathbf{b}} \Omega_{\mathbf{a}} 
=\frac 43(1+A_{\phi})\pi i,   \quad
J_{\mathbf{a,\, b}} =\int_{\mathbf{a, \, b}} w(A_{\phi}, z) dz.
$$
The derivation of the last equality is similar to that of Legendre's 
relation \cite{HC, WW}.
%%%%%%%%%%%%%%%%%%%%%%
%%%%% Lemma 5.8 %%%%%%%%%%
\begin{lem}\label{lem5.9}
$\mathcal{E}_{\mathbf{a}}\Omega_{\mathbf{b}} -
\mathcal{E}_{\mathbf{b}}\Omega_{\mathbf{a}} = 4 \pi i.$
\end{lem}
%%%%%%%%%%%%%%%%%
\begin{proof}
By the boundedness of $B_{\phi}(t),$ $\omega_{\mathbf{a},\, \mathbf{b}}=
\Omega_{\mathbf{a}, \, \mathbf{b}} +O(t^{-1})$ for $\mathbf{a},$ $\mathbf{b} 
\subset \Pi_{a_{\phi}} \cap \Pi_{A_{\phi}}.$ 
From Proposition \ref{prop5.5} and Corollary \ref{cor5.4}, it follows that
\begin{align*}
&\biggl( \int_{\mathbf{b}} - \tau \int_{\mathbf{a}}\biggr) (tW_0(z)-W_1(t)) dz
= -\frac{e^{i\phi}\pi i}{\omega_{\mathbf{a}}}t -\pi i F(z^{(-)}_+, z^{(+)}_+)
-\frac{\pi i}2 +O(t^{-1}) 
\\
&= - \frac{e^{i\phi}\pi i}{\omega_{\mathbf{a}}}t -2\pi i \Bigl( p_+(t) +\frac
{\omega_{\mathbf{b}}}{\omega_{\mathbf{a}}}q_+(t)\Bigr) + O(1) \ll 1
\end{align*}
with $p_+(t),$ $q_+(t) \in \mathbb{Z}$.
Write $e^{i\phi}t\mathcal{E}_{\mathbf{a}}/8 +\pi i q_+(t)= Xi$ and
$e^{i\phi}t\mathcal{E}_{\mathbf{b}}/8 -\pi i p_+(t)= Yi,$ where, by \eqref{2.1},
$|\im X|, |\im Y|\ll 1$. Then, by Lemma \ref{lem5.9}, the last line becomes
$$
-\frac{e^{i\phi}\pi i}{\Omega_{\mathbf{a}}}t - \frac{e^{i\phi}}4 \Bigl(
\mathcal{E}_{\mathbf{b}} -\frac{\Omega_{\mathbf{b}}}{\Omega_{\mathbf{a}}}
\mathcal{E}_{\mathbf{a}} \Bigr) t -2\Bigl(\frac{\omega_{\mathbf{b}}}
{\omega_{\mathbf{a}}} X-Y \Bigr) i +O(1) = - 2\Bigl(\frac{\omega_{\mathbf{b}}}
{\omega_{\mathbf{a}}} X-Y \Bigr) i +O(1) \ll 1,
$$
where $\im (\omega_{\mathbf{b}}/\omega_{\mathbf{a}}) \to  
\im (\Omega_{\mathbf{b}}/\Omega_{\mathbf{a}}) >0$ as $t \to \infty$. 
This implies $|X|$, $|Y| \ll 1$, and hence
%%%%%%% (5.11) %%%%%%
\begin{equation}\label{5.11}
\pi i q_+(t) = - e^{i\phi}\mathcal{E}_{\mathbf{a}} t/8 +O(1), \quad
\pi i p_+(t) = e^{i\phi}\mathcal{E}_{\mathbf{b}} t/8 +O(1)
\end{equation}
as $t\to \infty.$ We would like to evaluate
$$
\Upsilon= \biggl | \int^{z_+^{(+)}}_{z_+^{(-)}} \Bigl(
\frac 1{w(z)} -\frac 1{w(A_{\phi},z)} \Bigr) dz \biggr|,
$$
in which the integrand is
$$
\frac 1{w(z)} -\frac 1{w(A_{\phi},z)} =\frac {-B_{\phi}(t) t^{-1}}{2(A_{\phi}
-z^2) w(A_{\phi},z) } +O(t^{-2}).
$$
Observe that the contour $[z_+^{(-)}, z_+^{(+)}]^{\sim}$ on $\Pi_{a_{\phi}}
\cap \Pi_{A_{\phi}}$ may be 
decomposed into $2 p_+(t)\mathbf{a} \cup 2q_+(t)\mathbf{b} \cup \mathbf{a}_+
\cup \mathbf{a}_-,$ where the length of $\mathbf{a}_{\pm}$ is $\ll 1$. 
Using \eqref{5.11} and Lemma \ref{lem5.9}, we have
\begin{align*}
\Upsilon \ll & \biggl| \int_{z_+^{(-)}}^{z_+^{(+)}} \frac{B_{\phi}(t)t^{-1}}
{(A_{\phi}-z^2)w(A_{\phi},z)} dz \biggr| +O(t^{-1}) \ll |B_{\phi} t^{-1}|
|p_+(t) j_{\mathbf{a}} + q_+(t) j_{\mathbf{b}} | +O(t^{-1})
\\
 \ll & |\mathcal{E}_{\mathbf{b}} j_{\mathbf{a}} -\mathcal{E}
_{\mathbf{a}} j_{\mathbf{b}} | + O(t^{-1}) = 2 \Bigl| \frac{\partial}
{\partial A_{\phi}}(\mathcal{E}_{\mathbf{a}} \Omega_{\mathbf{b}} - 
\mathcal{E}_{\mathbf{b}} \Omega_{\mathbf{a}} ) \Bigr| +O(t^{-1}) \ll t^{-1}
\end{align*}
with
$$
j_{\mathbf{a, \, b}} = \int_{\mathbf{a, \, b}} \frac{dz}{(A_{\phi}-z^2)  
w(A_{\phi},z)},
$$
which completes the proof of the proposition.
\end{proof}
%%%%%%%%%%%%%%%%%%%%%%%%%%%%%%%%%%%%%%%%%%%%%%%%%
%%%%%%%%%%%%%%%%%%%%%%%%%%%%%%%%%%%%%%%%%%%%%%
%%%%%%%%%%%%% Section 6 %%%%%%%%%%%%%
\section{Proofs of Theorems \ref{thm2.1} and \ref{thm2.2}}\label{sc6}
%%%%%%%%%%%%%%%%%%%%%%%%%%%%%%%%%%%%%%%%%%%%%%%%%
Let $y(t)$ be a function satisfying \eqref{4.1}, and let
$(M^0,M^1)\in \mathcal{M}_{(\theta_0,\theta_1,\theta_{\infty})}$ be such that 
$\mathfrak{m}_{\phi} m^0_{21} m^1_{12}\not=0.$ Suppose that $0<|\phi|<\pi/2.$
%%%%%%%%%%%%%%%%%%%%%%%%%%%
%%%%% ssc 6.1 %%%%%%%%%%%%%
\subsection{Proof of Theorem \ref{thm2.1}}\label{ssc6.1}
%%%%%%%%%%%%%%%%%%%%%%%%%%%
Note that, by our convention, 
%%%%%%% (6.1) %%%%%%%%
\begin{equation}\label{6.1}
F(z^{(-)}_+, z^{(+)}_+) = \frac 1{\omega_{\mathbf{a}}} \int_{z^{(-)}_+} 
^{z^{(+)}_+}  \frac{dz}{w(z)}  
= \frac 2{\omega_{\mathbf{a}}}\int_{0^{(+)}}^{z^{(+)}_+}\frac{dz}{w(z)}-\frac 12
+O(t^{-1}) 
\end{equation}
on $\Pi_{a_{\phi}}.$ By Propositions \ref{prop5.5}, \ref{prop5.8} and 
Corollary \ref{cor5.4},
%%%%%%% (6.2) %%%%%%%%%%%
\begin{align}\label{6.2}
\omega_{\mathbf{a}} \biggl(\int_{\mathbf{b}} &-\tau \int_{\mathbf{a}} \biggr)
 (tW_0(z)- W_1(z)) dz
\\
\notag
& = \omega_{\mathbf{a}} \Bigl(\log \mathfrak{m}_{\phi} +\tau
(\log (m^0_{21}m^1_{12}) +\pi i 
(\theta_{\infty}+1) )
+\frac{\pi i\theta_{\infty}} 2 +O(t^{-\delta})\Bigr)
\\
\notag
& = \Omega_{\mathbf{a}} \Bigl(\log \mathfrak{m}_{\phi}
 +\frac{\Omega_{\mathbf{b}}}{\Omega_{\mathbf{a}}}
(\log (m^0_{21}m^1_{12}) +\pi i 
(\theta_{\infty}+1) )
+\frac{\pi i\theta_{\infty}} 2 +O(t^{-\delta})\Bigr)
\\
\notag
&= - e^{i\phi} \pi i t- \pi i \omega_{\mathbf{a}} ( F(z_+^{(-)}, z_+^{(+)})
+\tfrac 12)+O(t^{-\delta})
\\
\notag
&= - e^{i\phi} \pi i t- 2\pi i \int_{0^{(+)}}^{z_+^{(+)}} \frac{dz}{w(z)} 
 +O(t^{-\delta})
\\
\notag
&= - e^{i\phi} \pi i t- 2\pi i \int_{0^{(+)}}^{z_{+0}^{(+)}} 
\frac{dz}{w(A_{\phi},z)} 
 +O(t^{-\delta})
\end{align}
with $z^{(+)}_{+0} =(y(t)+1)/(y(t)-1).$
From \eqref{6.2} with $\omega_{\mathbf{a, \, b}}=\Omega_{\mathbf{a,\,
b}} +O(t^{-1}),$ it follows that 
$$ 
 \int_{0^{(+)}}^{z_{+0}^{(+)}} \frac{dz}{w(A_{\phi},z)} = -\frac{1}2 
(e^{i\phi}t-\tilde{x}_0) +O(t^{-\delta}),
$$
where $\tilde{x}_0=x_0 +\Omega_{\mathbf{a}}$ and
%%%%%%%%%%%%%%%%%%%%%%%%%%%%%%%
\begin{equation*}
 x_0 \equiv  \frac {-1}{\pi i}
 \Bigl(\Omega_{\mathbf{b}}\log (m^0_{21}m^1_{12}) 
+\Omega_{\mathbf{a}} \log \mathfrak{m}_{\phi} \Bigr)
-\Bigl(\frac{\Omega_{\mathbf{a}}} 2 +\Omega_{\mathbf{b}}\Bigr) 
( \theta_{\infty} + 1) -\frac{\Omega_{\mathbf{a}}}2  
 \mod 2\Omega_{\mathbf{a}}\mathbb{Z} + 2\Omega_{\mathbf{b}}\mathbb{Z}.
\end{equation*}
%%%%%%%%%%%%%%%%%%%%%%%%%%%
This gives 
%%%%%%%%%%%% (6.4) %%%%%%%%%%%%%%%
\begin{equation}\label{6.4}
\frac{y(t)+1}{y(t)-1} = A_{\phi}^{1/2} \mathrm{sn} ((e^{i\phi} t-x_0)/2
+ O(t^{-\delta});A_{\phi}^{1/2} ) 
\end{equation}
as $t\to \infty$ through $S_*(\phi, t'_{\infty}, \kappa_0, \delta_1 ).$ 
Thus we obtain the desired asymptotic form.
%%%%%%%%%%%%%%%%%%%%%%%%%%%%%%%%%%%%%%%%%%%%%%%%%%%%%%%%
%%%%%%%%%%%%%%%%%%%%%%%%%%
\par
Let $W_1^*(z)$ be the result of replacement of $w(z)$ with $w(A_{\phi},z)$
in $W_1(z)$, that is,
$$
W_1^*(z)=\frac 14 \Bigl(\frac{z_+-z_-}{w(A_{\phi},z )}+\frac{w(A_{\phi},z_+)}
{(z-z_+)w(A_{\phi},z)} -\frac{w(A_{\phi},z_-)}{(z-z_-)w(A_{\phi},z)} \Bigr),
$$
which differs from the early $W_1(z)$ by $O(t^{-1})$ along $\mathbf{a}$, since
$B_{\phi}(t) \ll 1.$ Then, by Proposition \ref{prop5.5}
$$
\int_{\mathbf{a}} W_1(z)= \int_{\mathbf{a}} W_1^*(z)dz+O(t^{-1}) 
= \frac{\vartheta'}{\vartheta} \Bigl(\frac 12 F^*(z_+^{(-)}, z_+^{(+)})-
\frac 14 , \frac{\Omega_{\mathbf{b}}}{\Omega_{\mathbf{a}}} \Bigr) +O(t^{-1})
$$
with 
$$
F^*(z^{(-)}_+, z^{(+)}_+) =\frac 1{\Omega_{\mathbf{a}}} \int^{z_+^{(+)}}
_{z_+^{(-)}} \frac {dz}{w(A_{\phi},z)}
 =\frac 2{\Omega_{\mathbf{a}}} \int^{z_+^{(+)}}
_{0^{(+)}} \frac {dz}{w(A_{\phi},z)} -\frac 12 +O(t^{-1}).
$$
The same argument as in the derivation of Proposition 
\ref{prop5.7} leads the following.
%%%%%%%%%%%%%%%%%%%%%%%%%%%%%
%%%%%%%% Corollary 6.1 %%%%%%%
%%%%%%%%%%%%%%%%%%%%%%%%%%%%%%%%%%
\begin{cor}\label{cor6.1}
In $S_*(\phi, t'_{\infty},\kappa_0, \delta_0),$
$$
\frac{e^{i\phi}}4 \Bigl(t\mathcal{E}_{\mathbf{a}} +\frac{\Omega_{\mathbf{a}}}
2 B_{\phi} \Bigr) = 
- \frac{\vartheta'}{\vartheta}\Bigl(\frac 1{2\Omega_{\mathbf{a}}}
(e^{i\phi}t -x_0), \frac{\Omega_{\mathbf{b}}}{\Omega_{\mathbf{a}}}  \Bigr) 
 -\log(m^0_{21}m^1_{12}) -\pi i(\theta_{\infty}+1) +O(t^{-\delta}).
$$
\end{cor}
%%%%%%%%%%%%%%%%%%%%%%%%%%%%%%%%%%%%%%
{\bf Justification}\,\,  The justification of \eqref{6.4} as a solution 
of (P$_{\mathrm{V}}$) is 
carried out along the argument in \cite[pp.~105--106, pp.~120--121]{Kitaev-3}.
Let $\mathcal{M}=(M^0,M^1)\in \mathcal{M}_{(\theta_0,\theta_1,\theta{\infty})}$
be such that $\mathfrak{m}_{\phi}m^0_{21}m^1_{12}\not=0$ 
(cf.~Remark \ref{rem2.4}). 
Relation \eqref{6.4} and Corollary \ref{cor6.1} provide the leading terms
$y_{\mathrm{as}}=y_{\mathrm{as}}(\mathcal{M},t)$ and
$(B_{\phi})_{\mathrm{as}}=(B_{\phi})_{\mathrm{as}}(\mathcal{M},t)$ without
$O(t^{-\delta})$. Viewing \eqref{4.1}, we set $2y^*_{\mathrm{as}}
=2y^*_{\mathrm{as}}(\mathcal{M},t)=\sqrt{\varphi(t,y_{\mathrm{as}},
(B_{\phi})_{\mathrm{as}})}$ with
\begin{align*}
\varphi(t,y,B) =& e^{2i\phi} y(4y +(1-A_{\phi})(y-1)^2 )
\\
& + e^{i\phi} y(y-1) (4(\theta_0+\theta_1)(y+1)-e^{i\phi}(y-1)B) t^{-1}
\\
& +(y-1)^3((\theta_0-\theta_1+\theta_{\infty})^2y-(\theta_0-\theta_1-\theta
_{\infty})^2)t^{-2},
\end{align*}
where the branch of the square root is chosen in such a way that the leading
term of $y^*_{\mathrm{as}}$ is compatible with $(d/dt)y_{\mathrm{as}}$.
Then $(y_{\mathrm{as}},y^*_{\mathrm{as}})
=(y_{\mathrm{as}}(\mathcal{M},t),y^*_{\mathrm{as}}(\mathcal{M},t))$
satisfies \eqref{4.1} with $B_{\phi}(t)=(B_{\phi})_{\mathrm{as}}(\mathcal{M},t)$ 
in $\hat{S}(\phi,t_{\infty}, \kappa_0,\delta_2)$, 
and the relation
$\mathcal{M}=\mathcal{M}(t,y_{\mathrm{as}},y^*_{\mathrm{as}})$ \cite
[(2.28)]{Kitaev-3}. Here 
$\hat{S}(\phi,t_{\infty},\kappa_0,\delta_2)
=\{t\,|\,\,\re t>t_{\infty}, |\im t|<\kappa_0 \}\setminus \bigcup_{\sigma
\in  Z_{\infty}\cup Z_{0}} \{|t-e^{-i\phi}\sigma|\le \delta_2\}$
with
%% $Z_{\pm 1}=x_0+\Omega_{\mathbf{a}}(\mathbb{Z}+\tfrac 12)+2\Omega
%% _{\mathbf{b}} \mathbb{Z},$
$Z_{\infty}=\{e^{i\phi}t \,|\,\, y_{\mathrm{as}}(t)=1\}=
x_0+\Omega_{\mathbf{a}}\mathbb{Z}+2\Omega_{\mathbf{b}}( \mathbb{Z}+\tfrac 12),$ 
$Z_0=\{e^{i\phi}t \,|\,\,y_{\mathrm{as}}(t)=0,\infty \}=x_0+\Omega_{\mathbf{a}}
(\mathbb{Z}+\tfrac 12)+2\Omega_{\mathbf{b}}( \mathbb{Z}+\tfrac 12),$ 
and
$\mathcal{M}(t,y,y^*)$ is a collection of explicit
functions $(t,y,y^*)$ resulting from the WKB procedure.\footnote[3] 
%%%%%%%%%%%%%%%%%%%%% [3]%%%%%%%%%%%%%%%%%%%%%%%%%%%%%%%%%%%%%
{In our case  $(y_{\mathrm{as}},y^*_{\mathrm{as}})$ may be replaced with
$(y_{\mathrm{as}},(B_{\phi})_{\mathrm{as}})$, and then  
$\mathcal{M}(t,y_{\mathrm{as}},y^*_{\mathrm{as}})$ consists of, say, 
\par
$-e^{i\phi}\pi it-\pi i \Omega_{\mathbf{a}}(F^*(z^{(-)}_{\mathrm{as}},
z^{(+)}_{\mathrm{as}})+\frac 12)
-\Omega_{\mathbf{a}}\pi i(\frac 32\theta_{\infty}+1),$
\par
$\frac 14 e^{i\phi}(t\mathcal{E}_{\mathbf{a}}+\frac 12\Omega_{\mathbf{a}}
(B_{\phi})_{\mathrm{as}})-(\vartheta'/\vartheta)(\frac 12 F^*
(z^{(-)}_{\mathrm{as}}, z^{(+)}_{\mathrm{as}})-\frac 14,\Omega_{\mathbf{b}}/
\Omega_{\mathbf{a}}) +\pi i(\theta_{\infty}+1)$,
\par
$z^{(+)}_{\mathrm{as}}=(y_{\mathrm{as}}+1)/(y_{\mathrm{as}}-1),$
\par\noindent
which are main parts of \eqref{6.2} and of the equation of Proposition 
\ref{prop5.7} without $O(t^{-\delta})$.}
%%%%%%%%%%%%%%%%%%%%%%%%%%%%%%%%%%%%%%%%%%%%%%%%%%%%%%%%%%%%
The monodromy data $\mathcal{M}_{\mathrm{as}}(t)$ for system \eqref{3.6}
with $\mathcal{B}(t,\lambda)$ containing $(y_{\mathrm{as}},y^*_{\mathrm{as}})$
is given by $\mathcal{M}_{\mathrm{as}}(t) =\mathcal{M}(t, y_{\mathrm{as}},
y^*_{\mathrm{as}})+O(t^{-\delta})$ \cite[(2.27)]{Kitaev-3} as a result of the
repeated WKB procedure. Thus, for $|t|\ge T_0$, we have 
$\|\mathcal{M}_{\mathrm{as}}(t)-\mathcal{M} \|\le C t^{-\delta}$ valid uniformly 
in a neighbourhood of $\mathcal{M}$, where $C$ and $T_0$ are independent of 
$\mathcal{M}$ \cite[(2.26)]{Kitaev-3}.  
Then the justification scheme \cite{Kitaev-1} applies 
to our case (see also \cite[Theorem 5.5]{FIKN}). 
This justification combined with the maximal modulus principle
in each neighbourhood of $\sigma\in Z_{0}$ leads us 
Theorem \ref{thm2.1} in $S(\rho,t_{\infty},\kappa_0,\delta_0).$
%%%%%%%%%%%%%%%%%%%%%%%%%%%%%%%%%%%%%%%%%%%%%%
\par The justification scheme is described as follows.
Let $\mathcal{M}_0 $ be a given monodromy data, and
$K(\varepsilon_0):$ $\|\mathcal{M}-\mathcal{M}_0\| \le \varepsilon_0$
a given compact ball with centre at $\mathcal{M}_0$. By the property above
combined with the compactness,
there exists positive numbers $T_{\infty}$ and $C_0$ such that, for every 
$\mathcal{M}\in K(\varepsilon_0)$, 
$\| \mathcal{M} -(\mathcal{M})_{\mathrm{as}} \|\le C_0|t|^{-\delta}$ if 
$|t|\ge T_{\infty}.$ 
Note that $f(t,\mathcal{M}):=\mathcal{M}_0-(\mathcal{M})_{\mathrm{as}}(t)
+\mathcal{M}$ is a map $f$: $K(\varepsilon_0) \to K(\varepsilon_0)$ 
continuous in $\mathcal{M}\in K(\varepsilon_0)$ if $\|\mathcal{M}-(\mathcal{M})
_{\mathrm{as}}(t) \|\le \varepsilon_0,$ i.e. if $|t|^{\delta} \ge C_0/
\varepsilon_0$.
Then by Brouwer's fixed point theorem, 
for each $|t|\ge \max\{T_{\infty},(C_0/\varepsilon_0)^{1/\delta}\}$,
there exists a fixed point $\mathcal{M}_*=\mathcal{M}_*(t) \in 
K(\varepsilon_0)$ such that $f(t,\mathcal{M}_*)=\mathcal{M}_*$ 
i.e. $(\mathcal{M}_*)_{\mathrm{as}}(t)=\mathcal{M}_0$. 
(As will be shown later $\mathcal{M}_*$ is a unique fixed point.)
This implies that $\|\mathcal{M}_*-\mathcal{M}_0 \|
=\|\mathcal{M}_* -(\mathcal{M}_*)_{\mathrm{as}}(t) \| \le C_0|t|^{-\delta}$
and hence
$\mathcal{M}_* =\mathcal{M}_0 +O(t^{-\delta})$
for $|t|\ge \max\{T_{\infty},(C_0/\varepsilon_0)^{1/\delta}\}$.
Then $(\mathcal{M}_*)_{\mathrm{as}}(t)=\mathcal{M}_0$ leads us to 
\begin{align*}
 (y_{\mathrm{as}}(\mathcal{M}_*,t),(B_{\phi})_{\mathrm{as}}(\mathcal{M}_*,t))
&= (y_{\mathrm{as}}(\mathcal{M}_0+O(t^{-\delta}),t),
(B_{\phi})_{\mathrm{as}}(\mathcal{M}_0+O(t^{-\delta}),t))
\\
&= (y_{\mathrm{as}}(\mathcal{M}_0,t+O(t^{-\delta})),
(B_{\phi})_{\mathrm{as}}(\mathcal{M}_0,t+O(t^{-\delta}))),
\end{align*}
which realises isomonodromy deformation with the invariant monodromy data
$\mathcal{M}_0$, and then $y^*_{\mathrm{as}}=(d/dt)y_{\mathrm{as}}$.
This provides the desired expression $y_{\mathrm{as}}(\mathcal{M}_0,t+O(t
^{-\delta}))$ in our main result.
It remains to show the uniqueness of $\mathcal{M}_*.$ To do so suppose that 
there exist $\mathcal{M}_{*1}$ and $\mathcal{M}_{*2}$ such that 
$\mathcal{M}_0=(\mathcal{M}_{*1})_{\mathrm{as}}(t)
=(\mathcal{M}_{*2})_{\mathrm{as}}(t),$ and set 
$y_1=y_{\mathrm{as}}(\mathcal{M}_{*1},t),$ 
$y_2=y_{\mathrm{as}}(\mathcal{M}_{*2},t)$. Let $V_1(\lambda)$ and 
$V_2(\lambda)$ be matrix solutions solving \eqref{3.6} with
$\mathcal{B}(t,\lambda)$ containing $(y_1,y_1')$ and $(y_2,y_2')$, respectively.
Since $V_1(\lambda)$ and $V_2(\lambda)$ yield the same monodromy, 
by the uniqueness of Proposition \ref{prop3.a}
we have $(dV_1/d\lambda)V_1^{-1}=(dY^{(2)}/d\lambda)V_2^{-1}$, 
implying $y_1=y_2$, which contradicts the supposition.
\par
Thus Theorem \ref{thm2.1} is proved.
%%%%%%%%%%%%%%%%%%%%%%%%%%
%%%%%%%%%%%%%%%%%%%%%%%%%%%%%%%%%%%%%%%%%%%%%%%%
%%%%%%%%%%% ssc 6.2 %%%%%%%%%%
\subsection{Proof of Theorem \ref{thm2.2}}\label{ssc6.5}
%%%%%%%%%%%%%%%%%%%%%%%%%%%%%%%%%%%%%
Suppose that $0<|\phi-\pi|<\pi/2$, i.e. $\pi/2<\phi <3\pi/2.$ Recall that 
$\mu(\lambda)$ is on the 
Riemann surface $\mathbb{P}_+\cup \mathbb{P}_-$ glued along the cuts
$[\lambda_1^0, \lambda_1]$ and $[\lambda_2, \lambda_2^0]$. Note that $A_{\phi}
=A_{\phi-\pi}$ by Lemma \ref{lemA.14}.  
Let $\mathcal{S}(\pi/2, 3\pi/2)$ on $\mathbb{P}_+^{\infty}$ be the limit 
Stokes graph as described in Figure \ref{stokes2} (a), (b), in which 
$\hat{\mathbf{c}}^{\infty}_3$ joins $\lambda_1$ or $\lambda_2$ to $e^{5\pi i
/2} \infty$,
and $\mathbf{c}^{\infty}_3$ joins $\lambda_2$ or $\lambda_1$ to $e^{3\pi i/2}
\infty$. 
The anticlockwise $\pi$-radian rotation of the Stokes 
graph $\mathcal{S}(-\pi/2,\pi/2)$ for $0<|\phi|<\pi/2$ as in Figure 
\ref{stokes} results in $\mathcal{S}(\pi/2,3\pi/2)$.
The curve $\mathbf{c}^{\infty}_3$ corresponds to $\mathbf{c}^{\infty}_1$ or
$\mathbf{c}^{\infty}_2$.  
Let the loops $\breve{l}_0$, $\breve{l}_1$ be the results of the 
same rotation of $\hat{l}_0$, $\hat{l}_1$ in Figure \ref{loops0}. 
The loops $\breve{l}_0$,
$\breve{l}_1$ and the starting point $\breve{p}_{\mathrm{st}}$ are as in Figure 
\ref{stokes2} (c), and $\arg(\breve{p}_{\mathrm{st}})=3\pi/2$. 
%%%%%%%%%%%%%%%%%%%%%%%%%%%%%%%%%%%%%%%%%%%%%%%%%%%
%%%%%%%%%%%%%%%%%%% Figure 6.1 %%%%%%%%%%%%%%%%%%%%%%%%%%%
%%%%%%%%%%%%%%%%%%%%%%%%%%%%%%%%%%%%%%
%%%%%%%%%%%%%%%%%%%%%%%%%%%%
{\small
\begin{figure}[htb]
\begin{center}
\unitlength=0.73mm
%%%%%%%%%%%%%%%%%%%%%%%%%%%%%%%%%%
%%%%%%%%%%%%%%%%%%%%%%%%%%%%%%%%%%
%%%%%%%%%%%%%%%%%%%%%%%%%%%%%%%%%%
%%%%%%%%%%%%%%%%%%%%%%%%%%%%%%%%%%
\begin{picture}(60,63)(-30,-35)
\put(18,-7){\circle*{1.5}}
\put(6,-8){\circle*{1.5}}
\put(-6,8){\circle*{1.5}}
\put(-18,7){\circle*{1.5}}
\thicklines
\qbezier (6,-8) (4.5,-4) (0, 0)
\qbezier (6,-8) (1,-20) (1, -30)
\qbezier (6,-8) (13,-8) (18, -7)
\qbezier (-6,8) (-4.5,4) (0, 0)
\qbezier (-6,8) (-1,20) (-1, 30)
\qbezier (-6,8) (-13,8) (-18, 7)
\put(20,-5.5){\makebox{$-e^{i\phi}$}}
\put(-27,3){\makebox{$e^{i\phi}$}}
\put(-1.5,-9.0){\makebox{$\lambda_2$}}
\put(-11.5,11){\makebox{$\lambda_1$}}
\put(1,24){\makebox{$\hat{\mathbf{c}}_3^{\infty}$}}
\put(3,-27){\makebox{$\mathbf{c}_3^{\infty}$}}
\put(0,2){\makebox{$\mathbf{c}_0$}}
\put(-22,-38){\makebox{(a) $\pi/2 <\phi <\pi$}}
\end{picture}
%%%%%%%%%%%%%%%%%%
%%%%%%%%%%%%%%%%%%
%%%%%%%%%%%%%%%%%%
\quad
%%%%%%%%%%%%%%%%%%%%%%%%%%%%%%%%%%
\begin{picture}(60,60)(-30,-35)
\put(-18,-7){\circle*{1.5}}
\put(-6,-8){\circle*{1.5}}
\put(6,8){\circle*{1.5}}
\put(18,7){\circle*{1.5}}
\thicklines
\qbezier (-6,-8) (-4.5,-4) (0, 0)
\qbezier (-6,-8) (-1,-20) (-1, -30)
\qbezier (-6,-8) (-13,-8) (-18, -7)
\qbezier (6,8) (4.5,4) (0, 0)
\qbezier (6,8) (1,20) (1, 30)
\qbezier (6,8) (13,8) (18, 7)
\put(-27,-7.5){\makebox{$e^{i\phi}$}}
\put(21,4.5){\makebox{$-e^{i\phi}$}}
\put(-3,-10.0){\makebox{$\lambda_1$}}
\put(7,11){\makebox{$\lambda_2$}}
\put(-7,24){\makebox{$\hat{\mathbf{c}}_3^{\infty}$}}
\put(1,-27){\makebox{$\mathbf{c}_3^{\infty}$}}
\put(-5,2){\makebox{$\mathbf{c}_0$}}
\put(-22,-38){\makebox{(b) $\pi <\phi <3\pi/2$}}
\end{picture}
%%%%%%%%%%%%%%%%%%
%%%%%%%%%%%%%%%%%%%%%%%%
\unitlength=0.65mm
\begin{picture}(63,45)(-55,-50)
\thicklines
\put(-16,-40){\circle*{1}}
\put(-16,-40){\line(2,5){14.28}}
\put(-16,-40){\line(-1,2){13.57}}
\put(0,0){\circle{9}}
\put(0,0){\circle*{1}}
\put(-31.8, -8.2){\circle{9}}
\put(-31.8, -8.2){\circle*{1}}
%%%%%%%%%%%%%%%%%%%%%%%%%%%%
\thinlines
\put(-12,-25){\line(2,5){3}}
\put(-12,-25){\vector(-1,-3){0}}
\put(-6,-20.2){\line(-2,-5){3}}
\put(-6,-20.2){\vector(1,3){0}}
\qbezier(2,7)(0,8)(-2,7)
\put(-2,7){\vector(-3,-2){0}}
\put(-22,-23.5){\line(1,-2){3}}
\put(-23,-30){\line(-1,2){3}}
\put(-22,-23.5){\vector(-1,2){0}}
\put(-23,-30){\vector(1,-2){0}}
\qbezier(-34.7,-1.9)(-36.8,-2.6)(-37.6,-4.7)
\put(-37.6,-4.7){\vector(-1,-4){0}}
%%%%%%%%%%%%%%%%%%
\put(-14,-43){\makebox{$\breve{p}_{\mathrm{st}}$}}
\put(-44,-16){\makebox{$e^{i\phi}$}}
\put(6,-4){\makebox{$-e^{i\phi}$}}
\put(-24.5,-18){\makebox{$\breve{l}_{1}$}}
\put(-3,-17){\makebox{$\breve{l}_{0}$}}

\put(-37,-53){\makebox{(c) Loops $\breve{l}_0$ and $\breve{l}_1$}}
\end{picture}
%%%%%%%%%%%%%%%%%%%%%%%%%%%%%%%%
%%%%%%%%%%%%%%%%%%%%%%%%%%%%%%%%%%
\end{center}
\caption{Limit Stokes graph and loops} 
\label{stokes2}
\end{figure}
}
%%%%%%%%%%%%%%%%%%%%%%%%%%%%%%%%%%%%%%%%%%%%%%%%%%%%
%%%%%%%%%%%%%%%%%%%%%%%%%%%%%%%%%%%%%%%%%%%%%%%%%%%%%
Let $\breve{M}^0$ and $\breve{M}^1$ be the monodromy matrices defined by
the analytic continuation of $Y_3(t,\lambda)$ along the loops
$\breve{l}_0$ and $\breve{l}_1$, respectively.
Recalling that 
$Y_3(t,\lambda)=Y(t,\lambda) S_2$, and that the analytic continuation of 
$Y(t,\lambda)=Y_2(t,\lambda)$ along $\hat{l}_0$, $\hat{l}_1$ are 
$Y(t,\lambda) M^0,$ 
$Y(t,\lambda)M^1$, respectively, we have $S^{-1}_2M^0S_2=\breve{M}^0,$ 
$S^{-1}_2 M^1 S_2=\breve{M}^1.$
\par
%%%%%%%%%%%%%%%%%%%%%%%%%%%%%%%%%
In the calculation of $\breve{M}^0$ and $\breve{M}^1$ this Stokes graph is
used. Suppose that $\pi <\phi< 3\pi/2.$ Let $Y_4(t,\lambda)$ be the matrix
solution admitting the same asymptotic representation as \eqref{3.8} in the
sector $|\arg\lambda -5\pi/2|<\pi.$ Denote by $\Gamma_3$ a connection 
matrix such that $Y_3=Y_4\Gamma_3$ along $(-\hat{\mathbf{c}}_3^{\infty}) 
\cup \mathbf{c}_0 \cup \mathbf{c}_3^{\infty}$ joining $e^{5\pi i/2}\infty$ 
to $e^{3\pi i/2}\infty$. The Stokes matrix $S_3= e^{\pi i\theta_{\infty}
\sigma_3} S_1 e^{-\pi i\theta_{\infty}\sigma_3}$ is given by $Y_4=Y_3 S_3$. 
Then $\Gamma_3 \breve{M}^0=S_3^{-1}$. The WKB analysis with the Stokes graph  
in Figure \ref{stokes2} (b) on $\mathbb{C}\setminus [-\infty, e^{i\phi}]$ 
leads us to
$$
\Gamma_3 =(I+O(t^{-\delta})) \begin{pmatrix}
e^{\hat{J}_3-J_3}(e^{-J_0}+c_0d_0^{-1} e^{J_0}) & ic_0 e^{J_0+\hat{J}_3+J_3}
\\
-id_0^{-1} e^{J_0-\hat{J}_3 -J_3}  &  e^{J_0 -\hat{J}_3 +J_3} 
\end {pmatrix},
$$
where $J_0=\int_{\lambda_2}^{\lambda_1}\Lambda_3(\tau)d\tau$, 
\begin{align*}
&J_3=\lim_{\substack{\lambda \to \infty \\ \lambda \in \mathbf{c}^{\infty}_3}}
\biggl(\int^{\lambda}_{\lambda_1} \Lambda_3(\tau)d\tau -\tfrac 14 (t\lambda
-2\theta_{\infty}\log \lambda)\sigma_3 \biggr),
\\
&\hat{J}_3=\lim_{\substack{\lambda \to \infty \\ \lambda \in \hat{\mathbf{c}}
^{\infty}_3}}
\biggl(\int^{\lambda}_{\lambda_2} \Lambda_3(\tau)d\tau -\tfrac 14 (t\lambda
-2\theta_{\infty}\log \lambda)\sigma_3 \biggr).
\end{align*}
Recall $\Gamma^{\infty}_{\infty 2}$ such that
$Y_3\Gamma^{\infty}_{\infty 2} =Y=Y_2$ along a path joining $e^{3\pi i/2}\infty$
to $e^{\pi i/2}\infty$ on the right-hand side of $e^{i\phi}$ in Figure 
\ref{stokes2} (b). Then $\breve{M}^1 \Gamma^{\infty}_{\infty 2}=S^{-1}_2,$
and, by the use of the Stokes graph on $\mathbb{C} \setminus [-e^{i\phi},
+\infty]$ with $\mathbf{c}^{\infty}_2$ joining
$\lambda_2$ to $e^{\pi i/2}\infty$ in place of $\hat{\mathbf{c}}^{\infty}_3$,
$$
\Gamma^{\infty}_{\infty 2}=(I+O(t^{-\delta})) \begin{pmatrix}
e^{J_3-J_2+J_0}    &  -i c_0 e^{J_0+J_3 +J_2}  \\
id_0^{-1} e^{J_0-J_3-J_2}  &  e^{-J_3+J_2}(e^{-J_0}+c_0 d_0^{-1} e^{J_0})
\end{pmatrix}
$$
with $J_2$ as in Section \ref{sc4}.
Note that $\breve{m}^0_{12}\breve{m}^1_{21}+\breve{m}^0_{22}
\breve{m}^1_{22}= e^{\pi i\theta_{\infty}}$ follows from 
$\breve{M}^1\breve{M}^0=S_2^{-1}S_1^{-1} e^{-\pi i\theta_{\infty}\sigma_3}$.
From the relations $\Gamma_3 \breve{M}^0 =S_3^{-1}$ and $\breve{M}^1\Gamma
^{\infty}_{\infty 2}=S_2^{-1},$ it follows that 
\begin{align*}
& \breve{m}^0_{12} =-ic_0 e^{\hat{J}_3+J_3+J_0}, \quad
 \breve{m}^0_{22} = e^{\hat{J}_3-J_3-J_0}(1+c_0d_0^{-1} e^{2J_0}), 
\\
& \breve{m}^1_{22} =e^{{J}_3+J_0-J_2}, \quad
 \breve{m}^1_{21} = -i d_0^{-1}e^{-J_3+J_0-J_2} 
\end{align*}
up to the factor $1+O(t^{-\delta})$, which implies
$\breve{m}^0_{22}\breve{m}^1_{22} (\breve{m}^0_{12}\breve{m}^1_{21})^{-1}
=-1-c_0^{-1}d_0 e^{-2J_0}.$ Then we derive
$-c_0^{-1}d_0 e^{-2J_0}(1+O(t^{-\delta}))= e^{\pi i\theta_{\infty}}
(\breve{m}^0_{12}\breve{m}^1_{21} )^{-1} $, and
$$
e^{J_3+J_0-\hat{J}_3}=\frac {1+c_0d_0^{-1} e^{2J_0}} {\breve{m}^{0}_{22}}
=\frac {-1-c_0^{-1}d_0 e^{-2J_0}} {\breve{m}^{0}_{22} \, 
( -c_0^{-1}d_0 e^{-2J_0})} 
=e^{-\pi i\theta_{\infty}}\breve{m}^1_{22} ,
$$ 
in which the contour of $J_3+J_0-\hat{J}_3$ on $\mathbb{C}\setminus 
[-\infty, e^{i\phi}]$ corresponds to the cycle $\mathbf{b}.$
These relations leads to the conclusion for $\pi<\phi <3\pi/2.$
In the case $\pi/2<\phi <\pi$, denoting $J_3|_{\lambda_1 \mapsto \lambda_2}$
and $\hat{J}_3|_{\lambda_2 \mapsto \lambda_1}$ by the same symbols
$J_3$ and $\hat{J}_3$, respectively, we have, by using  
the Stokes graph in Figure \ref{stokes2} (a),
%%%%%%%%%%%%%%%%%%%%%%%%%%%%%%%%%%%%%%%
\begin{align*}
&\Gamma_3  =(I+O(t^{-\delta})) \begin{pmatrix}  
e^{J_0 +\hat{J}_3 -J_3}  &  ic_0 e^{J_0+\hat{J}_3+J_3}
\\
-id_0^{-1} e^{J_0-\hat{J}_3 -J_3}  &
e^{-\hat{J}_3+J_3}(e^{-J_0}+c_0d_0^{-1} e^{J_0}) 
\end {pmatrix},
\\
&\Gamma^{\infty}_{\infty 2}=(I+O(t^{-\delta})) \begin{pmatrix}
 e^{J_3-J_1}(e^{-J_0}+c_0 d_0^{-1} e^{J_0}) & -i c_0 e^{J_0+J_3 +J_1}  \\
id_0^{-1} e^{J_0-J_3-J_1}  & e^{-J_3+J_1+J_0}    
\end{pmatrix}.
\end{align*}
These relation yields
$$
e^{J_3-J_0-\hat{J}_3}(1+O(t^{-\delta}))= ( \breve{m}^{0}_{22})^{-1},
\quad
-c_0^{-1}d_0 e^{-2J_0}(1+O(t^{-\delta}))=  e^{\pi i\theta_{\infty}}
({\breve{m}^0_{12}\breve{m}^1_{21} })^{-1},
$$ 
the first of which immediately follows from $\Gamma_3 \breve{M}^0=S^{-1}_3$.
From these relations the phase shift $\breve{x}_0$ is derived
by a procedure analogous to that for $x_0$.
Thus Theorem \ref{thm2.2} is obtained.
%%%%%%%%%%%%%%%%%%%%%%%%%%%%%%%%%%%%%%%%%%%%%%%%
\par
For $\phi$ such that $|\phi-k\pi|<\pi/2$ $(k\in \mathbb{Z})$, 
denote by $\hat{l}^{(k)}_0,$ $\hat{l}^{(k)}_1$ 
and $\mathcal{S}(k\pi -\pi/2,k\pi +\pi/2)$ the $k\pi$-rotation 
of $\hat{l}_0,$ $\hat{l}_1$ and $\mathcal{S}(-\pi/2, \pi/2)$, respectively.
Let $\tilde{Y}^p(t,\lambda)$ be the canonical solution of \eqref{3.6} admitting
the same form as of \eqref{3.8} in the sector $|\arg \lambda
- 2p\pi-\pi/2|<\pi,$ and let $M^0_p$ and $M^1_p$ be the monodromy matrices given
by the analytic continuations of $\tilde{Y}^p(t,\lambda)$ along $\hat{l}^{(2p)}
_0$ and $\hat{l}^{(2p)}_1$, respectively. Especially, $\tilde{Y}^0(t,\lambda)=
Y(t,\lambda),$ $\hat{l}^{(0)}_0=\hat{l}_0,$ $\hat{l}^{(0)}_1=\hat{l}_1,$  
$M^0_0=M^0,$ $M^1_0=M^1.$
Then $M^0_p$ and $M^1_p$ are as in Remark $\ref{rem2.5}$, and 
$M^0_p,$ $M^1_p$ are calculated on $\mathcal{S}(2p\pi-\pi/2,2p\pi+\pi/2)$. 
Thus this is reduced to the situation, to which Theorem \ref{thm2.1} applies. 
For the canonical solution $\breve{Y}^p(t,\lambda)$ in the sector 
$|\arg \lambda-(2p+1)\pi -\pi/2|<\pi$ such that $\breve{Y}^0(t,\lambda)=Y_3
(t,\lambda)$, the analytic continuations  
along $\hat{l}^{(2p+1)}_0$ and
$\hat{l}^{(2p+1)}_1$ yield $\breve{M}^0_p$ and $\breve{M}^1_p$ as in Remark \ref
{rem2.5}. Calculation of $\breve{M}^0_p$ and $\breve{M}^1_p$ on
$\mathcal{S}((2p+1)\pi-\pi/2,(2p+1)\pi+\pi/2)$  
leads to the results corresponding to Theorem \ref{thm2.2} 
(cf. Remark \ref{rem2.5}).
%%%%%%%%%%%%%%%%%%%%%%%%%%%%%%%%%%%%%%%%%%%
%%%%%%%%%%%%%%%%%%%%%%%%%%%%%%%%%
%%%%%%% Section 7 %%%%%%%%%%%%%%%%%%%%%%
\section{System equivalent to $(\mathrm{P}_{\mathrm{V}})$ and 
another approach to $B_{\phi}(t)$}\label{sc7}
%%%%%%%%
Multiplying both sides of (P$_{\mathrm{V}}$) by $2(dy/dx) y^{-1}(y-1)^{-2},$
we write (P$_{\mathrm{V}}$) in the form
%%%%% (3.3) %%%%%%
\begin{equation}\label{3.3}
\frac d{dx}L= -2x^{-1}L -\frac{2x^{-1} y}{(y-1)^2} +2(1-\theta_0-\theta_1)\frac
{x^{-2}}{y-1},
\end{equation}
where
\begin{align*}%%%%% \label{3.4}
L=L(x) :=& \frac{(y')^2}{y(y-1)^2} -\frac y{(y-1)^2} +2(1-\theta_0-\theta_1)
\frac{x^{-1}}{y-1}
\\
\notag
& -\frac{x^{-2}}4 \Bigl((\theta_0-\theta_1 +\theta_{\infty})^2 y +
(\theta_0-\theta_1-\theta_{\infty})^2\frac 1y \Bigr).
\end{align*}
Furthermore $L=L(x)$ is written in terms of $\psi:=(y+1)/(y-1)$:
%%%%%%% (6.3) %%%%%%%%%%%%%%%%%%%%%%
\begin{align}\label{6.3}
L(x) = & \frac{(\psi')^2}{\psi^2 -1} -\frac 14(\psi^2-1)
 -(1-\theta_0 -\theta_1) x^{-1} (1-\psi)
\\
\notag
& +\frac{x^{-2}}4 \Bigl( (\theta_0-\theta_1 +\theta_{\infty})^2 \frac{1+\psi}
{1-\psi} +(\theta_0 -\theta_1 -\theta_{\infty})^2 \frac{1-\psi}{1+\psi}
\Bigr)
\end{align}
%%%%%%%%%%%%%%%%%%%%%%%%5
with $\psi'=d\psi/dx.$ Then \eqref{3.3} equivalent to (P$_{\mathrm{V}}$) 
becomes
%%%%%%% (6.5) %%%%%%
\begin{equation}\label{6.5}
\frac{d}{dx} L = -2 x^{-1} L
 -\frac 12(\psi^2-1) x^{-1} +(\theta_0+\theta_1 -1)(1-\psi) x^{-2}.
\end{equation}
%%%%%%%%%%%%%%%%%%%%
The quantity $a_{\phi}$ defined by \eqref{3.12} with 
$y^*=y_t(t)=e^{i\phi}y'(x)$ is rewritten in the form
\begin{equation*}
a_{\phi} =1- 4\frac{(y')^2-y^2}{y(y-1)^2} +4(\theta_{0}+\theta_1) x^{-1}
\frac{y+1}{y-1} + x^{-2} \frac{y-1}y ((\theta_0-\theta_1 +\theta_{\infty})^2y
-(\theta_0-\theta_1-\theta_{\infty})^2 ),
\end{equation*}
and then
%%%%%%% (6.6) %%%%%
\begin{align}\label{6.6}
4(\psi')^2 =& (1-\psi^2)(a_{\phi}-\psi^2) - 4(\theta_0+\theta_1)x^{-1}\psi
(1-\psi^2) 
\\
 \notag
&\phantom{------} +4 x^{-2} (2(\theta_0 -\theta_1) \theta_{\infty}  \psi
+(\theta_0-\theta_1)^2 +\theta_{\infty}^2).
\end{align}
From \eqref{6.3} and \eqref{6.6} it follows that
%%%%% (6.7) %%%%%%%%%%%%%%
\begin{equation}\label{6.7}
L= \frac 14 (1-a_{\phi}) + (\theta_0+\theta_1 -1 +\psi)x^{-1} 
-\frac 12 ( (\theta_0-\theta_1)^2+\theta_{\infty}^2) {x^{-2}}.
\end{equation}  
The system consisting of \eqref{6.5} and \eqref{6.6} with \eqref{6.7}
may be regarded as one with respect to $\psi$ and $a_{\phi}=A_{\phi}+x^{-1}
b(x)$ that is equivalent to (P$_{\mathrm{V}}$). 
%%%%%%%%%%%%%%%%%%%%%%%%%%%%%%%%%%%%%%%%%%%%%%%%%%%%%
%%%%%%%%%%%%%%%%%%%%%%%%%%%%%%%%%%%%%%%%%%%%%%%%%%%%%%%%%
The system of equations \eqref{6.5} and \eqref{6.6} is also written in the form
%%%%%%%% (6.8), (6.9) %%%%%%%%%%%%%%%%%%
\begin{align}\label{6.8}
 4(\psi')^2 = &(1-\psi^2)(A_{\phi}-\psi^2) -(1-\psi^2)(4(\theta_0+\theta_1)\psi
-b)x^{-1}
\\
\notag
&  \phantom{---} + 4(2(\theta_0- \theta_1)\theta_{\infty} \psi +(\theta_0-
\theta_1)^2 +\theta_{\infty}^2) x^{-2},
\\
\label{6.9}
 b' = & -2(A_{\phi}-\psi^2) +4\psi' +(4(\theta_0+\theta_1)\psi -b) x^{-1},
\end{align}
which follows from the substitution $a_{\phi}(x) \mapsto A_{\phi} + x^{-1}b(x)$
in \eqref{6.5} and \eqref{6.6} with \eqref{6.7}.
%%%%%%%%%%%%%%%%%%%%%%%%%%%%%%%%%%%%%%%%%%%%%%%%%%%%%%%%
%%%%%%%%%%%%%%%%%%%%%%%%%%%%%%%%%%%%%%%%%%%%%%%%%%%%
%%%%%%%%%%%%%%%%%%%%%%%%%%%%%%%%%%%%%%%%%%%%%%%%%%%%
%%%%% ssc 6.3 %%%%%%%%%%%%%%%%%%%%%%%%%%%%%%%
%%%%%%%%%%%%%%%%%%%%%%%%%%%%%%
Neglecting the terms with the multiplier $x^{-1}$ in \eqref{6.8} and \eqref{6.9},
we have 
\begin{equation*}
 4(\tilde{\psi}')^2 = (1-\tilde{\psi}^2)(A_{\phi}-\tilde{\psi}^2),
\quad
 \tilde{b}' = -2(A_{\phi}-\tilde{\psi}^2) +4\tilde{\psi}'. 
\end{equation*}
The first equation admits the solution
$$
\psi_0(x)=A_{\phi}^{1/2} \mathrm{sn}((x-x_0)/2; A_{\phi}^{1/2}),
\quad   
 4(\psi_0')^2 = (1-\psi_0^2)(A_{\phi}-\psi_0^2)
$$ 
expressed by the Jacobi $\mathrm{sn}$-function with
$\Omega_{\mathbf{a}}=4K,$ $ \Omega_{\mathbf{b}}=2iK',$ $ A_{\phi}^{1/2}=k.$ 
\par
Let us seek a function $b_0(x)$ that solves
$ \tilde{b}' = -2(A_{\phi}-\psi_0^2) +4\psi_0'$ 
and is consistent with $B_{\phi}(t)$ for \eqref{6.4}. 
Put $u=(x-x_0)/2.$ Then this becomes 
%%%%%% (6.10) %%%%%%%%%%%%%%
\begin{equation}\label{6.10}
(b_0)_u = 4(\psi_0)_u + 4(\psi_0^2-A_{\phi}) = 4(\psi_0)_u +4A_{\phi}
(\mathrm{sn}^2 u -1).
\end{equation}
Comparison of double poles of doubly periodic functions yields
$$
(\psi_0)_u + A_{\phi}(\mathrm{sn}^2u -1) +\frac{2}{\Omega_{\mathbf{a}}}
\frac d{du} \Bigl( \frac{\vartheta'}{\vartheta}\Bigl(\frac{u}{\Omega_{\mathbf
{a}}} ,\tau_0 \Bigr) \Bigr) \equiv c_0 \in \mathbb{C}
$$
with $\tau_0=\Omega_{\mathbf{b}}/\Omega_{\mathbf{a}}.$ Integrating this with
\eqref{6.10} along 
$[0,u]$ and putting $u=2K=\Omega_{\mathbf{a}}/2,$ we have
$$
b_0(x)=b_0(x_0) - \frac{2\mathcal{E}_{\mathbf{a}}}{\Omega_{\mathbf{a}}}
(x-x_0)
- \frac{8}{\Omega_{\mathbf{a}}} \frac {\vartheta'}{\vartheta} \Bigl( \frac 1
{2\Omega_{\mathbf{a}}}(x-x_0), \tau_0 \Bigr),
$$
since $2c_0 = -2 \mathcal{E}_{\mathbf{a}}/\Omega_{\mathbf{a}}$ follows from 
$$
A_{\phi} \int^K_0 (\mathrm{sn}^2u -1) du = - \int^{A^{1/2}_{\phi}}_0
\sqrt{\frac{A_{\phi}-z^2}{1-z^2} } dz = - \frac{\mathcal{E}_{\mathbf{a}}}4
\qquad (z= A^{1/2}_{\phi} \mathrm{sn}\, u).
$$
This is consistent with Corollary \ref{cor6.1} if
%%%%%%%%%%%% (6.10*) %%%%%%%%%%%%%%%%%%%%%%5
\begin{equation}\label{6.10*}
b_0(x_0)=\beta_0-\frac{2\mathcal{E}_{\mathbf{a}}} {\Omega_{\mathbf{a}}} x_0
= - \frac 8{\Omega_{\mathbf{a}}}\bigl( \log ( m^0_{21} m^1_{12})
 +\pi i (\theta_{\infty} +1)\bigr) -\frac{2\mathcal{E}_{\mathbf{a}}}
{\Omega_{\mathbf{a}}} x_0.
\end{equation}
Therefore $b_0(x)$ satisfies 
%%%%%% (6.11) %%%%%%%%%%%%%
\begin{equation}\label{6.11}
b'_0(x) = 2(\psi_0(x)^2-A_{\phi}) + 4\psi'_0(x) 
\end{equation} 
and $b_0(e^{i\phi}t)-e^{i\phi}
B_{\phi}(t) \ll t^{-\delta}$ in $S(\phi, t_{\infty}, \kappa_{0},\delta_0)$.
%%%%%%%%%%%%%%%%%%%%%%%%%%%%%%%%%%%%
%%%%%%%%%%%%%%%%%%%%%%%%%%%%%%%%%%%
%%%%%%%%%%%%%%%%%%%%%%%%%%%%%%%%%%%%%%%%%%%%%
%%%%%%%%%%%%%%%%%%%%%%%%%%%%%%%%%%%%%%%%%%%%%%%
%%%% Section A %%%%%%%%%%%%%%%%%%%%%%
\section{Modulus $A_{\phi}$ and the Boutroux equations}\label{scA}
%%%%%%%%%%%%%%%%%%%%%%%%%%%%%%%%%%%%%%%%
We examine a solution $A\in \mathbb{C}$ of the Boutroux equations. 
Let the branch of
$A^{1/2}$ $(\not=0)$ be fixed in such a way that 
$
\re A^{1/2} \ge 0, \quad \text{and} \,\,\,\,  \im A^{1/2} >0 \quad \text{if
$\re A^{1/2}=0.$}
$
In accordance with \cite[Appendix I]{Novokshenov-2} set
$$
I_{\mathbf{a}}(A)=\int_{\mathbf{a}} \sqrt{\frac{A-z^2}{1-z^2}} dz, \quad
I_{\mathbf{b}}(A)=\int_{\mathbf{b}} \sqrt{\frac{A-z^2}{1-z^2}} dz, \quad
\mathcal{I}(A)= \frac{I_{\mathbf{a}}(A)}{I_{\mathbf{b}}(A)},
$$
in which the cycles $\mathbf{a}$ and $\mathbf{b}$ are as in Figure \ref{cycles1}
with $A_{\phi}=A$.
%%%%%%%%%%%%%%%%%%%%%%%%%%%%%%%%%%%%%%%%%
%%% Lemma A1 %%%%%%%%%
\begin{lem}\label{lemA.1}
Let $A\in \mathbb{C}.$ Then $\mathcal{I}(A) \in \mathbb{R}$ if and only if, 
for some $\phi \in \mathbb{R},$ $A$ solves the Boutroux equations
$(\mathrm{BE})_{\phi} :$ $\re e^{i\phi} I_{\mathbf{a}}(A)= \re e^{i\phi} 
I_{\mathbf{b}}(A)=0.$
\end{lem}
%%%%%%%%%%%%%%%
\begin{proof}
Suppose that $\mathcal{I}(A)=\rho \in \mathbb{R},$ and write $I_{\mathbf{a}}(A)
= u+iv,$ $I_{\mathbf{b}}(A)=U+iV$ with $u,$ $v,$ $U,$ $V \in \mathbb{R}.$ Then
$u=\rho U,$ $v=\rho V,$ and hence $u/v= U/V =\tan \phi$ for some $\phi\in 
\mathbb{R},$ which implies $\re e^{i\phi} I_{\mathbf{a}}(A)= \re e^{i\phi} 
I_{\mathbf{b}}(A)=0.$ 
\end{proof}
By Lemma \ref{lem5.9},
\begin{align*}
&\mathcal{I}'(A) =\frac 1{2I_{\mathbf{b}}(A)^2}(\omega_{\mathbf{a}}(A)
I_{\mathbf{b}}(A)-\omega_{\mathbf{b}}(A)I_{\mathbf{a}}(A)) = - \frac{2\pi i}
{I_{\mathbf{b}}(A)^2},
\\
& \omega_{\mathbf{a,\, b}}(A) = \int_{\mathbf{a,\, b}} \frac{dz}
{\sqrt{(A-z^2)(1-z^2)}},
\end{align*}
which implies the following.
%%%%%%%%%%%%%%%%%%%%%%%%%%%%%
%%%%% Lemma A2 %%%%%%%%%%%%%%%%%%%
\begin{lem}\label{lemA.2}
The map $\mathcal{I}(A)$ is conformal on $\mathbb{C}$ as long as $I_{\mathbf
{b}}(A) \not= 0, \infty.$
\end{lem}
%%%%%%%%%%%%%
Near $A=\infty$, observing that
$$
I_{\mathbf{a}}(A) =4A^{1/2} \int^1_0 \sqrt{\frac{1-z^2/A}{1-z^2}} dz
=2\pi A^{1/2}(1+O(A^{-1})),
$$
and that
\begin{align*}
I_{\mathbf{b}}(A) =& 2A^{1/2} \int^{A^{1/2}}_1 \sqrt{\frac{1-z^2/A}{1-z^2}} dz
\\
=& -2i A^{1/2} \int^{A^{1/2}}_1  \Bigl( \frac 1 {\sqrt{z^2-1}}+\frac{-z^2/A}
 {\sqrt{z^2-1}(1+\sqrt{1-z^2/A})} \Bigr) dz
\\
=& -i A^{1/2} \log A \,\,(1+O(|\log A|^{-1})),
\end{align*}
we have $\im (1/\mathcal{I}(A)) = - (2\pi)^{-1} \log|A| \,\, (1+o(1))$ as
$A\to \infty.$ 
%%%%%%%%%%%%%%%%%%
%%%% Lemma A3 %%%%
\begin{lem}\label{lemA.3}
The set
$$
\mathcal{R}(\mathrm{BE})=
\mathcal{I}^{-1}(\mathbb{R}) =\{ A\in \mathbb{C}\, ; \,\, \text{$A$ solves
$(\mathrm{BE})_{\phi}$ for some $\phi \in \mathbb{R}$} \}
$$
is bounded. 
\end{lem}
%%%%%%%%%%%%%%%%%%%%%%
Let us observe the dependence of $A\in \mathcal{R}(\mathrm{BE})$ on $\phi$ or
$t=\tan \phi.$ 
\par
Since $I_{\mathbf{a}}(0) =0$ and $I_{\mathbf{b}}(0)=2i,$ 
$A=0$ solves (BE)$_{\phi=0}$. Conversely we may give the uniqueness lemma, 
which is crucial in
discussing (BE)$_{\phi}$. This is proved by an argument similar to that in 
\cite[\S 7]{Kitaev-2}.
%%%%%%%%%%%%%%%%%%%%%%%%%%%%%%%
%%%%% Lemma A4 %%%%%
\begin{lem}\label{lemA.4}
If $A$ solves $(\mathrm{BE})_{\phi=0}$, then $A=0.$
\end{lem}
%%%%%%%%%%%%%%%%%%%%%
\begin{proof}
Suppose that $\re I_{\mathbf{a}}(A) =\re I_{\mathbf{b}}(A) =0.$ Then $I_
{\mathbf{b}}(A)$ is pure imaginary, and $I_{\mathbf{b}}(A)=
- \overline{I_{\mathbf{b}}(A)} = - I_{\overline{\mathbf{b}}}(\overline{A})
=I_{\mathbf{b}}(\overline{A}),$ that is,
$ I_{\mathbf{b}}(A)-I_{\mathbf{b}}(\overline{A})= 0.$
\par
%%%%%%%%%%%%%%%%%%%%%%%%%%%%%%%%%%%%%%%%%%%%%%%%%%%%%%%%%%%%%%%%%%%%%%%
%%%%%%%%%%%%%%%%%%%%%%%%%%%%%%%%%%%%%%%%%%%%%%%%%%%%%%%%%%%%%%%%%%%%%%
%%%%%%%%%%%%%%%%% Figure 7.1 %%%%%%%%%%%%%%%%%%%%%%
%%%%%%%%%%%%%%%%%%%%%%%%%%%%%%%%%%%%%%%%%%%%%%%%%%%%%%%%%%%%%%%%%%%%%%
%%%%%%%%%%%%%%%%%%%%%%%%%%%%
{\small
\begin{figure}[htb]
\begin{center}
\unitlength=0.75mm
%%%%%%%%%%%%%%%%%%%%%%%%%%%%%%%%%%
%%%%%%%%%%%% (a)  %%%%%%%%%%%%%%%%%%%
%%%%%%%%%%%%%%%%%%%%%%%%%%%%%%%%
\begin{picture}(60,65)(-35,-35)
\put(20,0){\circle*{1.5}}
\put(-20,0){\circle*{1.5}}
\put(10,0){\circle*{1.5}}
\put(10,20){\circle*{1.5}}
\put(10,-20){\circle*{1.5}}
\put(27,0){\line(-1,0){17}}
\put(-20,0){\line(-1,0){13}}
\put(10.5,-20){\line(0,1){40}}
\put(9.5,-20){\line(0,1){40}}
\put(10,0.5){\line(-1,0){30}}
\put(10,-0.5){\line(-1,0){30}}
\put(-5,7){\vector(-1,0){8}}
\put(-2,7){\makebox{$\mathbf{b}$}}
\put(20,-6){\makebox{$0$}}
\put(-3,-8){\makebox{$-\alpha$}}
\put(-28,-9){\makebox{$-1$}}
\put(-15,-24){\makebox{$-\alpha-\beta i$}}
\put(-15,22){\makebox{$-\alpha+\beta i$}}

\put(-20,-36){\makebox{(a) $0\le \re A^{1/2} <1$}}
\thicklines
\put(14,-20){\line(0,1){40}}
\put(6,-20){\line(0,1){16}}
\put(6,4){\line(0,1){16}}
\put(6,-4){\line(-1,0){26}}
\put(6,4){\line(-1,0){26}}
\qbezier (14,-20) (13.5, -23.5) (10,-24) 
\qbezier (6,-20) (6.5,-23.5) (10,-24) 
\qbezier (14,20) (13.5, 23.5) (10,24) 
\qbezier (6,20) (6.5,23.5) (10,24) 
\qbezier (-20,-4) (-23.5, -3.5) (-24,0) 
\qbezier (-20,4) (-23.5, 3.5) (-24,0) 
\end{picture}
%%%%%%%%%%%%%%%%%%%%%%%%%%%%%%
%%%%%%%%%%%%%%%%%%
\qquad
%%%%%%%%%%%%%%%%%%%%%%%%%%%%%%%%%%
%%%%%%%%%  (b)  %%%%%%%%%%%%%%%%%%%%%%
%%%%%%%%%%%%%%%%%%%%%%%%%%%%%%%
\begin{picture}(50,55)(-45,-35)
\put(-18,0){\circle*{1.5}}
\put(-18,-20){\circle*{1.5}}
\put(-18.5,0){\line(0,-1){20}}
\put(-17.5,0){\line(0,-1){20}}
\put(-35,0){\line(1,0){40}}
\put(-11,-11){\vector (0,1){8}}
\put(-9,-14){\makebox{$\mathbf{b}$}}
\put(-17,5){\makebox{$-1$}}
\put(-14,-25){\makebox{$-1-\beta i$}}
\thicklines
\put(-14,0){\line (0,-1) {20}}
\put(-22,0){\line (0,-1) {20}}
\qbezier (-14,-20) (-14.5,-23.5) (-18, -24)
\qbezier (-18,-24) (-22.5,-23.5) (-22, -20)
\qbezier (-14,0) (-14.5,3.5) (-18, 4)
\qbezier (-18,4) (-22.5,3.5) (-22, 0)
\put(-36,-36){\makebox{(b) $\re A^{1/2}= 1$}}
\end{picture}
%%%%%%%%%%%%%%%%%%%%%%%%%%%%%%
\qquad
%%%%%%%%%%%%%%%%%%%%%%%%%%%%%%%%%%
%%%%%%%%%%%% (c)  %%%%%%%%%%%%%%%%%%%
%%%%%%%%%%%%%%%%%%%%%%%%%%%%%%%%
\begin{picture}(60,65)(-35,-35)
\put(20,0){\circle*{1.5}}
\put(-10,0){\circle*{1.5}}
\put(-10,20){\circle*{1.5}}
\put(-10,-20){\circle*{1.5}}
\put(-25,0){\line(1,0){15}}
\put(20,0){\line(1,0){13}}
\put(-10.5,-20){\line(0,1){40}}
\put(-9.5,-20){\line(0,1){40}}
\put(-10,0.5){\line(1,0){30}}
\put(-10,-0.5){\line(1,0){30}}
\put(10,7){\vector(-1,0){8}}
\put(13,7){\makebox{$\mathbf{b}$}}
%%%%%%  \put(-20,-6){\makebox{$0$}}
\put(-23,-5){\makebox{$-\alpha$}}
\put(20,-9){\makebox{$-1$}}
\put(-4,-24){\makebox{$-\alpha-\beta i$}}
\put(-4,22){\makebox{$-\alpha+\beta i$}}

\put(-20,-36){\makebox{(c) $\re A^{1/2} >1$}}
\thicklines
\put(-14,-20){\line(0,1){40}}
\put(-6,-20){\line(0,1){16}}
\put(-6,4){\line(0,1){16}}
\put(-6,-4){\line(1,0){26}}
\put(-6,4){\line(1,0){26}}
\qbezier (-14,-20) (-13.5, -23.5) (-10,-24) 
\qbezier (-6,-20) (-6.5,-23.5) (-10,-24) 
\qbezier (-14,20) (-13.5, 23.5) (-10,24) 
\qbezier (-6,20) (-6.5,23.5) (-10,24) 
\qbezier (20,-4) (23.5, -3.5) (24,0) 
\qbezier (20,4) (23.5, 3.5) (24,0) 
\end{picture}
%%%%%%%%%%%%%%%%%%%%%%%%%%%%%%
%%%%%%%%%%%%%%%%%%
\qquad
%%%%%%%%%%%%%%%%%%%%%%%%%%%%%%%%%
\end{center}
\caption{Cycle $\mathbf{b}$} 
\label{A}
\end{figure}
}
%%%%%%%%%%%%%%%%%%%%%%%%%%%%%%%%%%%%%%%%%%%%%%%%%%%%%%%%%%%
%%%%%%%%%%%%%%%%%%%%%%%%%%%%%%%%%%%%%%%%%
%%%%%%%%%%%%%%%%%%%%%%%%%%%%%%%%%%%%%%%%%%%%%%%%%%%%%%%%%%%
%%%%%%%%%%%%%%%%%%%%%%%%%%%%%%%%%%%%%%%%%%%%%%%%%%%%%%%%%
(a) {\it Case where $0\le \re A^{1/2} <1$}:  Write $A^{1/2}=\alpha +i\beta$
with $0\le \alpha <1,$ say, $\beta \ge 0.$ Then the cycle $\mathbf{b}$ may be 
deformed in such a way that $\mathbf{b}$ surrounds anticlockwise the cuts 
$[-\alpha-i\beta, -\alpha+ i\beta] \cup [-1,-\alpha]$, where $-\alpha -i\beta
=-A^{1/2},$ $-\alpha +i\beta = -\overline{A}^{1/2}$ (Figure \ref{A}, (a)). 
The function
$\sqrt{(A-z^2)(1-z^2)}$ (respectively, $\sqrt{(\overline{A}-z^2)(1-z^2)}$\,)
may be treated on the plane with the cuts 
$ [-1,-\alpha]\cup [-\alpha, -\alpha-i\beta]$
(respectively, $[-1,-\alpha] \cup [-\alpha,-\alpha+i\beta ]$). We have
$$
I_{\mathbf{b}}(A) -I_{\mathbf{b}}(\overline{A}) = \int_{\mathbf{b}}
\Bigl( \sqrt{\frac{A-z^2}{1-z^2}} -\sqrt{\frac{\overline{A}-z^2}{1-z^2}}\,\Bigr)
dz =(A-\overline{A})I_{\mathbf{b}}(A,\overline{A}) =0,
$$
where
$$
I_{\mathbf{b}}(A,\overline{A}) = \int_{\mathbf{b}} \frac{dz}{\sqrt{1-z^2}
(\sqrt{A-z^2}+\sqrt{\overline{A}-z^2}) }.
$$
To show $A\in \mathbb{R}$, suppose the contrary $A-\overline{A} \not=0.$
Dividing $\mathbf{b}$ into five parts, we have
$$
I_{\mathbf{b}}(A,\overline{A}) = I_0^{\beta} + \tilde{I}^0_{\beta} + H_{-\alpha}
^{-1} +\tilde{J}^{-\beta}_0 + J^0_{-\beta},
$$
in which
\begin{align*}
 I^{\beta}_0 &= \int^{\beta}_0 \frac{i dt}{\sqrt{1-(-\alpha+it)^2}
(\sqrt{A-(-\alpha+it)^2}+\sqrt{\overline{A}-(-\alpha+it)^2}) },
\\
 \tilde{I}^0_{\beta} &= \int^{0}_{\beta} \frac{i dt}{\sqrt{1-(-\alpha+it)^2}
(\sqrt{A-(-\alpha+it)^2}-\sqrt{\overline{A}-(-\alpha+it)^2}) },
\\
 H^{-1}_{-\alpha} &= \int^{-1}_{-\alpha} \frac{2 dt}{\sqrt{1-t^2}
(\sqrt{A-t^2}-\sqrt{\overline{A}-t^2}) },
\\
 \tilde{J}^{-\beta}_0 &= \int^{-\beta}_0 \frac{i dt}{-\sqrt{1-(-\alpha+it)^2}
(\sqrt{A-(-\alpha+it)^2}-\sqrt{\overline{A}-(-\alpha+it)^2}) },
\\
 J^0_{-\beta} &= \int^{0}_{-\beta} \frac{i dt}{-\sqrt{1-(-\alpha+it)^2}
(-\sqrt{A-(-\alpha+it)^2}-\sqrt{\overline{A}-(-\alpha+it)^2}) }.
\end{align*}
Then
$$
( I_0^{\beta} + \tilde{I}^0_{\beta} ) +(\tilde{J}^{-\beta}_0 + J^0_{-\beta})
= \frac{2i}{A-\overline{A}} \int^{\beta}_0 \biggl(\sqrt{\frac{A-(\alpha +it)^2}
{1-(\alpha+it)^2} } 
- \overline {\sqrt{\frac{A-(\alpha +it)^2}{1-(\alpha+it)^2} } } \biggr)  dt
 \in i \mathbb{R}
$$
(for the branch of $\sqrt{(A-z^2)/(1-z^2)}$ see Section \ref{sc5}).
The remaining integral $H^{-1}_{-\alpha}$ is
\begin{align*}
-\frac 12 H^{-1}_{-\alpha} &=
\frac{i}{A-\overline{A}} \int^1_{\alpha} \frac{\sqrt{t^2-(\alpha+i\beta)^2 } 
+ \sqrt{t^2-(\alpha-i\beta)^2 } }
{\sqrt{1-t^2 }}  dt
\\
&=\frac{2i}{A-\overline{A}} \int^1_{\alpha} \frac{\re \sqrt{t^2-\alpha^2
+\beta^2 -2i\alpha\beta } } {\sqrt{1-t^2} } dt 
\\
&=\frac{\sqrt{2}i}{A-\overline{A}} \int^1_{\alpha} \frac{\sqrt{t^2-\alpha^2
+\beta^2 +\sqrt{(t^2-\alpha^2+\beta^2)^2+4\alpha^2\beta^2}}}{\sqrt{1-t^2}} dt 
\in \mathbb{R} \setminus\{0\}.
\end{align*}
Hence $I_{\mathbf{b}}(A,\overline{A}) \not= 0,$ yielding the contradiction
$A-\overline{A} =0.$ In this case we have $A\in \mathbb{R}.$
\par 
(b) {\it Case where $ \re A^{1/2} = 1$}:  
Write $A^{1/2}=1+i \beta$, say $\beta \ge 0$ (cf. Figure \ref{A}, (b)). Then
%%%%%%%%%%%%%%%%%%%%%%%%%%%%%%%%%%%%%%
\begin{align*}
I_{\mathbf{b}}(A) &= \int_{\mathbf{b}} \sqrt{\frac{A-z^2}{1-z^2} } dz
=2i \int^0_{-\beta}\sqrt{ \frac{(-1-i \beta)^2-(-1+it)^2}{1-(-1+it)^2}}\,\, dt
\\
=&2i \int^{\beta}_0 \sqrt{ \frac{t^2-\beta^2 -2(t-\beta)i}{t^2-2ti} }\,\,dt
= -2 \int^{\beta}_0 \sqrt{\frac{\beta-t}{t(4+t^2)}  }
\sqrt{t^2+\beta t + 4 + 2\beta i  } \, dt
\end{align*}
%%%%%%%%%%%%%%%%%%%%%%%
with
$$
\re \sqrt{t^2+\beta t + 4 +2\beta i}=\sqrt{ t^2+\beta t +4 + 
\sqrt{(t^2+\beta t +4)^2 +4\beta^2 } } \ge 2\sqrt{2},
$$
which implies $\re I_{\mathbf{b}}(A) \not=0.$
\par
(c) {\it Case where $ \re A^{1/2} > 1$}:  
It is shown that $\re I_{\mathbf{b}}(A) =0$ implies $A\in \mathbb{R}$ by an
argument similar to that in the case (a).
\par
Thus in every case we have shown $A \in \mathbb{R}$ or $\re I_{\mathbf{b}}(A)
\not=0.$ We may examine $I_{\mathbf{a}}(A)$ and $I_{\mathbf{b}}(A)$ for each 
$A\in \mathbb{R}$ to conclude that $\re I_{\mathbf{a}}(A)=\re I_{\mathbf{b}}
(A)=0$ if and only if $A=0.$ This completes the proof. 
\end{proof}
%%%%%%%%%%%%%%%%%%%%%%%%%%%
%%%% Corollary A5 %%%%%
\begin{cor}\label{corA.5}
For every $A\in \mathbb{C}$, $(I_{\mathbf{a}}(A), I_{\mathbf{b}}(A))\not=
(0,0).$
\end{cor}
%%%%%%%%%%%%%%%%%%%%%%%%%%%%%
%%%% Corollary A6 %%%%%
\begin{cor}\label{corA.6}
If $\re I_{\mathbf{b}}(A)=0$, then $A=0.$
\end{cor}
%%%%%%%%%%%%%%%%%%%%%%%%%%%%
Since $I_{\mathbf{a}}(1)=4,$ $I_{\mathbf{b}}(1)=0,$ the number $A=1$ solves 
(BE)$_{\phi=\pm\pi/2}$. Observe that $\re iI_{\mathbf{b}}(A)=0$ implies
$I_{\mathbf{b}}(A)= -I_{\mathbf{b}}(\overline{A}).$ Then, similarly we have 
the following.
%%%%%%%%%%%%%%%%%%%%%%%%
%%%%%% Lemma A7 %%%%%
\begin{lem}\label{lemA.7}
If $A$ solves $(\mathrm{BE})_{\phi=\pm \pi/2}$, then $A=1.$
\end{lem}
%%%%%%%%%%%%%%%%%%%%%%%
%%%% Corollary A8 %%%%%
\begin{cor}\label{corA.8}
If $\re iI_{\mathbf{b}}(A)=0$, then $A=1.$
\end{cor}
%%%%%%%%%%%%%%%%%%%%%%%%
%%%%%% Lemma A9 %%%%%
\begin{lem}\label{lemA.9}
If $|\phi|$ is sufficiently small, equations $(\mathrm{BE})_{\phi}$ admit 
a solution
$A_{\phi}=x(\phi)+iy(\phi)$ such that
$$
x(\phi)= -\frac{4\phi^2}{\log \phi} (1+o(1)), \quad 
y(\phi)= -\frac{4\phi}{\log \phi} (1+o(1)), 
$$
which is unique around $A=0$.
\end{lem}
%%%%%%%%%%%%%%%%%%%%%%%
\begin{proof}
Suppose that $|A|$ is small and $\re A^{1/2} \ge 0.$ Then
\begin{align*}
I_{\mathbf{a}}(A) &= \int_{\mathbf{a}} \sqrt{\frac{A-z^2}{1-z^2} }dz
= 2\int_{-A^{1/2}}^{A^{1/2}} \sqrt{\frac{A-z^2}{1-z^2} }dz
= 2A\int_{-1}^{1} \sqrt{\frac{1-t^2}{1-At^2} }dt =\pi A +O(A^2),
\\
I_{\mathbf{b}}(A) &= \int_{\mathbf{b}} \sqrt{\frac{A-z^2}{1-z^2} }dz
= 2i \int_{A^{1/2}}^{1} \sqrt{\frac{z^2-A}{1-z^2} }dz
\\
&= 2i \biggl( \int_{A^{1/2}}^{1} \frac{zdz}{\sqrt{1-z^2}}
-A \int_{A^{1/2}}^{1} \frac{dz}{\sqrt{1-z^2}(z+\sqrt{z^2-A})} \biggr)
\\
&=\frac i2(4 +A\log A +O(A)).
\end{align*}
From $\re e^{i\phi}I_{\mathbf{a}}(A_{\phi})= \re e^{i\phi}I_{\mathbf{b}}
(A_{\phi})=0,$ that is, 
\begin{equation*}
\re((A_{\phi}+O(A_{\phi}^2))(\cos\phi+i\sin\phi) )
=\re(i(4+A_{\phi}\log A_{\phi}+O(A_{\phi}))(\cos\phi+i\sin\phi) ) 
=0
\end{equation*}
with $A_{\phi}=x(\phi)+iy(\phi),$ the conclusion follows.
\end{proof}
%%%%%%%%%%%%
Similarly we have the following.
%%%%%%%%%%%%%%%%%%%%
%%%% Lemma A.10 %%%%%
\begin{lem}\label{lemA.10}
If $|\phi \mp \pi/2|$ is sufficiently small, equations $(\mathrm{BE})_{\phi}$ 
admit a solution $A_{\phi}=x(\phi)+i y(\phi)$ such that
$$
x(\phi)= 1+ \frac{4\tilde{\phi}_{\pm}^2}{\log\tilde{\phi}_{\pm}}(1+o(1)),
\quad
y(\phi)= \frac{4\tilde{\phi}_{\pm}}{\log\tilde{\phi}_{\pm}}(1+o(1))
$$
with $\phi=\pm \pi/2 +\tilde{\phi}_{\pm}.$
\end{lem}
%%%%%%%%%%%%%%%%%%%%%%%%%%%
%%%% Lemma A.11 %%%%%%
\begin{lem}\label{lemA.11}
Suppose that $0<|\phi_0|<\pi/2$ and that $A(\phi_0)$ solves 
$(\mathrm{BE})_{\phi=\phi_0}$. 
Then there exists a curve $\Gamma(\phi_0)$ given by $A={A}(\phi_0,\phi)$ 
for $|\phi|\le \pi/2$, where ${A}(\phi_0,\phi)$ has the properties $:$
\par
$(\mathrm{i})$ ${A}(\phi_0,\phi_0) = A(\phi_0)$, $A(\phi_0,0)=0,$
$A(\phi_0, \pm \pi/2)=1;$ 
\par
$(\mathrm{ii})$ ${A}(\phi_0,\phi)$ is continuous in $\phi$ for $|\phi|\le
\pi/2$ and smooth for $0<|\phi|<\pi/2;$
\par
$(\mathrm{iii})$ ${A}(\phi_0,\phi)$ solves $(\mathrm{BE})_{\phi}$ for 
$|\phi|\le \pi/2$.
\end{lem}
%%%%%%%%%%%%%%%%%%%%%%%%
\begin{proof}
Set 
$$
A=x+iy, \quad I_{\mathbf{a}}(A)=u(A)+iv(A), \quad I_{\mathbf{b}}(A)=U(A)+iV(A) 
$$
with $x,$ $y,$ $u(A),$ $v(A),$ $U(A),$ $V(A) \in \mathbb{R}.$ Then $A$ solves
(BE)$_{\phi}$ if and only if 
$$
\re e^{i\phi}I_{\mathbf{a}}(A) =u(A)\cos\phi -v(A)\sin\phi=
\re e^{i\phi}I_{\mathbf{b}}(A) =U(A)\cos\phi -V(A)\sin\phi=0,
$$
that is, 
%%%%%%% (A.1) %%%%%%%
\begin{equation}\label{A.1}
u(A)-v(A) t=  U(A)-V(A)t =0  \quad \text{with  $t=\tan \phi.$}
\end{equation}
%%%%%%%%%%%%%%%%%%%%%%%%
\par
By the Cauchy-Riemann relations the Jacobian for \eqref{A.1} is
%%%%%% (A.2) %%%%%%
\begin{align}\label{A.2}
J(v, & V,t; A)=\det \begin{pmatrix} v_y -tv_x & -v_x-tv_y \\
                               V_y -tV_x & -V_x-tV_y \\
\end{pmatrix}=(1+t^2)(v_xV_y -v_yV_x)
\\
\notag
&=-\frac i8 (1+t^2)(\omega_{\mathbf{a}}(A)\overline{\omega_{\mathbf{b}}(A)} -
\overline{\omega_{\mathbf{a}}(A)} \omega_{\mathbf{b}}(A)) 
=-\frac 14 (1+t^2) \im
(\overline{\omega_{\mathbf{a}}(A)} \omega_{\mathbf{b}}(A)) 
\\
\notag
&=-\frac 14 (1+t^2) |\omega_{\mathbf{a}}(A)|^2 \im \frac 
{\omega_{\mathbf{b}}(A)} {\omega_{\mathbf{a}}(A)}  <0, \,\, \not=\infty, 
\end{align}
provided that $A\not=0,1.$ By supposition, since $A(\phi_0)\not=0,1$, 
there exists a function $A_0(\phi)$ with the properties:
\par
(a) $A_0(\phi_0)=A(\phi_0);$
\par
(b) $A_0(\phi)$ is smooth for $|\phi-\phi_0|<
\varepsilon_*$, $\varepsilon_*$ being sufficiently small;
\par
(c) $A_0(\phi)$ solves \eqref{A.1}, i.e. (BE)$_{\phi}$ for $|\phi-\phi_0|<
\varepsilon_*$ and is a unique solution in a small neighbourhood of $A(\phi_0)$.
\par
Let us consider the case $0<\phi_0<\pi/2$. Denote by $\mathcal{F}(\phi_0)$ 
the family of functions $\hat{A}_{\nu}(\phi)$ with the properties:
\par
(a$_{\nu}$) $\hat{A}_{\nu}(\phi_0)=A(\phi_0);$
\par
(b$_{\nu}$) $\hat{A}_{\nu}(\phi)$ is smooth for $\phi_0
-\varepsilon_* <\phi <\phi_{\nu} <\pi/2;$
\par
(c$_{\nu}$) $\hat{A}_{\nu}(\phi)$ solves \eqref{A.1} for $\phi_0-\varepsilon_*
<\phi <\phi_{\nu}.$
\par
Then $A_0(\phi) \in \mathcal{F}(\phi_0)$. For any $\hat{A}_{\nu}(\phi),$
$\hat{A}_{\nu'}(\phi) \in \mathcal{F}(\phi_0)$ with $\phi_{\nu} > \phi_{\nu'}$, 
$\hat{A}_{\nu}(\phi)\equiv \hat{A}_{\nu'}(\phi)$ holds if $\phi_0-\varepsilon_* 
<\phi <\phi_{\nu'}.$ 
Let $\hat{A}_{\infty}(\phi)$ be the maximal extension of all $\hat{A}_{\nu}(\phi)
\in \mathcal{F}(\phi_0)$, and set $\phi_{\infty}=\sup_{\nu} \phi_{\nu}$. 
Then $\hat{A}_{\infty}(\phi)$ solves \eqref{A.1} and is smooth 
for $\phi_0-\varepsilon_* <\phi<\phi_{\infty}.$
Suppose that $\phi_{\infty}<\pi/2.$ Since $\hat{A}_{\infty}(\phi_{\infty})\not
=0,1$, the Jacobian $J(v,V,\tan \phi_{\infty}; \hat{A}_{\infty}(\phi_{\infty}))$
does not vanish, which implies that $\hat{A}_{\infty}(\phi)$ may be extended 
beyond
$\phi_{\infty}.$ This is a contradiction. Hence we have $\phi_{\infty}=\pi/2,$
and by Lemma \ref{lemA.7}, $\hat{A}_{\infty}(\pi/2)=1.$ 
Similarly we may construct a lower extension $\hat{A}_{\infty}^*(\phi)$ for
$0\le \phi < \phi_0+\varepsilon_*$ satisfying $\hat{A}_{\infty}^*(0)=0,$
and then have the extension $A^+(\phi_0,\phi)$ for $0\le \phi \le \pi/2.$ 
The case $-\pi/2 <\phi_0< 0$ may be treated in the same way to obtain 
$A^-(\phi_0,\phi)$ for $-\pi/2 \le \phi \le 0.$   
For a given $\phi_0$ satisfying, say, $0<\phi_0<\pi/2,$ combining $A^+(\phi_0,
\phi)$ with $A^-(\phi_-,\phi)$, where $\phi_-<0$ close to
$0$ is such that $A_{\phi_-}$ is a solution given in Lemma \ref{lemA.9}, 
we obtain the desired extension $A(\phi_0, \phi)$ for $|\phi|\le \pi/2.$ 
Thus the lemma is proved.
\end{proof}
%%%%%%%%%%%%%%%%%%%%%%%%%%%%%
%%%%% Corollary A.12 %%%%%%%%%%%
\begin{cor}\label{corA.12}
Under the same supposition as in Lemma $\ref{lemA.11}$, $(d/d\phi)A(\phi_0,\phi)
\not=0$ for $0<|\phi|<\pi/2.$ Furthermore, $(d/d\phi)\mathcal{I}(A(\phi_0,\phi))
>0$ or $<0$ for $0<\phi<\pi/2,$ and so for $-\pi/2<\phi < 0.$
\end{cor}
%%%%%%%%%%%%%%%%%%%%
\begin{proof}
From \eqref{A.1} it follows that 
$$
J(v,V,\tan\phi; A(\phi_0,\phi)) \begin{pmatrix}  x'(t) \\ y'(t) 
\end{pmatrix} - \begin{pmatrix} v(A(\phi_0,\phi)) \\ V(A(\phi_0,\phi))
\end{pmatrix} \equiv \mathbf{o},
$$
where $A(\phi_0,\phi)=x(t)+iy(t).$
By Corollary \ref{corA.5} and \eqref{A.2}, $(x'(t),y'(t))\not= (0,0)$
if $0<|\phi|<\pi/2,$ i.e. $t \in \mathbb{R}\setminus \{0\},$ and then
$ (d/d\phi)A(\phi_0,\phi)= (x'(t)+iy'(t))/\cos^2\phi \not=0.$ 
Since $\mathcal{I}(A(\phi_0,\phi)) \in \mathbb{R}$ by Lemma \ref{lemA.1},
we have 
$$
\frac d{d\phi}\mathcal{I}(A(\phi_0,\phi)) =\frac d{d\phi}A(\phi_0,\phi)
\frac{-2\pi i}{I_{\mathbf{b}}(A(\phi_0,\phi))^2} \in 
\mathbb{R} \setminus \{0\}
$$
for $0<|\phi|<\pi/2,$ from which the conclusion follows.
\end{proof}
%%%%%%%%%%%%%%%%%%%%%%%%%%%%%%%%%%%
%%%% Proposition A.13 %%%%%
\begin{prop}\label{propA.13}
For each $\phi_*$ such that $|\phi_*|\le \pi/2,$ equations 
$(\mathrm{BE})_{\phi=\phi_*}$
admit a unique solution $A_{\phi_*} \in \mathbb{C}.$
\end{prop}
%%%%%%%%%%%%%%%%%%%%%%%
\begin{proof}
Let $\hat{\phi}_0$ be so close to $0$ 
that $A_{\hat{\phi}_0}$ is a solution given in Lemma \ref{lemA.9}. 
Lemma \ref{lemA.11} with $\phi_0=\hat{\phi}_0$ provides a curve 
$\Gamma(\hat{\phi}_0)$ containing a solution of (BE)$_{\phi=\phi_*}$ 
for each fixed $\phi_*$. It remains to show
the uniqueness of a solution for $\phi_*\not=0, \pm\pi/2.$ Suppose that 
$A_{\phi_*}$ and $A'_{\phi_*}$ solve (BE)$_{\phi=\phi_*}$. Then, by Lemmas
\ref{lemA.4} and
\ref{lemA.11}, there exist curves $\Gamma(\phi_*)$ and $\Gamma'(\phi_*)$ such
that $\Gamma(\phi_*) \ni 0, A_*,$ $\Gamma'(\phi_*)\ni 0, A'_*.$ Then, by 
\eqref{A.2} (or the conformality of Lemma \ref{lemA.2}),
we have $\Gamma(\phi_*)=\Gamma'(\phi_*)\ni A_{\phi_*}=A'_{\phi_*},$ which
completes the proof.
\end{proof}
%%%%%%%%%%%%%%%%%%%%%%%%%%%%%
By the uniqueness above we easily have the following.
%%%%%%%%%%%%%%%%%%%%%%%%%%%%%%%%%%%%%
%%%%%%% Lemma A.14 %%%%%
\begin{lem}\label{lemA.14}
For $\phi \in\mathbb{R}$, $(\mathrm{BE})_{\phi}$ admit a unique solution
$A_{\phi}$, which satisfies 
$$
A_{\phi\pm \pi}=A_{\phi}, \quad A_{-\phi}=\overline{A_{\phi}}.
$$
\end{lem}
%%%%%%%%%%%%%%%%%%%%%%%%
%%%%%%%%%%%%%%%%%%%%%%%%%%%%%%%%%%%%%%%
%%% Lemma A.16 %%%%%%%
\begin{lem}\label{lemA.16}
Each $A_{\phi}$ given in Lemma $\ref{lemA.14}$ satisfies $0 \le \re A_{\phi} 
\le 1.$
For $0<\phi <\pi/2$ $($respectively, $-\pi/2 <\phi<0)$, $(d/d\phi) \re A_{\phi}
>0$ $($respectively, $<0)$.
\end{lem}
%%%%%%%%%%%%%%%%%%%%%%%%%%%%%%%%%%%%%%%
\begin{proof}
Let $A_{\phi}=x(t)+iy(t),$ $t=\tan\phi.$ Then, by Corollary \ref{corA.12},
$$
(d/dt)\mathcal{I}(A_{\phi})=(x'(t)+iy'(t))(-2\pi i)I_{\mathbf{b}}(A_{\phi})^{-2}
\in \mathbb{R} \setminus \{0\}
$$
for $0<|\phi|<\pi/2.$ This yields $x'(t)(U_*^2-V_*^2)-2y'(t) U_*V_*=0,$ where
$I_{\mathbf{b}}(A_{\phi})^{-1}=U_*+iV_*.$ Suppose that, $x'(t_0)=0$ and 
$0< \re A_{\phi_0} <1,$ for some $t_0=\tan \phi_0 \not=0, \pm\infty.$
Since $y'(t_0)\not=0,$ $U_*V_*=0.$ If $U_*=0,$ then $\re I_{\mathbf{b}}( A
_{\phi_0})=0,$ and hence $A_{\phi_0}=0,$ i.e. $\phi_0=0$ by Corollary 
\ref{corA.6}.
If $V_*=0$, then $\re iI_{\mathbf{b}}(A_{\phi_0})=0,$ and hence $A_{\phi_0}=1,$ 
i.e.
$\phi_0=\pm \pi/2$ by Corollary \ref{corA.8}. Thus we have shown that $x'(t)
>0$ or $x'(t)<0$ for $0<|\phi|<\pi/2,$ $t=\tan\phi,$ which implies 
$0 \le \re A_{\phi} \le 1.$
\end{proof}
%%%%%%%%%%%%%%%%%%%%%%%%%
%%%%% Remark A.1 %%%
\begin{rem}\label{remA.1}
In the proof above, it is easy to see that $y'(t)=0$ occurs if and only if
$U_*=\pm V_*$, that is, $\phi=\pm \pi/4,$ $\pm 3\pi/4.$
\end{rem}
%%%%%%%%%%%%%%%%%%%%%%%%%%%%%%%%
Lemmas \ref{lemA.9}, \ref{lemA.11}, \ref{lemA.16}, Proposition 
\ref{propA.13} and Remark \ref{remA.1} leads to the following.
%%%%%%%%%%%%%%%%%%%%%%%%%%%%%%
%%%%% Proposition A.15 %%%%%%
\begin{prop}\label{propA.15}
There exists a Jordan closed curve $\Gamma_0=\{A_{\phi}\,;\,\,|\phi|\le \pi/2\}$
with the properties$:$
\par
$(\mathrm{i})$ $A_0=0,$ $A_{\pm\pi/2}=1;$
\par
$(\mathrm{ii})$ $A_{\phi}$ is smooth for $0<|\phi|<\pi/2;$ 
\par
$(\mathrm{iii})$ for every $\phi, |\phi|\le \pi/2,$ $A_{\phi}$ solves $(\mathrm
{BE})_{\phi}.$ 
\end{prop}
%%%%%%%%%%%%%%%%%%%%%%%%
\par
By the properties above the trajectory of $A_{\phi}$
for $|\phi|\le \pi/2$ is as in Figure \ref{trajectories}, (a).
%%%%%%%%%%%%%%%%%%%%%%%%%%%%%%%%%%%%%%%%%%%%%%%%%%%%%%%%%%%%
%%%%%%%%%%%%%%%%%%%%%%%%%%%%%%%%%%%%%%%%%%%%%%%%%%%%%%%%
%%%%%%%%%%%%%%%% Figure 7.2 %%%%%%%%%%%%%%%%%%%%%%%%%%
%%%%%%%%%%%%%%%%%%%%%%%%%%%%%%%%%%%%%%%%%%%%%%%%
%%%%%%%%%%%%%%%%%%%%%%%%%%%%%%%%%%%%%%%%%%
%%%%%%%%%%%%%%%%%%%%%%%%%%%%
{\small
\begin{figure}[htb]
\begin{center}
\unitlength=0.75mm
%%%%%%%%%%%%%%%%%%%%%%%%%%%%%%%%%%
%%%%%%%%%%%%%%%%%%%%%%%%%%%%%%%
\begin{picture}(70,65)(-35,-35)
\put(20,0){\circle*{1.5}}
\put(-20,0){\circle*{1.5}}
\put(-30,0){\line(1,0){60}}

\qbezier (-15,15) (-10,17.8) (-4,18)
\put(-4,18){\vector (1,0){0}}
\qbezier (15,-15) (10,-17.8) (4,-18)
\put(4,-18){\vector (-1,0){0}}

\put(-25,-6){\makebox{$0$}}
\put(22,-6){\makebox{$1$}}
\put(-34,3){\makebox{$\phi=0$}}
\put(22,3){\makebox{$\phi=\pm\pi/2$}}
\put(0,20){\makebox{$0<\phi<\pi/2$}}
\put(-32,-22){\makebox{$-\pi/2<\phi<0$}}

\put(-10,-36){\makebox{(a) $A_{\phi}$}}
\thicklines
\qbezier (-20,0) (-20, 15) (0,15)
\qbezier (20,0) (20, 15) (0,15)
\qbezier (-20,0) (-20, -15) (0,-15)
\qbezier (20,0) (20, -15) (0,-15)
\end{picture}
%%%%%%%%%%%%%%%%%%%%%%%%%%%%%%
%%%%%%%%%%%%%%%%%%
%%%%%%%%%%%%%%%%%%
\qquad\qquad\qquad
%%%%%%%%%%%%%%%%%%%%%%%%%%%%%%%
\begin{picture}(70,65)(-35,-35)
\put(20,0){\circle*{1.5}}
\put(-20,0){\circle*{1.5}}
\put(-30,0){\line(1,0){60}}

\put(0,16.5){\makebox{$0<\phi<\pi/2$}}
\put(-32,-19){\makebox{$-\pi/2<\phi<0$}}
\put(-25,-6){\makebox{$0$}}
\put(22,-6){\makebox{$1$}}
\put(-34,3){\makebox{$\phi=0$}}
\put(22,3){\makebox{$\phi=\pm\pi/2$}}
\put(-10,-36){\makebox{(b) $A^{1/2}_{\phi}$}}

\qbezier (15,-12.6) (10,-15.4) (4,-13.6)
\put(4,-13.6){\vector (-4,1){0}}
\qbezier (-15,7.2) (-10,10.2) (-4,12.2)
\put(-4,12.2){\vector (3,1){0}}
\thicklines
\qbezier  (10,11) (20, 9)(20,0)
\qbezier  (-5,8.5) (5, 12)(10,11)
\qbezier  (-20,0)(-17,3) (-5, 8.5)
\qbezier  (10,-11) (20, -9)(20,0)
\qbezier  (-5,-8.5) (5, -12)(10,-11)
\qbezier  (-20,0)(-17,-3) (-5, -8.5)
\end{picture}
%%%%%%%%%%%%%%%%%%%%%%%%%%%%%%
%%%%%%%%%%%%%%%%%%%%%%%%%%%%%%%%%
\end{center}
\caption{Rough drawings of the trajectories $A_{\phi}$ and $A^{1/2}_{\phi}$} 
\label{trajectories}
\end{figure}
}
%%%%%%%%%%%%%%%%%%%%%%%%%%%%%%%%%%%%%%%%%%%
%%%%%%%%%%%%%%%%%%%%%%%%%%%%%%%%%%%%%%%%%%%
%%%%%%%%%%%%%%%%%%%%%%%%%%%%%%%%%%%%%%%%%%%%
\par
The properties of $A_{\phi}$ are summarised as follows.
%%%%%%%%%%%%%%%%%%%%%%%%
%%%%% Proposition A.17 %%%%%%
\begin{prop}\label{propA.17}
$(1)$ For $|\phi|\le \pi/2,$ the Boutroux equations $(\mathrm{BE})_{\phi}$
have a unique solution $A_{\phi}$ with the properties$:$
\par
$(\mathrm{i})$ $A_0=0,$ $A_{\pm \pi/2}=1;$
\par
$(\mathrm{ii})$ $A_{\phi}$ is smooth in $\phi$ such that
$0<|\phi|<\pi/2;$  
\par
$(\mathrm{iii})$ for $0<\phi <\pi/2$ $($respectively, $-\pi/2<\phi <0)$,
$x(t)=\re A_{\phi},$ $t=\tan\phi$ satisfies $x'(t)>0$ $($respectively,
$x'(t)<0 )$, and $y(t)=\im A_{\phi}$ satisfies
$y'(t)=0$ if and only if $\phi=\pm \pi/4,$ $\pm 3\pi/4;$
\par
$(\mathrm{iv})$ $0<\re A_{\phi} <1$ for $0<|\phi|<\pi/2,$ and
$\im A_{\phi} >0$ for $0<\phi<\pi/2$ and $<0$ for $-\pi/2<\phi<0.$
\par
$(2)$ For $\phi\in \mathbb{R},$ $A_{\phi}$ may be extended by using the 
relations $A_{-\phi}=\overline{A_{\phi}},$ $A_{\phi\pm \pi}=A_{\phi}.$
\end{prop}
%%%%%%%%%%%%%%%%%%%%%%%%%%%%%%%%%%
%%%%% Remark A.2 %%%%%%%%%%%%%%%%
%%%%%%%%%%%%%%%%%%%%%%%%%%%%%%%%%%
\begin{rem}\label{remA.2}
By Lemma \ref{lemA.9}, around $\phi=0,$ $A_{\phi}=x(\phi)+iy(\phi)$ is 
expressed as $A_{\phi}=\pm i 2^{3/2}(xL(x)^{-1})^{1/2}(1+o(1))$
and $A_{\phi}^{1/2}=\pm e^{\pi i/4}2^{3/4}(xL(x)^{-1})^{1/4}(1+o(1))$ 
as $x=x(\phi)=4\phi^2L(\phi)^{-1}(1+o(1))\to 0$, where $L(x)=|\log x|.$
By Lemma \ref{lemA.10}, around $\phi=\pm \pi/2$, $A_{\phi}=1-\frac 14 y^2L(y)
(1+o(1))+iy$ and $A_{\phi}^{1/2}=1-\frac 18 y^2L(y)(1+o(1))+iy/2$ 
as $y=y({\phi})=-4\tilde{\phi}L(\tilde{\phi})^{-1}(1+o(1))\to 0$ 
$(\tilde{\phi}=\phi\mp\pi/2)$. Taking these local shapes around $\phi=0,$ 
and $\pm \pi/2$ into account, which are important in finding the Stokes graph,
we have a rough drawing of the trajectory $A_{\phi}^{1/2}$ as in Figure 
\ref{trajectories}, (b).
%% It is easily verified that $0<\re A^{1/2}_{\phi} <1$ for $0<|\phi|<\pi/2,$
%% if the trajectory $\Gamma_0 =\{A_{\phi}; \, |\phi|\le \pi/2 \}$ is contained
%% in $P_0 =\{x+ iy ;\, y^2 \le 4(1-x) \}.$ For some $\varepsilon_1 \ge
%% \varepsilon_0 >0,$ by Lemma \ref{lemA.10}, the set $P_0$ contains the arc
%% $\{A_{\phi};\, |\phi-\pi/2|<\varepsilon_0 \}$, and $\{A_{\phi};\, |\phi|
%% <\pi/2-\varepsilon_1 \}$ as well. It is likely that the trajectory of
%% $A^{1/2}_{\phi}$ is as in Figure \ref{trajectories}, (b). It remains, however,
%% to exclude the possibility of the existence of $A_{\phi}\in \Gamma_0$ 
%% such that $\re A_{\phi}^{1/2} >1$ and 
%% $|\im A^{1/2}_{\phi}|>\varepsilon_2 >0.$ To do so
%% it is necessary to show that, say, $\im \mathcal{I}(A) \not=0$
%% on $\{A; \re A^{1/2}=1 \}\setminus\{A=1 \}=\partial P_0\setminus\{(1,0)\}.$
\end{rem}
%%%%%%%%%%%%%%%%%%%%%%%%%
%%%%%%%%%%%%%%%%%%%%%%%%%%%%%%%%%%%%%%%%%%%%
%%%%%%%%%%%%%%%%%%%%%%%%%%%%%%%%%
%%%%% References %%%%%%%%%%%%%%%%%%%%%%%%%
%%%%%%%%%%%%%%%%%%%%%%%%%

\end{document}